\documentclass[12pt]{amsart}
\DeclareRobustCommand{\gobblefive}[5]{}
\newcommand*{\SkipTocEntry}{\addtocontents{toc}{\gobblefive}}
\usepackage[margin=3.3cm]{geometry}
\usepackage{amscd,amsmath,amsfonts,amssymb,latexsym}
\usepackage{hyperref,graphicx}
\usepackage{xcolor}
\hypersetup{
    colorlinks=true,
    citecolor=blue,
    linkcolor=blue,
    urlcolor=blue,
}
\usepackage{tikz}
\usepackage[all]{xy}
\usepackage{slashed}
\usepackage{soul}
\usepackage{comment}
\usepackage{extarrows}
\usepackage{mathtools}
\usepackage{url}
\usepackage{lipsum}

\numberwithin{equation}{section}
\theoremstyle{plain}
\newtheorem{lemma}{Lemma}[section]
\newtheorem{proposition}[lemma]{Proposition}

\newtheorem{theorem}[lemma]{Theorem}
\newtheorem{corollary}[lemma]{Corollary}

\theoremstyle{definition}
\newtheorem{definition}[lemma]{Definition}
\newtheorem{remark}[lemma]{Remark}

\let\C\relax
\newcommand{\C}{{\mathbb C}}
\newcommand{\R}{{\mathbb R}}
\newcommand{\Z}{{\mathbb Z}} 

\newcommand{\Hh}{{\mathcal H}}
\newcommand{\Ll}{{\mathcal L}}
\newcommand{\Mm}{{\mathcal M}}

\newcommand{\id}{{\rm id}}
\newcommand{\Om}{{\Omega}}
\newcommand{\om}{{\omega}}

\newcommand{\la}{\langle}
\newcommand{\ra}{\rangle}
\newcommand{\p}{{\partial}}
\newcommand{\bp}{{\bar{\partial}}}
\newcommand{\vol}{\mbox{\rm vol}}
\newcommand{\rmspan}{\mbox{\rm span}}

\newcommand{\Sp}{{\text{\rm Spin}(7)}}

\renewcommand{\Im}{{ \rm Im \,}}

\newcommand{\Aa}{{\mathcal A}}

\newcommand{\Ff}{{\mathcal F}}
 
\newcommand{\Gg}{{\mathcal G}}
\newcommand{\Kk}{{\mathcal K}}

\newcommand{\eps}{{\varepsilon}}

\newcommand{\lb}{{[\![}}
\newcommand{\rb}{{]\!]}}

\newcommand{\FF}{{E}}

\renewcommand{\i}{{\sqrt{-1}}}
\renewcommand{\l}{{\ell}}

\newcommand{\n}{{\nabla}}

\begin{document}


\title
{Mirror of volume functionals on manifolds with special holonomy}

\author{Kotaro Kawai}
\address{Department of Mathematics, Faculty of Science, Gakushuin University, 1-5-1 Mejiro, Toshima-ku, Tokyo 171-8588, Japan}
\email{kkawai@math.gakushuin.ac.jp}

\author{Hikaru Yamamoto}
\address{Department of Mathematics, Faculty of Pure and Applied Science, University of Tsukuba, 1-1-1 Tennodai, Tsukuba, Ibaraki 305-8577, Japan}
\email{hyamamoto@math.tsukuba.ac.jp}

\thanks{The first named author is supported by 
JSPS KAKENHI Grant Number JP17K14181, 
and the second named author is supported by JSPS KAKENHI Grant Number 
JP18K13415 and Osaka City University Advanced Mathematical Institute (MEXT Joint Usage/Research Center on Mathematics and Theoretical Physics)}
\begin{abstract}
We can define the ``volume'' $V$ for Hermitian connections 
on a Hermitian complex line bundle over a Riemannian manifold $X$, which can be considered to be 
the ``mirror'' of the standard volume for submanifolds. 
This is called the Dirac-Born-Infeld (DBI) action in physics. 

In this paper, 
(1) we introduce the negative gradient flow of $V$, which we call the line bundle mean curvature flow. 
Then, we show the short-time existence and uniqueness of this flow. 
When $X$ is K\"ahler, we relate the negative gradient of $V$ to the angle function and 
deduce the mean curvature for Hermitian metrics on a holomorphic line bundle defined by Jacob and Yau.

(2) We relate the functional $V$ to 
a deformed Hermitian Yang--Mills (dHYM) connection,  
a deformed Donaldson--Thomas connection for a $G_2$-manifold (a $G_2$-dDT connection), 
a deformed Donaldson--Thomas connection for a ${\rm Spin}(7)$-manifold (a ${\rm Spin}(7)$-dDT connection), 
which are considered to be the ``mirror'' of 
special Lagrangian, (co)associative and Cayley submanifolds, respectively. 
When $X$ is a compact ${\rm Spin}(7)$-manifold, 
we prove the ``mirror'' of the Cayley equality, which implies the following. 
(a) Any ${\rm Spin}(7)$-dDT connection is a global minimizer of $V$ and its value is topological. 
(b) Any ${\rm Spin}(7)$-dDT connection is flat on a flat line bundle. 
(c) If $X$ is a product of $S^1$ and a compact $G_2$-manifold $Y$, 
any ${\rm Spin}(7)$-dDT connection on the pullback of the Hermitian complex line bundle over $Y$ 
is the pullback of a $G_2$-dDT connection modulo closed 1-forms. 

We also prove analogous statements for $G_2$-manifolds 
and K\"ahler manifolds of dimension 3 or 4. 
\end{abstract}


\maketitle

\tableofcontents
\section{Introduction}
Let $X$ be a Calabi--Yau, $G_2$- or $\Sp$-manifold. 
If $X$ is the total space of a torus fiber bundle, 
we can send geometric objects on $X$ to those on the ``mirror'' manifold via the real Fourier--Mukai transform. 
For example, 
special Lagrangian, (co)associative and Cayley submanifolds 
(more generally, cycles) 
are sent to 
deformed Hermitian Yang--Mills (dHYM) connections,  
deformed Donaldson--Thomas connections for a $G_2$-manifold ($G_2$-dDT connections), 
deformed Donaldson--Thomas connections for a $\Sp$-manifold ($\Sp$-dDT connections), respectively. 
These are defined without a torus fiber bundle structure  
and they are 
Hermitian connections on a Hermitian line bundle $L$ over $X$ defined by fully nonlinear PDEs. 
As the names indicate, 
dHYM and $G_2, \Sp$-dDT connections can also be considered as 
analogues of Hermitian Yang--Mills (HYM) connections 
and Donaldson--Thomas connections ($G_2, \Sp$-instantons), respectively. 
$$  
\xymatrix@!C{
\left\{ \mbox{ \begin{tabular}{c} calibrated \\ submanifolds \end{tabular}} \right \} 
& \left\{ \mbox{ \begin{tabular}{c} dHYM, \\ $G_2,\Sp$-dDT \\ connections \end{tabular}} \right \} 
\ar@{<->}[l]_-{mirror} 
\ar@{<->}[r]^-{analogue} 
 & \left\{ \mbox{ \begin{tabular}{c} HYM, \\ $G_2, \Sp$-\\ instantons \end{tabular}} \right \}
  }
$$

Readers may wonder why the discussion is limited to line bundles. 
The reason for the dHYM case is explained in \cite[Section 8.1]{CY}. 
For holomorphic bundles with non-abelian gauge group, 
the analogue of the dHYM equation is not known. 
Though there is a natural guess for the higher rank dHYM equation, the resulting equations are 
fully nonlinear systems, and the analytic difficulties in addressing them are formidable. 
We focus on line bundles in the dDT case for the same reason.

In this paper, we study the ``volume'' for Hermitian connections, 
which can be considered to be a ``mirror'' of the standard volume for submanifolds 
via the real Fourier--Mukai transform. 
This is called the Dirac-Born-Infeld (DBI) action in physics. 
Then, we show that dHYM, $G_2,\Sp$-dDT connections indeed have 
similar properties to calibrated submanifolds and HYM, $G_2, \Sp$-instantons 
in addition to the similarities of moduli spaces given in \cite{KY, KYSpin7}. 
More explicitly, we mainly study the following two topics.

\subsection{The ``mirror'' of the mean curvature flow}
Let $(X,g)$ be a compact oriented $n$-dimensional Riemannian manifold 
and $L \to X$ be a smooth complex line bundle with a Hermitian metric $h$.
Let $\Aa_0$ be the space of Hermitian connections on $(L,h)$. 
We regard the curvature 2-form $F_\n$ of $\n \in \Aa_0$ as a $\i \R$-valued closed 2-form on $X$. 
Define the \emph{volume functional} $V: \Aa_0 \rightarrow \R$ by 
\[
V(\n) = \int_X \sqrt{\det (\id_{TX} + (-\i F_\n)^\sharp)} \ \vol_g, 
\]
where 
$(-\sqrt{-1}F_\n)^\sharp \in \Gamma (X, \mathop{\mathrm{End}} TX)$ is defined by $u\mapsto \left(-\sqrt{-1}i(u)F_{\nabla}\right)^{\sharp}$ and $\vol_g$ is the volume form defined by the Riemannian metric $g$. 
Note that $\det (\id_{TX} + (-\i F_\n)^\sharp) \geq 1$ 
since $(-\i F_\n)^\sharp$ is skew-symmetric. See Lemma \ref{lem:det} for details. 

As the negative gradient of $V$, we define the mean curvature $H(\n) \in \Om^1$ for $\n \in \Aa_0$, 
which is described explicitly in \eqref{eq:MC}. 
Then, following \cite{JY}, we say that 
a smooth family $\{\, \n_{t}\,\}_{t\in[0,T)} \subset \Aa_0$, where $T\in(0,\infty]$, 
satisfies the \emph{line bundle mean curvature flow} if 
\begin{equation}\label{maineq:lbmcf}
	\frac{\p}{\p t}\left(\frac{ \n_t}{\i}\right) = H(\n_t). 
\end{equation}

We first show the short-time existence and uniqueness of \eqref{maineq:lbmcf}.

\begin{theorem}[Theorem \ref{short-ex}] \label{mainthm:MCF}
\begin{enumerate}
\item
For any $\n_0 \in \Aa_0$, there exist $\eps >0$ and a smooth family of Hermitian connections 
$\{ \n_t \}_{t \in [0,\eps]}$ satisfying \eqref{maineq:lbmcf} and $\n_t|_{t=0}=\n_0$. 

\item
Suppose that 
$\{ \n^1_t \}_{t \in [0,\eps]}$ and $\{ \n^2_t \}_{t \in [0,\eps]}$ satisfy \eqref{maineq:lbmcf} for $\eps >0$. 
If $\n^1_t|_{t=0}=\n^2_t|_{t=0}$, then $\n^2_t-\n^1_t$ is 
a (pure imaginary-valued) time-dependent exact 1-form for any $t \in [0,\eps]$. 
\end{enumerate}
\end{theorem}
Since $L$ is a line bundle, the curvature is invariant under the addition of closed 1-forms, 
that is, $F_{\n+\i a} = F_\n$ for any $\n \in \Aa_0$ and $a \in Z^1$, where $Z^1$ is the space of closed 1-forms. 
This implies that $V$ is invariant under the addition of closed 1-forms, and hence, 
$H$ degenerates in the direction of $\i Z^1$. 
In particular, \eqref{maineq:lbmcf} is not strongly parabolic.

However, giving a ``nice'' expression of $H(\n)$ in Proposition \ref{prop:fistvar}, 
we show that the principal symbol of $H$ degenerates only in this direction, 
just as in the case of the mean curvature for submanifolds. 
This motivates us to use DeTurck's trick to 
prove the short-time existence and uniqueness of \eqref{maineq:lbmcf}. 
By the similarity between mean curvature flows and line bundle mean curvature flows, 
we can introduce a modified parabolic flow to obtain Theorem \ref{mainthm:MCF}. \\

When $X$ is a K\"ahler manifold, 
we consider the space $\Aa$ of $\n \in \Aa_0$ 
with $F_\n^{0,2} =0$, where $F_\n^{0,2}$ is the $(0,2)$-part of $F_\n$. 
The space $\Aa$ can be considered as the set of holomorphic bundle structures in $L$ by \cite[Section 7.1]{Kob}. 
For $\n \in \Aa$, we relate $H(\n)$ to the exterior derivative of the angle function $\theta_\n$ defined by \eqref{eq:r theta} 
as in the case of Lagrangian submanifolds. 

\begin{theorem}[Theorem \ref{thm:Dazord}] \label{mainthm:Dazord}
Suppose that $X$ is compact and connected K\"ahler manifold. 
For any $\n \in \Aa$, we have 
$$
H(\n)
= - (\det G_\n)^{1/4} (G_\n^{-1})^* (J d \theta_\n).   
$$
Here, $G_\n = \id_{TX} - (-\i F_\n)^\sharp \circ (-\i F_\n)^\sharp \in \Gamma (X, \mathop{\mathrm{End}} TX)$, 
where $(-\sqrt{-1}F_\n)^\sharp \in \Gamma (X, \mathop{\mathrm{End}} TX)$ is defined by $u\mapsto \left(-\sqrt{-1}i(u)F_{\nabla}\right)^{\sharp}$, 
$J$ is the complex structure on $X$ 
and we set $J d \theta_\n = d \theta_\n (J (\cdot))$. 

In particular, for $\n \in \Aa$, 
$H(\n)=0$ if and only if $\n$ is a dHYM connection with phase $e^{\i \theta}$ for some $\theta \in \R$. \end{theorem}

Jacob and Yau \cite{JY} considered a similar volume functional 
in \cite[Definition 3.1]{JY} for Hermitian metrics on a holomorphic line bundle
and defined the mean curvature and the line bundle mean curvature flow. 
In Remark \ref{rem:JY MC}, 
we explain the relation between the volume functional $V$ in this paper and that in \cite{JY} 
and show that we can recover the first variation formula in \cite[Proposition 3.4]{JY} from Theorem \ref{mainthm:Dazord}. \\

Due to the introduction of the flow \eqref{maineq:lbmcf} for Hermitian connections, many issues arise simultaneously. 
For instance, are all $\nabla_{t}$ integrable if the initial connection $\nabla_{t}|_{t=0}$ is integrable (i.e. $\nabla_{t}|_{t=0}\in \Aa$) along the line bundle mean curvature flow? 
It seems to be true comparing the Lagrangian mean curvature flow with a slogan that ``integrable'' corresponds to ``Lagrangian'' in mirror symmetry. 
Actually, in the case of Lagrangian mean curvature flows, the Lagrangian condition is preserved along the mean curvature flow in a K\"ahler-Einstein manifold (see \cite{Smo}). 
Following this slogan, the authors expect that studies of Lagrangian mean curvature flows can be imported into the study of 
our line bundle mean curvature flow \eqref{maineq:lbmcf}. 
Moreover, the authors presume that the Lagrangian deformation corresponds to the deformation in $\Aa$, 
and the Hamiltonian deformation corresponds to the deformation in the space of the Chern connections 
of Hermitian metrics of a line bundle with a fixed holomorphic structure. 
From this perspective, the authors also expect that studies of the Hamiltonian stability of Lagrangian submanifolds can be also imported into the study of Hermitian connections.

\subsection{The ``mirror'' of special Lagrangian, associator, Cayley equalities}
\label{sec:main equality}

Special Lagrangian, associator, Cayley equalities are fundamental identities in calibrated geometry introduced in \cite{HL}. 
Using these, we can show that 
the real part of a holomorphic volume, a $G_2$-structure, a $\Sp$-structure are calibrations 
and characterize special Lagrangian, associative, Cayley submanifolds, respectively, 
by the vanishing of certain tensors. 
We first show the ``mirrors'' of these equalities, 
which are predicted by the real Fourier--Mukai transform in \cite[Lemmas 4.3 and 5.5]{KYFM}. 
The following is considered to be the ``mirror'' of the Cayley equality.

\begin{theorem}[Theorem \ref{thm:Cayley eq}] \label{mainthm:Cayley eq}
Let $X^8$ be an 8-manifold with a ${\rm Spin}(7)$-structure $\Phi$
and $L \to X$ be a smooth complex line bundle with a Hermitian metric $h$.
Let $\mathcal{A}_{0}$ be the space of Hermitian connections of $(L,h)$. 
We regard the curvature 2-form $F_\n$ of $\n$ as a $\i \R$-valued closed 2-form on $X$. 
Denote by $\pi^k_\l: \Om^k \to \Om^k_\l$ the projection onto the $\Sp$-irreducible component of rank $\l$. 
Then, for any $\n \in \Aa_0$, we have 
\begin{align*}
&\left( 1+ \frac{1}{2} \la F_\n^2, \Phi \ra + \frac{* F^4_\n}{24}  \right)^2
+
4 \left| \pi^2_7 \left( F_\n + \frac{1}{6} * F_\n^3\right) \right|^2
+
2 \left| \pi^4_7 \left( F_\n^2 \right) \right|^2 \\
=&
\det (\id_{TX} + (-\i F_\n)^\sharp), 
\end{align*}
where $(-\sqrt{-1}F_\n)^\sharp \in \Gamma (X, \mathop{\mathrm{End}} TX)$ is defined by $u\mapsto \left(-\sqrt{-1}i(u)F_{\nabla}\right)^{\sharp}$. 
In particular, 
\[
\left| 1+ \frac{1}{2} \la F_\n^2, \Phi \ra + \frac{* F^4_\n}{24}  \right|
\leq \sqrt{\det (\id_{TX} + (-\i F_\n)^\sharp)}
\]
for any $\n \in \Aa_0$ and the equality holds 
if and only if $\n$ is a $\Sp$-dDT connection. 
\end{theorem}

Note that the $\Sp$-structure $\Phi$ does not have to be torsion-free and $X^8$ is not necessarily compact. 
We need some tricky and complicated computations for the proof 
and we would not have found this without the prediction by the real Fourier--Mukai transform.

Theorem \ref{mainthm:Cayley eq} is the core of results in Section \ref{sec:main equality}. 
From Theorem \ref{mainthm:Cayley eq}, 
we obtain the ``mirror'' of the associator equality in the $G_2$-case 
(Theorem \ref{thm:asso eq}) 
and that of the special Lagrangian equality in the K\"ahler case of dimension 3 or 4 
(Theorems \ref{thm:SL3 eq} and \ref{thm:SL4 eq}). 
Moreover, there are many applications of these equalities. 
The first application is the following Theorem \ref{mainthm:Cayvol ineq} 
together with the corresponding statements 
for the $G_2$-case 
(Theorem \ref{thm:assovol ineq}, Corollaries \ref{cor:volmin G2} and \ref{cor:flat G2})
and the K\"ahler case of dimension 3 or 4 
(Theorem \ref{thm:SL34vol ineq}, Corollaries \ref{cor:volmin SL34} and \ref{cor:flat SL34}). 
By integrating the ``mirror'' of the Cayley equality, 
we can relate the volume functional $V$ to a $\Sp$-dDT connection and 
we see that the volume of a $\Sp$-dDT connection is topological. 
This further implies some interesting results.

\begin{theorem}[Theorem \ref{thm:Cayvol ineq}, Corollaries \ref{cor:volmin} and \ref{cor:flat Spin7}] \label{mainthm:Cayvol ineq}
In addition to the assumptions of Theorem \ref{mainthm:Cayley eq}, 
suppose that $X^8$ is compact and connected.  
\begin{enumerate}
\item
For any $\n \in \Aa_0$, we have 
\begin{align} \label{maineq:Cayley eq}
\left| \int_X \left( 1 + \frac{1}{2} \la F_\n^2, \Phi \ra + \frac{* F^4_\n}{24}  \right) \vol_g \right|
\leq V(\n)
\end{align}
and the equality holds if and only if $\n$ is a $\Sp$-dDT connection. 

\item
If the $\Sp$-structure $\Phi$ is torsion-free, the left hand side of \eqref{maineq:Cayley eq} 
is given by
$$
\left| {\rm Vol}(X) +\left( -2 \pi^2 c_1(L)^2 \cup [\Phi]  + \frac{2}{3} \pi^4 c_1(L)^4 \right) 
\cdot [X] \right|, 
$$
where $c_1(L)$ is the first Chern class of $L$. 
In particular, 
for any $\Sp$-dDT connection $\n$, 
$V(\n)$ is topological 
and 
any $\Sp$-dDT connection is a global minimizer of $V$. 

\item
Suppose that the $\Sp$-structure $\Phi$ is torsion-free and $L$ is a flat line bundle. 
Then, any $\Sp$-dDT connection is a flat connection. 
In particular, the moduli space of $\Sp$-dDT connections 
is $H^1(X, \R)/2 \pi H^1(X, \Z)$. 
\end{enumerate}
\end{theorem}

It is known that there are similar results to (1) and (2) for calibrated submanifolds and HYM, $G_2, \Sp$-instantons. 
That is, each calibrated submanifold is homologically volume minimizing and the volume is topological 
by \cite{HL}. 
Similarly, any HYM, $G_2, \Sp$-instanton is a global minimizer of Yang--Mills functional and the value is topological by \cite{Tian}. 
These make us confirm more the similarities explained at the beginning of the introduction.

Theorem \ref{mainthm:Cayvol ineq} (3) states that 
the moduli space of $\Sp$-dDT connections on a flat line bundle is completely determined. 
Note that there are analogous results to (3) for HYM, $G_2, \Sp$-instantons on line bundles. 
(For example, see \cite[Section 7]{MS2} for the case of $\Sp$-instantons. We can also prove in the other cases similarly.)
However, these results depend on the fact that 
the conditions of HYM, $G_2, \Sp$-instantons on line bundles are linear. 
Since the defining equation of the $\Sp$-dDT is nonlinear, 
(3) is a much more nontrivial result (though results look similar). \\

Next, we give another application of the ``mirror'' of Cayley and associator equalities. 
That is, we study the relation between $\Sp$-dDT connections and $G_2$-dDT connections. 
Let $(Y^7, \varphi)$ be a $G_2$-manifold and 
$L \rightarrow Y$ be a smooth complex line bundle with a Hermitian metric $h$. 
Then, $X^8=S^1 \times Y^7$ admits a canonical torsion-free $\Sp$-structure. 
Denote by $\pi_Y: X^8=S^1 \times Y^7 \to Y^7$ the projection. 

In \cite[Lemma 7.1]{KYFM}, 
we prove that a Hermitian connection $\nabla$ of $(L,h)$ 
is a $G_2$-dDT connection if and only if the pullback $\pi_Y^* \nabla$ is a $\Sp$-dDT connection of $\pi_Y^*L$. 
However, there may be a $\Sp$-dDT connection of $\pi_Y^*L$ which is not of the form $\pi_Y^* \nabla$. 
In fact, as an application of the ``mirror'' of Cayley and associator equalities, 
we can show that any $\Sp$-dDT connection of $\pi_Y^*L$ is essentially 
the pullback of a $G_2$-dDT connection if $Y^7$ is compact and connected. 

\begin{theorem}[Theorem \ref{thm:Spin7G2dDT}] \label{mainthm:Spin7G2dDT}
Suppose that $Y^7$ is a compact and connected $G_2$-manifold. 

\begin{enumerate}
\item
The pullback $\pi_Y^*L \to X^8$ admits a $\Sp$-dDT connection if and only if 
$L \to Y^7$ admits a $G_2$-dDT connection. 

\item
For any $\Sp$-dDT connection $\widetilde \n$ of $\pi_Y^* L$, 
there exist a $G_2$-dDT connection $\n$ of $L$ and 
a closed 1-form $\xi \in \i \Om^1 (X^8)$ 
such that 
$\widetilde \n = \pi_Y^* \n + \xi$. 

\item 
Denote by $\Mm_\Sp$ the moduli space of $\Sp$-dDT connections of $\pi_Y^* L$ 
and denote by $\Mm_{G_2}$ the moduli space of $G_2$-dDT connections of $L$ 
as defined by \eqref{eq:mod Spin7} and \eqref{eq:mod G2}. 
Then, $\Mm_\Sp$ is homeomorphic to $S^1 \times \Mm_{G_2}$. 
\end{enumerate}
\end{theorem}

Results similar to Theorem \ref{mainthm:Spin7G2dDT} hold 
between $G_2$-dDT and dHYM connections (Theorem \ref{thm:G2dDTdHYM})
and between $\Sp$-dDT and dHYM connections (Theorem \ref{thm:Spin7dDTdHYM}). 

In (2), we show that any $\Sp$-dDT connection of $\pi_Y^* L$ is essentially the pullback of a $G_2$-dDT connection. 
We can prove (3) from (2) and it implies that the topology of $\Mm_\Sp$ is determined by that of $\Mm_{G_2}$. 
Theorem \ref{mainthm:Spin7G2dDT} (1) might have interesting implications. 
In the case between $G_2$-dDT and dHYM connections (Theorem \ref{thm:G2dDTdHYM}), 
we show that 
the existence of a $G_2$-dDT connection is equivalent to that of a dHYM connection. 
It is conjectured in \cite[Conjecture 1.5]{CJY} that the existence of a dHYM connection is 
equivalent to a certain stability condition and 
the conjecture is partially proved in \cite{Chen}. 
This implies that 
the existence of a $G_2$-dDT connection of $\pi_Y^*L \to X^7$ would be equivalent to this stability condition.  
More generally, 
the existence of a $G_2$- or $\Sp$-dDT connection of a general line bundle over a general $G_2$- or $\Sp$-manifold 
might be related to a certain stability condition.

There are similar results for calibrated submanifolds and HYM, $G_2, \Sp$-instantons. 
That is, the moduli space of irreducible $\Sp$-instantons of $\pi_Y^* L$ is homeomorphic to 
the product of $S^1$ and the moduli space of irreducible $G_2$-instantons of $L$ by \cite[Theorem 1.17]{Wang}. 
By \cite[Proposition 5.20]{Ohst}, we see that 
the moduli space of all local Cayley deformations of $S^1 \times A^3$, 
where $A^3$ is a given compact associative submanifold in $Y^7$, 
is identified with the moduli space of all local associative deformations of $A^3$. 
These also make us confirm more the similarities explained at the beginning of the introduction.

One of the problems of the geometry of $G_2, \Sp$-dDT connections is that few explicit examples are known. 
(Recently, Lotay and Oliveira \cite{LO} constructed 
nontrivial examples of $G_2$-dDT connections on the trivial complex line bundle over a manifold with a coclosed $G_2$-structure.)
Theorems \ref{mainthm:Cayvol ineq} (3),  \ref{mainthm:Spin7G2dDT} and \ref{thm:Spin7dDTdHYM} 
impose several restrictions 
to construct nontrivial examples on a compact and connected $\Sp$-manifold.

\SkipTocEntry \subsection*{Organization of this paper}
This paper is organized as follows. 
Section \ref{sec:basic} gives basic identities in $G_2$- and ${\rm Spin}(7)$-, and 
3,4-dimensional Calabi--Yau geometry that are used in this paper. 
Section \ref{sec:vol} is devoted to the study of the ``mirror'' of the mean curvature flow. 
We prove Theorem \ref{mainthm:MCF} (Theorem \ref{short-ex}) here.
In Section \ref{sec:Spin7dDT}, 
reviewing the definition of $\Sp$-dDT connections we state 
Theorem \ref{mainthm:Cayley eq} (Theorem \ref{thm:Cayley eq})
with some remarks and prove 
Theorem \ref{mainthm:Cayvol ineq} (Theorem \ref{thm:Cayvol ineq}). 
We state and prove the corresponding statement 
on $G_2$-manifolds (Theorems \ref{thm:asso eq}, \ref{thm:assovol ineq}, Corollaries \ref{cor:volmin G2} and \ref{cor:flat G2}) in Section \ref{sec:G2dDT} 
and 
that on 3, 4-dimensional K\"ahler manifolds (Theorems \ref{thm:SL3 eq}, \ref{thm:SL4 eq}, \ref{thm:SL34vol ineq}, Corollaries \ref{cor:volmin SL34} and \ref{cor:flat SL34}) 
in Section \ref{sec:SL34}. 
In Section \ref{sec:SL34}, we also prove related results that hold for any dimension 
including Theorem \ref{mainthm:Dazord} (Theorem \ref{thm:Dazord}) under an additional assumption. 
In Section \ref{sec:hol red}, 
we study the relation between $\Sp$-dDT and $G_2$-dDT connections and prove 
Theorem \ref{mainthm:Spin7G2dDT} (Theorem \ref{thm:Spin7G2dDT}) here. 
We also prove the corresponding statement between $G_2$-dDT and dHYM connections (Theorem \ref{thm:G2dDTdHYM})
and that between $\Sp$-dDT and dHYM connections (Theorem \ref{thm:Spin7dDTdHYM}). 
In Appendix \ref{app:eq pt}, we prove Theorems \ref{thm:Cayley eq}, \ref{thm:asso eq}, 
\ref{thm:SL3 eq} and \ref{thm:SL4 eq}. 
Appendix \ref{app:notation} is the list of notation in this paper.

\vspace{0.5cm}
\noindent{{\bf Acknowledgements}}: 
The authors would like to thank anonymous referees for the careful reading 
of an earlier version of this paper and many useful comments which helped to improve the quality of the paper.


\section{Basics on $G_2$-, ${\rm Spin}(7)$-, ${\rm SU}(3)$- and ${\rm SU}(4)$-geometry} \label{sec:basic}
In this section, 
we collect some basic definitions and equations on 
$G_2$-, ${\rm Spin}(7)$-, ${\rm SU}(3)$- and ${\rm SU}(4)$-geometry which we need in the calculations in this paper.

\subsection{The Hodge star operator}
Let $V$ be an $n$-dimensional oriented real vector space with an inner product $g$. 
Denote by $\la \cdot, \cdot \ra$ the 
induced inner product on $\Lambda^k V^*$ from $g$. 
Let $\ast$ be the Hodge star operator.
The following identities are frequently used throughout this paper. 

For $\alpha, \beta \in \Lambda^k V^*$ and $v \in V$, we have 
\[
\begin{aligned}
\ast^2|_{\Lambda^k V^*} &= (-1)^{k(n-k)} {\rm id}_{\Lambda^k V^*}, & 
\la \ast \alpha, \ast \beta \ra &= \la \alpha, \beta \ra, \\ 
i(v) \ast \alpha &= (-1)^k \ast (v^\flat \wedge \alpha), & 
\ast( i(v) \alpha) &= (-1)^{k+1} v^\flat \wedge \ast \alpha. 
\end{aligned}
\]



\subsection{$G_2$-geometry} \label{sec:G2 geometry}
Let $V$ be an oriented $7$-dimensional vector space. A \emph{$G_2$-structure} on $V$ is 
a 3-form $\varphi \in \Lambda^3 V^*$ such that there is a positively oriented basis 
$\{\, e_i \,\}_{i=1}^7$ of $V$ 
with the dual basis $\{\, e^i \,\}_{i=1}^7$ of $V^\ast$ satisfying 
\begin{equation} \label{varphi}
\varphi = e^{123} + e^{145} + e^{167} + e^{246} - e^{257} - e^{347} - e^{356},
\end{equation}
where $e^{i_1 \dots i_k}$ is short for $e^{i_1} \wedge \cdots \wedge e^{i_k}$. Setting $\vol := e^{1 \cdots 7}$, 
the 3-form $\varphi$ uniquely determines an inner product $g_\varphi$ via 
\begin{equation} \label{eq:form-1def}
g_\varphi(u,v)\; \vol = \dfrac16 i(u) \varphi \wedge i(v) \varphi \wedge \varphi 
\end{equation}
for $u,v \in V$. 
It follows that any oriented basis $\{\, e_i \,\}_{i=1}^7$ for which \eqref{varphi} holds is orthonormal with respect to $g_\varphi$. Thus, the Hodge-dual of $\varphi$ with respect to $g_\varphi$ is given by 
\begin{equation} \label{varphi*}
\ast \varphi = e^{4567} + e^{2367} + e^{2345} + e^{1357} - e^{1346} - e^{1256} - e^{1247}.
\end{equation}
The stabilizer of $\varphi$ is known to be the exceptional $14$-dimensional simple Lie group 
$G_2 \subset {\rm GL}(V)$. The elements of $G_2$ preserve both $g_\varphi$ and $\vol$, that is, 
$G_2 \subset {\rm SO}(V, g_\varphi)$.

We summarize important well-known facts about the decomposition of exterior powers of $G_2$-modules into irreducible summands. 
Denote by $V_k$ the $k$-dimensional irreducible $G_2$-module if there is a unique such module. For instance, $V_7$ is the irreducible $7$-dimensional $G_2$-module $V$ from above, 
and $V_7^* \cong V_7$. For its exterior powers, we obtain the decompositions
\begin{equation} \label{eq:DiffForm-V7}
\begin{array}{rlrl}
\Lambda^0 V^* \cong \Lambda^7 V^* \cong V_1, \quad
& \Lambda^2 V^*  \cong \Lambda^5 V^* \cong V_7 \oplus  V_{14},\\[2mm]
\Lambda^1 V^* \cong \Lambda^6 V^* \cong V_7, \quad
& \Lambda^3 V^* \cong \Lambda^4 V^* \cong V_1 \oplus V_7 \oplus V_{27},
\end{array}
\end{equation}
where $\Lambda^k V^* \cong \Lambda^{7-k} V^*$ due to the $G_2$-invariance of the Hodge isomorphism $\ast: \Lambda^k V^* \to \Lambda^{7-k} V^*$. We denote by $\Lambda^k_\l V^* \subset \Lambda^k V^*$ the subspace isomorphic to $V_\l$. 
Let 
\[
\pi^k_\l: \Lambda^k V^* \rightarrow \Lambda^k_\l V^*
\]
be the canonical projection. 
Note that we have the following. 
\begin{equation}\label{decom-L-V7}
\begin{aligned}
\Lambda^2_7 V^* =& \{\, i(u) \varphi \mid u \in V \,\}
= \{\, \alpha \in \Lambda^2 V^* \mid * (\varphi \wedge \alpha) = 2 \alpha \,\},\\
\Lambda^2_{14} V^* =& \{\, \alpha \in \Lambda^2 V^* \mid \ast \varphi \wedge \alpha = 0\,\} 
= \{\, \alpha \in \Lambda^2 V^* \mid * (\varphi \wedge \alpha) = - \alpha \,\},\\
\Lambda^3_1 V^* =& \R \varphi, \\
\Lambda^3_7 V^* =& \{\, i(u) * \varphi \in \Lambda^3 V^* \mid u \in V \,\}. 
\end{aligned}
\end{equation}
The following equations are well-known and useful in this paper. 

\begin{lemma} \label{lem:G2 identities}
For any $u \in V$, we have the following identities. 
\[
\begin{aligned}
\varphi \wedge i(u) * \varphi &= -4 * u^{\flat}, 
\\
* \varphi \wedge i(u) \varphi &= 3 * u^{\flat}, \\
\varphi \wedge i(u) \varphi &= 2 * (i(u) \varphi) = 2 u^{\flat} \wedge * \varphi. 
\end{aligned}
\]
\end{lemma}

\begin{definition} \label{def:G2mfd}
Let $X$ be an oriented 7-manifold. A \emph{$G_2$-structure} on $X$ is a $3$-form $\varphi \in \Om^3$ 
such that at each $p \in X$ there is a positively oriented basis 
$\{\, e_i \,\}_{i=1}^7$ of $T_p X$ such that $\varphi_p \in \Lambda^3 T^*_p X$ is of the form \eqref{varphi}. As noted above, $\varphi$ determines a unique Riemannian metric $g = g_\varphi$ on $X$ by \eqref{eq:form-1def}, 
and the basis $\{\, e_i \,\}_{i=1}^7$ is orthonormal with respect to $g$.
A $G_2$-structure $\varphi$ is called \emph{torsion-free} if 
it is parallel with respect to the Levi-Civita connection of $g=g_\varphi$. 
A manifold with a torsion-free $G_2$-structure is called a \emph{$G_2$-manifold}. 
\end{definition}

A manifold $X$ admits a $G_2$-structure if and only if 
its frame bundle is reduced to a $G_2$-subbundle. 
Hence, considering its associated subbundles, 
we see that 
$\Lambda^* T^* X$ has the same decomposition as in \eqref{eq:DiffForm-V7}. 
The algebraic identities above also hold. 

\subsection{${\rm Spin}(7)$-geometry} \label{sec:Spin7 geometry}
Let $W$ be an $8$-dimensional oriented real vector space. 
A \emph{$\Sp$-structure on $W$} is 
a 4-form $\Phi \in \Lambda^4 W^*$ such that there is a positively oriented basis 
$\{\, e_i \,\}_{i=0}^7$ of $W$ 
with dual basis $\{\, e^i \,\}_{i=0}^7$ of $W^\ast$ satisfying 
\begin{equation}\label{Phi4}
\begin{aligned}
\Phi := & e^{0123} + e^{0145} + e^{0167} + e^{0246} - e^{0257} - e^{0347} - e^{0356}\\
& + e^{4567} + e^{2367} + e^{2345} + e^{1357} - e^{1346} - e^{1256} - e^{1247}.
\end{aligned}
\end{equation}
where $e^{i_1 \dots i_k}$ is short for $e^{i_1} \wedge \cdots \wedge e^{i_k}$.
Defining forms 
$\varphi$ and $\ast_7 \varphi$ on $V := \rmspan \{\, e_i \,\}_{i=1}^7 \subset W$ 
as in \eqref{varphi} and \eqref{varphi*}, where $*_7$ stands for the Hodge star operator on $V$, 
we have 
\[
\Phi = e^0 \wedge \varphi + \ast_7 \varphi. 
\]
Note that $\Phi$ is self-dual, that is, $*_8 \Phi = \Phi$, where 
$*_8$ is the Hodge star operator on $W$. 
It is known that $\Phi$ uniquely determines an inner product $g_\Phi$ and the volume form 
and the subgroup of ${\rm GL}(W)$ preserving $\Phi$ is isomorphic to $\Sp$. 
As in Definition \ref{def:G2mfd}, we can define 
an 8-manifold with a $\Sp$-structure and a $\Sp$-manifold.

Denote by $W_k$ the $k$-dimensional irreducible $\Sp$-module if there is a unique such module. 
For example, $W_8$ is the irreducible $8$-dimensional $\Sp$-module from above, and $W_8^* \cong W_8$. 
The group $\Sp$ acts irreducibly 
on $W_7 \cong \R^7$ as the double cover of ${\rm SO}(7)$. 
For its exterior powers, we obtain the decompositions 
\begin{equation}\label{decom1-L-W8}
\begin{aligned}
\Lambda^0 W^* &\cong \Lambda^8 W^* \cong W_1, \quad
& \Lambda^2 W^*  \cong \Lambda^6 W^* &\cong W_7 \oplus  W_{21},\\
\Lambda^1 W^* &\cong \Lambda^7 W^* \cong W_8, \quad
& \Lambda^3 W^* \cong \Lambda^5 W^* &\cong W_8 \oplus W_{48},\\
\Lambda^4 W^* &\cong W_1 \oplus W_7 \oplus W_{27} \oplus W_{35}
\end{aligned}
\end{equation}
where $\Lambda^k W^* \cong \Lambda^{8-k} W^*$ due to the $\Sp$-invariance of the Hodge isomorphism 
$*=\ast_8: \Lambda^k W^* \to \Lambda^{8-k} W^*$. Again, we denote by $\Lambda^k_\l W^* \subset \Lambda^k W^*$ the subspace isomorphic to $W_\l$ in the above notation.

The space $\Lambda^k_7 W^*$ for $k = 2,4,6$ 
is explicitly given as follows. 
For the explicit descriptions of the other irreducible summands, 
see for example \cite[(4.7)]{KLS}. 

\begin{lemma} [{\cite[Lemma 4.2]{KLS}, \cite[Lemma 3.4]{KYFM}}]
\label{lem:lambdas}
Let $e^0 \in W^*$ be a unit vector. 
Set $V^* = (\R e^0)^\perp$, the orthogonal complement of $\R e^0$. 
The group $\Sp$ acts irreducibly on $V^*$ as the double cover of ${\rm SO}(7)$, 
and hence, we have the identification $V^* \cong W_7$. 
Then, the following maps are $\Sp$-equivariant isometries. 
\begin{equation} \label{def:lambda}
\lambda^k: V^* \longrightarrow \Lambda^k_7 W^*, \quad 
\begin{array}{ll} 
\lambda^2(\alpha) := \frac{1}{2} \left( e^0 \wedge \alpha + i (\alpha^\sharp) \varphi \right), \\[2mm]
\lambda^4(\alpha) := 
\frac{1}{\sqrt{8}} \left( e^0 \wedge i (\alpha^\sharp) \ast_7 \varphi - \alpha \wedge \varphi \right), \\[2mm]
\lambda^6(\alpha) := \frac{1}{3} \Phi \wedge \lambda^2(\alpha) = \ast_8 \lambda^2(\alpha). 
\end{array}
\end{equation}
Here, $*=\ast_8$ and $\ast_7$ are the Hodge star operators on $W^*$ and $V^*$, respectively.
\end{lemma}

The following decomposition is also well-known. 
\begin{align*}
\Lambda^2_7 W^* 
=& \{\, \alpha \in \Lambda^2 W^* \mid \alpha \wedge \Phi = 3 * \alpha \,\}, \\
\Lambda^2_{21} W^*
=& 
\{\, \alpha \in \Lambda^2 W^* \mid \alpha \wedge \Phi = - * \alpha \,\}.  
\end{align*}
Denote by 
\begin{align*}
\pi^k_\l : \Lambda^k W^* \rightarrow \Lambda^k_\l W^*
\end{align*}
the canonical projection. 
Since 
$\alpha = \pi^2_7 (\alpha) + \pi^2_{21} (\alpha)$ and 
$* (\Phi \wedge \alpha) = 3 \pi^2_7 (\alpha) - \pi^2_{21} (\alpha)$
for a 2-form $\alpha \in \Lambda^2 W^*$, it follows that 
\begin{equation} \label{eq:Spin7 proj 2}
\pi^2_7 (\alpha) = \frac{\alpha + * (\Phi \wedge \alpha)}{4}, \quad
\pi^2_{21} (\alpha) = \frac{3 \alpha - * (\Phi \wedge \alpha)}{4}. 
\end{equation}
For the wedge product of two 2-forms, we have the following by \cite[Proposition 2.1]{MS2}. 
\begin{align} \label{eq:F2decomp}
\begin{split}
\Lambda^2_7 W^* \wedge \Lambda^2_7 W^* 
&= \Lambda^4_1 W^*  \oplus \Lambda^4_{27} W^*, \\
\Lambda^2_7 W^* \wedge \Lambda^2_{21} W^* 
&= \Lambda^4_7 W^*  \oplus \Lambda^4_{35} W^*, \\
\Lambda^2_{21} W^* \wedge \Lambda^2_{21} W^* 
&= \Lambda^4_1 W^*  \oplus \Lambda^4_{27} W^* \oplus \Lambda^4_{35} W^*. 
\end{split}
\end{align}
The following is important in the proof of Proposition \ref{prop:Cayley eq pt}. 

\begin{proposition}[{\cite[Proposition 2.7]{KYSpin7}}] \label{prop:2form norm}
For any $\beta \in \Lambda^2_7 W^*$ and $\gamma \in \Lambda^2_{21} W^*$, we have 
\begin{align*}
\beta^3 &= \frac{3}{2} |\beta|^2 * \beta, \\
|\beta|^4 &= \frac{2}{3} |\beta^2|^2, \\
|\gamma|^4 &= |\gamma^2|^2 - \frac{1}{3} * \gamma^4, \\
|\beta|^2 |\gamma|^2 &= 2 |\beta \wedge \gamma|^2. 
\end{align*}
\end{proposition}

The following lemma is very useful, 
which is used in the proof of Lemmas \ref{lem:Caypt deg} and \ref{lem:SL4pt deg} several times. 
Note that the following identity depends only on the metric and the orientation 
and is independent of the $\Sp$-structure.

\begin{lemma} \label{lem:24form}
For any $F \in \Lambda^2 W^*$ and $\xi \in \Lambda^4 W^*$, we have 
$$
*\left( \xi \wedge (* F^3)^2 \right)
= \frac{3}{2} \la F^2, \xi \ra* F^4. 
$$
\end{lemma}

\begin{proof}
Set $F_{i j} =F(e_i, e_j)$. Then, $F= (1/2) \sum_{i, j} F_{i j} e^{i j}$. 
We use 
$$
F^{k+1} = \frac{1}{2} \sum_{i, j} F_{i j} e^{i j} \wedge F^k \quad \mbox{and} \quad 
* \left( (i(e_j) F) \wedge F^3 \right) 
= 
* \left( \frac{i(e_j) F^4}{4} \right)
=
-\left( \frac{*F^4}{4} \right) e^j
$$ 
several times to prove the statement. 
Since 
$F^3 = (1/2) \sum_{i,j} F_{i j} e^{i j} \wedge F^2$, 
we have 
\begin{align} \label{eq:24form 1}
\begin{split}
*\left( \xi \wedge (* F^3)^2 \right)
=& \left \la F^3,  (* F^3) \wedge \xi \right \ra \\
=& \frac{1}{2} \sum_{i,j} F_{i j} \left \la F^2, i(e_j) i(e_i) ((* F^3) \wedge \xi) \right \ra
=
I_1 + I_2 + I_3, 
\end{split}
\end{align}
where 
\begin{align*}
I_1 &= \frac{1}{2} \sum_{i,j} F_{i j} \left \la (* F^3) (e_i, e_j) \xi, F^2 \right \ra, \\
I_2 &= - \sum_{i,j} F_{i j} \left \la i(e_i) (* F^3) \wedge i(e_j) \xi, F^2 \right \ra, \\
I_3 &= \frac{1}{2} \sum_{i,j} F_{i j} \left \la (* F^3) \wedge i(e_j) i(e_i) \xi, F^2 \right \ra. \\
\end{align*}
We compute $I_1, I_2$ and $I_3$. 
Since 
\begin{align} \label{eq:24form 2}
\sum_{i,j} F_{i j} (* F^3) (e_i, e_j) = 
\sum_{i,j} F_{i j} *(e^{i j} \wedge F^3) 
= 2 * F^4, 
\end{align}
we have 
\begin{align} \label{eq:24form 3}
I_1 = \la  F^2, \xi \ra * F^4. 
\end{align}

Since 
\begin{align} \label{eq:24form 4}
\begin{split}
\sum_i F_{i j} i(e_i) (* F^3) 
= \sum_i F_{i j} * (e^i \wedge F^3) 
= - * \left( (i(e_j)F) \wedge F^3 \right)
=\frac{* F^4}{4} e^j,
\end{split} 
\end{align}
we obtain 
\begin{align} \label{eq:24form 5}
I_2
= -\frac{* F^4}{4} \sum_j \left \la e^j \wedge i(e_j) \xi, F^2 \right \ra 
= - \la  F^2, \xi \ra * F^4.  
\end{align}

Since $(1/2) \sum_{i,j} F_{i j} i(e_j) i(e_i) \xi = *(F \wedge * \xi)$ 
and $F^2 = (1/2) \sum_{i,j} F_{i j} e^{i j} \wedge F$, we compute 
\begin{align*}
I_3
&= \left \la (* F^3) \wedge  *(F \wedge * \xi), F^2 \right \ra \\
&= (1/2) \sum_{i,j} F_{i j} \left \la i(e_j) i(e_i) \left( (* F^3) \wedge  *(F \wedge * \xi) \right), F \right \ra
= I_{3,1} + I_{3,2} + I_{3,3}, 
\end{align*}
where 
\begin{align*}
I_{3,1}&= \frac{1}{2} \sum_{i,j} F_{i j} (* F^3) (e_i, e_j) \left \la *(F \wedge * \xi), F \right \ra, \\
I_{3,2}&=- \sum_{i,j} F_{i j} \left \la i(e_i) (* F^3) \wedge i(e_j) *(F \wedge * \xi), F \right \ra, \\
I_{3,3}&= \frac{1}{2} \left( \sum_{i,j} F_{i j} i(e_j) i(e_i) *(F \wedge * \xi) \right) \left \la * F^3, F \right \ra. 
\end{align*}
By \eqref{eq:24form 2}, we have 
$$
I_{3,1} = \la  F^2, \xi \ra * F^4. 
$$
By \eqref{eq:24form 4}, we have 
\begin{align*}
I_{3,2}
=- \frac{*F^4}{4} \sum_{j}  \left \la e^j \wedge i(e_j) *(F \wedge * \xi), F \right \ra 
= - \frac{1}{2} \la F^2, \xi \ra * F^4. 
\end{align*}
Since 
$\sum_{i,j} F_{i j} i(e_j) i(e_i) *(F \wedge * \xi) = 2 * (F^2 \wedge * \xi)$, we have 
$$
I_{3,3}= * (F^2 \wedge * \xi) \left \la * F^3, F \right \ra = \la F^2, \xi \ra * F^4. 
$$
Hence, it follows that 
\begin{align} \label{eq:24form 6}
I_3= \frac{3}{2} \la F^2, * \xi \ra * F^4. 
\end{align}
By \eqref{eq:24form 1}, \eqref{eq:24form 3}, \eqref{eq:24form 5} and \eqref{eq:24form 6}, the proof is completed. 
\end{proof}

\subsection{${\rm SU}(3)$-geometry} \label{sec:SU3 geometry}
Set $U=\R^6 \cong \C^3$. 
Denote by $\{\, e_i \,\}_{i=2}^7$ and $\{\, e^i \,\}_{i=2}^7$
the standard basis of $U$ and its dual, respectively. 
Denote by $g$ the standard inner product on $U$. 
This together with the standard orientation induces the Hodge star operator $*$. 
Denote by $\om=e^{23}+e^{45}+e^{67}$ the standard K\"ahler form on $U \cong \C^3$. 
Set $f^2=e^2+ \i e^3, f^3=e^4+ \i e^5$ and $f^4=e^6+ \i e^7$. 
The holomorphic volume form $\Om$ is given by $\Om=f^{234}:=f^2 \wedge f^3 \wedge f^4$. 
It is well-known that 
\begin{align} \label{eq:SU3 wk}
* \om = \frac{1}{2} \om^2, \qquad * {\rm Re} \Om = \Im \Om, \qquad * \Im \Om = - {\rm Re} \Om. 
\end{align}
The standard complex structure $J$ on $U \cong \C^3$ 
induces the decomposition $U \otimes \C = U^{1,0} \oplus U^{0,1}$, 
where $U^{1,0}$ and $U^{0,1}$ are $\i$- and $-\i$-eigenspaces of $J$, respectively.

Set $\Lambda^{p,q} = \Lambda^p (U^{1,0})^* \otimes \Lambda^q (U^{0, 1})^*$. 
Define real vector spaces 
$\lb \Lambda^{p,q} \rb$ for $p \neq q$ and $[\Lambda^{p, p}]$ by 
\[
\begin{aligned}
\lb \Lambda^{p,q} \rb 
=& (\Lambda^{p, q} \oplus \Lambda^{q, p}) \cap \Lambda^{p + q} U^*
= \{\,\alpha \in \Lambda^{p, q} \oplus \Lambda^{q, p} \mid \bar \alpha = \alpha \,\}, \\
[\Lambda^{p, p}]
=&
\Lambda^{p, p} \cap \Lambda^{2p} U^*
= \{\, \alpha \in \Lambda^{p, p} \mid \bar \alpha = \alpha \,\}. 
\end{aligned}
\]
The K\"ahler form $\om$ is contained in $[\Lambda^{1,1}]$ 
and denote by $[\Lambda^{1,1}_0]$ the orthogonal complement of $\R \om$ in $[\Lambda^{1, 1}]$. 
Then, we have the decomposition 
\[
\Lambda^{2} U^* = \lb \Lambda^{2,0} \rb \oplus [\Lambda^{1,1}_0] \oplus \R \om
\]
and $[\Lambda^{1,1}], [\Lambda^{1,1}_0]$ are identified with $\mathfrak{u}(3), \mathfrak{su}(3)$, respectively.  
For a subspace $S \subset \Lambda^k U^*$, denote by 
$$
\pi_S: \Lambda^k U^* \to S
$$ 
the orthogonal projection. 
For simplicity, set 
$$
\pi^{\lb p,q \rb} = \pi_{\lb \Lambda^{p,q} \rb}: \Lambda^{p+q} U^* \to \lb \Lambda^{p,q} \rb,  \qquad
\pi^{[p, p]} = \pi_{[\Lambda^{p,p}]}: \Lambda^{p+q} U^* \to [\Lambda^{p,p}]. 
$$
Since $* \Lambda^{p,q} = \Lambda^{3-q, 3-p}$, we see that 
\begin{align} \label{eq:Hodge proj SU3}
* \circ \pi^{\lb p,q \rb} = \pi^{\lb 3-p, 3-q \rb} \circ *. 
\end{align}

\begin{lemma} \label{lem:SU3 id}
For any $u \in U$, we have 
\begin{align*}
*\left( (i(u) {\rm Re} \Om) \wedge {\rm Re} \Om \right) = 2 (J u)^\flat, \\
*\left( (i(u) {\rm Re} \Om) \wedge {\rm Im} \Om \right) = -2 u^\flat, \\
|i(u) {\rm Re} \Om|^2 = 2 |u|^2. 
\end{align*}
\end{lemma}
\begin{proof}
The first two equations follow from \cite[(2), (3)]{MNS}. By the second equation and \eqref{eq:SU3 wk}, we have 
$$
|i(u) {\rm Re} \Om|^2 
= * \left(i(u) {\rm Re} \Om \wedge * (i(u) {\rm Re} \Om) \right)
= * \left((i(u) {\rm Re} \Om) \wedge u^\flat \wedge {\rm Im} \Om \right)
= 2 |u|^2. 
$$
\end{proof}

\begin{lemma} \label{lem:SU3 2form}
\begin{align*}
\lb \Lambda^{2,0} \rb &= \{i(u) {\rm Re} \Om \mid u \in U \} = \{ \beta \in \Lambda^2 U^* \mid *(\om \wedge \beta) = \beta \}, \\
[\Lambda^{1,1}_0]      &= \{ \beta \in \Lambda^2 U^* \mid *(\om \wedge \beta) = - \beta \}, \\
\R \om                   &= \{ \beta \in \Lambda^2 U^* \mid *(\om \wedge \beta) = 2 \beta \}. 
\end{align*}
\end{lemma}

\begin{proof}
The equations for $\R \om$ and $[\Lambda^{1,1}_0]$ hold by \eqref{eq:SU3 wk} and \cite[(13)]{MNS}, respectively. 
The first equation for $\lb \Lambda^{2,0} \rb$ follows from \cite[p.59]{MNS}. 
By the equation $\om \wedge {\rm Re} \Om = 0$ and \eqref{eq:SU3 wk}, we have for $u \in U$ 
\begin{align*}
* \left( \om \wedge i(u) {\rm Re} \Om \right)
=&
- * \left(i(u) \om \wedge {\rm Re} \Om \right) \\
=&
i(J u) \Im \Om
=
\Im \left( i(J u) \Om \right)
= \Im \left( \i i(u) \Om \right) 
=i(u) {\rm Re} \Om, 
\end{align*}
which implies the last equation for $\lb \Lambda^{2,0} \rb$. 
\end{proof}

By Lemmas \ref{lem:SU3 id} and \ref{lem:SU3 2form}, we obtain the following. 

\begin{corollary} \label{cor:SU3 2form norm}
For any $\beta \in \Lambda^2 U^*$, we have 
$$
|\beta \wedge {\rm Re} \Om|^2 = |\beta \wedge {\rm Im} \Om|^2 = 2 |\pi^{\lb 2,0 \rb}(\beta) |^2
$$
\end{corollary}

\begin{proof}
Since $\beta \wedge {\rm Re} \Om = \pi^{\lb 2,0 \rb}(\beta) \wedge {\rm Re} \Om$ 
and $\pi^{\lb 2,0 \rb}(\beta) = i(u) {\rm Re} \Om$ for some $u \in U$ by Lemma \ref{lem:SU3 2form}, 
Lemma \ref{lem:SU3 id} implies that 
$$
|\beta \wedge {\rm Re} \Om|^2 = 4|u|^2 = 2 |\pi^{\lb 2,0 \rb}(\beta) |^2. 
$$
We can compute $|\beta \wedge {\rm Im} \Om|^2$ similarly and the proof is completed. 
\end{proof}

\subsection{${\rm SU}(4)$-geometry} \label{sec:SU4 geometry}
Set $W=\R^8 \cong \C^4$. 
Denote by $\{\, e_i \,\}_{i=0}^7$ and $\{\, e^i \,\}_{i=0}^7$
the standard basis of $W$ and its dual, respectively. 
Denote by $g$ the standard inner product on $W$. 
This together with the standard orientation induces the Hodge star operator $*$. 
Denote by $\om=e^{01}+e^{23}+e^{45}+e^{67}$ the standard K\"ahler form on $W \cong \C^4$. 
Set $f^1=e^0+ \i e^1, f^2=e^2+ \i e^3, f^3=e^4+ \i e^5$ and $f^4=e^6+ \i e^7$. 
The holomorphic volume form $\Om$ is given by $\Om=f^{1234}:=f^1 \wedge f^2 \wedge f^3 \wedge f^4$. 
Using the standard complex structure $J$ on $W \cong \C^4$, 
we can define spaces 
$\lb \Lambda^{p,q} \rb$ and $[\Lambda^{p, p}]$ as in Section \ref{sec:SU3 geometry}. 
For a subspace $S \in \Lambda^k W^*$, denote by 
$$
\pi_S: \Lambda^k W^* \to S
$$ 
the orthogonal projection. 
For simplicity, set 
$$
\pi^{\lb p,q \rb} = \pi_{\lb \Lambda^{p,q} \rb}: \Lambda^{p+q} W^* \to \lb \Lambda^{p,q} \rb,  \qquad
\pi^{[p, p]} = \pi_{[\Lambda^{p,p}]}: \Lambda^{p+q} W^* \to [\Lambda^{p,p}]. 
$$
Since $* \Lambda^{p,q} = \Lambda^{4-q, 4-p}$, we see that 
\begin{align} \label{eq:Hodge proj SU4}
* \circ \pi^{\lb p,q \rb} = \pi^{\lb 4-p, 4-q \rb} \circ *. 
\end{align}
It is well-known that 
\begin{align} \label{eq:SU4 wk}
* \om = \frac{1}{6} \om^3, \qquad * \om^2=\om^2, \qquad * {\rm Re} \Om = {\rm Re} \Om, \qquad * \Im \Om = \Im \Om. 
\end{align}
and 
\begin{align}\label{eq:SU4 40}
\lb \Lambda^{4,0} \rb = \R {\rm Re} \Om \oplus \R \Im \Om. 
\end{align}
The K\"ahler form $\om$ is contained in $[\Lambda^{1,1}]$ 
and denote by $[\Lambda^{1,1}_0]$ the orthogonal complement of $\R \om$ in $[\Lambda^{1,1}]$. 
We have the following irreducible decomposition with respect to the standard ${\rm U}(4)$-action 
$$
\Lambda^{2} W^* = \lb \Lambda^{2,0} \rb \oplus [\Lambda^{1,1}_0] \oplus \R \om, 
$$
where $[\Lambda^{1,1}], [\Lambda^{1,1}_0]$ are identified with $\mathfrak{u}(4), \mathfrak{su}(4)$, respectively.  
Note that 
$$
\pi_{\R \om} (\beta) = \frac{\la \beta, \om \ra}{4} \om. 
$$
These spaces are characterized as follows. 

\begin{lemma} \label{lem:SU4 2form}
\begin{align*}
\lb \Lambda^{2,0} \rb &= \{ \beta \in \Lambda^2 W^* \mid *(\om^2 \wedge \beta) = 2 \beta \}, \\
[\Lambda^{1,1}_0]      &= \{ \beta \in \Lambda^2 W^* \mid *(\om^2 \wedge \beta) = -2 \beta \}, \\
\R \om                   &= \{ \beta \in \Lambda^2 W^* \mid *(\om^2 \wedge \beta) = 6 \beta \}. 
\end{align*}
\end{lemma}

\begin{proof}
The equation for $\R \om$ holds by \eqref{eq:SU4 wk}. 
Set 
$\beta={\rm Re}(f^{12})=e^{02}-e^{13} \in \lb \Lambda^{2,0} \rb$ and 
$\beta'= e^{01}-e^{23} \in [\Lambda^{1,1}_0]$. 
Since 
$\om^2=2\left( e^{0123} + e^{0145} + e^{0167} + e^{2345} + e^{2367} + e^{4567} \right)$, we see that 
$$
*(\om^2 \wedge \beta) = 2* (e^{024567}-e^{134567}) = 2 \beta, \qquad
*(\om^2 \wedge \beta') = 2* (e^{014567}-e^{234567}) = -2 \beta'. 
$$
Since $*(\om^2 \wedge \cdot): \Lambda^2 W^* \to \Lambda^2 W^*$ is ${\rm U}(4)$-equivariant, 
the proof is completed by Schur's lemma. 
\end{proof}

The spaces $[\Lambda^{1,1}_0]$ and $\R \om$ are also irreducible with respect to the standard ${\rm SU}(4)$-action.
As a representation of  ${\rm SU}(4)$, $\lb \Lambda^{2,0} \rb$ further decomposes as follows. 

\begin{lemma} \label{lem:Apm}
We have the following irreducible decomposition with respect to the standard ${\rm SU}(4)$-action: 
$$
\lb \Lambda^{2,0} \rb = A_+ \oplus A_-, 
$$
where 
\begin{align*}
A_+ &= \{ \beta \in \Lambda^2 W^* \mid *({\rm Re} \Om \wedge \beta) = 2 \beta \}, \\ 
A_- &= \{ \beta \in \Lambda^2 W^* \mid *({\rm Re} \Om \wedge \beta) = -2 \beta \}. 
\end{align*}
\end{lemma}
\begin{proof}
The decomposition $\lb \Lambda^{2,0} \rb = A_+ \oplus A_-$ is given by \cite[(3)]{Munoz}. 
Our definition of $A_\pm$ may look different from that of \cite{Munoz}, but 
we see the equivalence as follows. 
As we see below, the ${\rm SU}(4)$-structure induces the $\Sp$-structure $\Phi$ given by \eqref{eq:Phi4 SU4}. 
By \cite[Proposition 2]{Munoz}, $A_+ \subset \Lambda^2_7 W^*$ and 
$A_- \subset \Lambda^2_{21} W^*$. 
Then, for $a_+ \in A_+$, we have 
$$
3 a_+ = * (\Phi \wedge a_+) = \frac{1}{2} *(\om^2 \wedge a_+) + *({\rm Re} \Om \wedge a_+). 
$$
By Lemma \ref{lem:SU4 2form}, we have $*(\om^2 \wedge a_+) = 2 a_+$. 
Hence, we obtain $*({\rm Re} \Om \wedge a_+) = 2 a_+$. 
Similarly, we see $*({\rm Re} \Om \wedge a_-) = -2 a_-$ for  $a_- \in A_-$. 
Since $*({\rm Re} \Om \wedge \cdot): \Lambda^2 W^* \to \Lambda^2 W^*$ is the zero map 
on $[\Lambda^{1,1}]$, the proof is completed. 
\end{proof}

Since ${\rm SU}(4) \subset \Sp$, 
the ${\rm SU}(4)$-structure induces the $\Sp$-structure. 
Explicitly, the 4-form in \eqref{Phi4} is given by 
\begin{align} \label{eq:Phi4 SU4} 
\Phi = \frac{1}{2} \om^2 + {\rm Re} \Om. 
\end{align}
The relation among the irreducible representations of ${\rm SU}(4)$ and $\Sp$ is studied in detail in \cite[Proposition 2]{Munoz}. 
The following decomposition is useful in this paper. 
\begin{align} \label{eq:SU4 id7}
\Lambda^2_7 W^* = \R \om \oplus A_+, \qquad \Lambda^4_7 W^* = \R \Im \Om \oplus (\om \wedge A_-). 
\end{align}

We use the following in the proof of Proposition \ref{prop:SL4 eq pt}. 

\begin{lemma} \label{lem:SU4 4form}
For any 4-form $\xi \in \Lambda^4 W^*$, we have 
$$
2 |\pi_{\om \wedge A_-} (\xi)|^2 = \left |\pi_{A_-} \left(* (\om \wedge \xi) \right) \right|^2. 
$$
\end{lemma}
\begin{proof}
Let $\{ \beta_j \}_{j=1}^6$ be an orthonormal basis of $A_-$. 
Since any element of $\Lambda^4_7 W^*$ is self dual, $\om \wedge \beta_j$ is self dual by \eqref{eq:SU4 id7}. Then, 
by Lemma \ref{lem:SU4 2form}, we have 
$$
\la \om \wedge \beta_j, \om \wedge \beta_k \ra
=
* (\beta_j \wedge \beta_k \wedge \om^2)
=
2 *( \beta_j \wedge * \beta_k) 
= 2 \delta_{j k}. 
$$
Thus, we see that $\{ \om \wedge \beta_j/\sqrt{2} \}_{j=1}^6$ is an orthonormal basis of $\om \wedge A_-$. 
Hence, 
\begin{align*}
2 |\pi_{\om \wedge A_-} (\xi)|^2
&= \sum_j \la \xi, \om \wedge \beta_j \ra^2 \\
&= \sum_j \left(* (\xi \wedge \om \wedge \beta_j) \right )^2
= \sum_j  \la *(\om \wedge \xi), \beta_j \ra^2
= \left |\pi_{A_-} \left(* (\om \wedge \xi) \right) \right|^2. 
\end{align*}
\end{proof}

\section{The volume functional} \label{sec:vol}
In this section, we introduce the volume functional $V$, 
which corresponds to the volume functional for submanifolds 
via the real Fourier--Mukai transform 
and is called the Dirac-Born-Infeld (DBI) action in theoretical physics \cite{MMMS}. 
The relation to the similar volume functional in \cite[Definition 3.1]{JY} for Hermitian metrics on a holomorphic line bundle 
is explained in Remark \ref{rem:JY MC}.  

Let 
$(X,g)$ be a compact oriented $n$-dimensional Riemannian manifold 
and $L \to X$ be a smooth complex line bundle with a Hermitian metric $h$. 
Set 
\[
\begin{aligned}
\mathcal{A}_{0}=\{\, \mbox{Hermitian connections of }(L,h) \,\}
= \nabla + \i \Om^1 \cdot \id_L, 
\end{aligned}
 \]
where $\nabla \in \Aa_{0}$ is any fixed connection. 
We regard the curvature 2-form $F_\n$ of $\n$ as a $\i \R$-valued closed 2-form on $X$. 
For simplicity, set 
$$
\FF_\n := - \i F_\n \in \Om^2. 
$$
Define the \emph{volume functional} $V: \Aa_0 \rightarrow \R$ by 
\[
V(\n) = \int_X \sqrt{\det (\id_{TX} + \FF_\n^\sharp)}\ \vol_g, 
\]
where 
$\FF_\n^\sharp \in \Gamma (X, {\rm End} TX)$ 
is defined by $u\mapsto \left(i(u) \FF_\n \right)^{\sharp}$ 
and $\vol_g$ is the volume form defined by the Riemannian metric $g$. 
Note that $\det (\id_{TX} + \FF_\n^\sharp) \geq 1$ by the argument in Lemma \ref{lem:det}. 
We see that 
the functional $V$ corresponds to the volume functional for submanifolds 
via the real Fourier--Mukai transform by the proof of \cite[Lemma 4.3]{KYFM}.

\begin{remark} \label{rem:det Gn}
It is useful to introduce $G_{\nabla}:TX \to TX$ for $\n \in \Aa_0$ defined by 
\begin{align} \label{eq:Gn}
\begin{split}
G_{\nabla}:=&\id_{TX} - \FF_\n^\sharp\circ \FF_\n^\sharp \\
=& \left(\id_{TX} - \FF_\n^\sharp \right) \circ \left(\id_{TX} + \FF_\n^\sharp \right)
= {}^t\! \left(\id_{TX} + \FF_\n^\sharp \right) \circ \left(\id_{TX} + \FF_\n^\sharp \right), 
\end{split}
\end{align}
where we denote by ${}^t T$ the transpose of a linear map $T:TX \to TX$. 
We see that $G_\n$ is positive definite. 
Since 
$\FF_\n^\sharp \circ G_\n = G_\n \circ \FF_\n^\sharp$, 
we have 
\begin{align} \label{eq:FG comm}
\FF_\n^\sharp \circ G_\n^{-1} = G_\n^{-1}\circ \FF_\n^\sharp, 
\end{align}
which we use frequently. 
We see that $G_\n^{-1}\circ \FF_\n^\sharp$ is skew-symmetric by \eqref{eq:FG comm}.

The volume functional $V$ is rewritten as 
\[
V(\n) = \int_X \left(\det G_\n \right)^{1/4} \ \vol_g. 
\]
In this section, we mainly use this expression for $V$, 
which enables us to deduce the following results. 
Note that setting $K_\n=G_\n^{1/4}:TX \to TX$, we see that $V(\n)=\int_X \vol_{K_\n^* g}$, 
where $\vol_{K_\n^* g}$ is the volume form defined by the metric $K_\n^* g$. 
\end{remark}

We compute the first variation of $V$. 

\begin{proposition} \label{prop:fistvar}
Let $\delta_\n V: T_\n \Aa_0= \i \Om^1 \to \R$ be the linearization of $V:\Aa_0 \to \R$
at $\n \in \Aa_0$. Then, we have 
\[
(\delta_\n V)(\i a) = - \la a, H(\n) \ra_{L^2}. 
\]
for $a \in \Om^1$. Here, $\la \,\cdot\,, \,\cdot\, \ra_{L^2}$ 
is the $L^2$ inner product with respect to the metric $g$ 
and 
$H(\n) \in \Om^1$ is defined by 
\begin{align} \label{eq:MC}
H(\n) =- d^{\ast}\left(\left(\det G_\n \right)^{1/4}\left(G_\n^{-1}\circ \FF_\n^\sharp \right)^{\flat}\right), 
\end{align}
where for a skew-symmetric endomorphism $K$, $K^{\flat}$ is a 2-form 
defined by $K^{\flat}(u, v)=g(K(u), v)$ for $u,v \in TX$. 
We call $H(\n)$ the \emph{mean curvature 1-form} of $\nabla$. 
\end{proposition}

In \cite[Definition 2.3]{JY}, a similar notion of the mean curvature 1-form is defined. 
We will explain the relation in Remark \ref{rem:JY MC}.

\begin{proof}
Set $\n_t = \n +t \i a$ for $t \in \R$. 
Then, 
\[
\begin{aligned}
(\delta_\n V)(\i a) 
=&\frac{d}{dt} \bigg|_{t=0} V(\n_t)\\
=& - \frac{1}{4}\int_X {\rm tr} 
\left(G_\n^{-1}\circ  \left( (da)^\sharp \circ \FF_\n^{\sharp} + \FF_\n^{\sharp} \circ (da)^\sharp \right) \right) \left(\det G_\n\right)^{1/4} \  \vol_g, \\
=&
- \frac{1}{2}\int_X {\rm tr} \left(G_\n^{-1}\circ \FF_\n^{\sharp} \circ (da)^\sharp\right)\left(\det G_\n\right)^{1/4} \  \vol_g,  
\end{aligned}
\]
where we use \eqref{eq:FG comm} and $(da)^\sharp \in \Gamma (X, {\rm End}(TX))$ is defined by 
$g ((da)^\sharp (u),v) = da (u,v)$ for $u,v \in TX$. 
Since ${\rm tr} \left(A^{\sharp} \circ B^\sharp\right)
=-2\left\la A, B \right\ra$ for any 2-forms $A,B$, 
where $\langle\,\cdot\,,\,\cdot\,\rangle$ on the right hand side 
is the inner product for 2-forms, 
we have 
\begin{equation}\label{202010112359-1}
	\begin{aligned}
		{\rm tr} \left(G_\n^{-1}\circ \FF_\n^{\sharp} \circ (da)^\sharp\right)=-2\left\la da,\left(G_\n^{-1}\circ \FF_\n^\sharp \right)^{\flat} \right\ra. 
	\end{aligned}
\end{equation}
Multiplying \eqref{202010112359-1} by $(-1/2) (\det G_\n)^{1/4}$, 
integrating it over $X$ with respect to $\vol_g$ and using the integration by parts, 
we have 
$(\delta_\n V)(a) = - \la a, H(\n) \ra_{L^2}$ with 
\[
H(\n) =- d^{\ast}\left(\left(\det G_\n \right)^{1/4}\left(G_\n^{-1}\circ \FF_\n^\sharp \right)^{\flat}\right), 
\]
and the proof is completed. 
\end{proof}

\begin{proposition}\label{202012122315}
Define $\Ll: \Om^1 \to \Om^1$ by 
$\Ll (a) = \delta_{\nabla}H (\i a)$, where 
$\delta_{\nabla}H$ is the linearization of the operator $H:\Aa_0\to  \Om^1$ at $\nabla \in \Aa_0$. 
Then, the principal symbol $\sigma_{\Ll}(\xi): T_{p}^{*}X \to T_{p}^{*}X$ of $\Ll$ 
satisfies 
\begin{equation}\label{202102031137}
\begin{aligned}
&\left\la\sigma_{\Ll}(\xi)(a),a\right\ra\\
=&\left(\det G_{\nabla}\right)^{1/4}\left(
a\left(G_{\nabla}^{-1}(a^\sharp)\right)\xi\left(G_{\n}^{-1}(\xi^\sharp)\right)
-\left\{a\left(G_{\nabla}^{-1}(\xi^\sharp)\right)\right\}^2
\right)
\end{aligned}
\end{equation}
for $\xi,a\in T_{p}^{*}X$ at each $p\in X$. 
\end{proposition}

\begin{proof}
First, we compute $\Ll$. 
Put $\nabla_{t}:= \nabla+t\i a$. Then, 
\[\Ll (a)
= - \frac{d}{dt}\bigg|_{t=0} d^{\ast}\left(\left(\det G_{\n_{t}}\right)^{1/4}\left(G_{\n_{t}}^{-1}\circ \FF_{\n_{t}}^\sharp \right)^{\flat}\right). 
\]
Since 
$$
0= \frac{d}{dt}\bigg|_{t=0} G_{\n_t} \circ G_{\n_t}^{-1}
=  \left(- (da)^\sharp \circ \FF_{\n}^\sharp - \FF_{\n}^\sharp\circ (da)^\sharp \right) \circ G_{\nabla}^{-1}
+G_\n \circ \frac{d G_{\n_t}^{-1} }{dt}\bigg|_{t=0}, 
$$
we have 
\begin{align} \label{eq:diff Ginv}
\frac{d G_{\n_t}^{-1} }{dt}\bigg|_{t=0} = G_\n^{-1} \circ 
\left( (da)^\sharp \circ \FF_{\n}^\sharp + \FF_{\n}^\sharp\circ (da)^\sharp \right) \circ G_{\nabla}^{-1}. 
\end{align}
Then, we compute 
\begin{align} \label{eq:J1234}
\Ll (a)
= 
J_1(a) + J_2(a) + J_3(a) + J_4(a), 
\end{align}
where 
\begin{align*}
J_1(a) &= \frac{1}{2} d^{\ast}\bigg( \left(\det G_{\nabla}\right)^{1/4} \mathop{\rm tr}\left(G_{\nabla}^{-1}\circ \FF_{\n}^\sharp\circ (da)^\sharp\right)\left(G_{\n}^{-1}\circ \FF_{\n}^\sharp \right)^{\flat} \bigg), \\
J_2(a) &= 
- d^{\ast}\bigg( \left(\det G_{\nabla}\right)^{1/4} 
\left(\left(G_{\nabla}^{-1}\circ (da)^\sharp \circ \FF_{\n}^\sharp\circ G_{\nabla}^{-1}\right) \circ \FF_{\n}^\sharp \right)^\flat \bigg), \\
J_3(a) &=
- d^{\ast}\bigg( \left(\det G_{\nabla}\right)^{1/4} 
\left(\left(G_{\nabla}^{-1}\circ \FF_{\n}^\sharp\circ (da)^\sharp \circ G_{\nabla}^{-1}\right) \circ \FF_{\n}^\sharp \right)^\flat \bigg), \\
J_4(a) &=
- d^{\ast}\bigg( \left(\det G_{\nabla}\right)^{1/4} 
\left(G_{\nabla}^{-1}\circ (da)^\sharp\right)^\flat \bigg). 
\end{align*}
Denote by $\sigma_{J_k}(\xi)$ the principal symbol of the linear differential operator $J_k: \Om^1 \to \Om^1$ ($k=1, \cdots, 4$) 
for $\xi \in T^*_p X$ at $p \in X$. 
Then, for $a \in T^*_p X$, we have 
\begin{align} \label{eq:symb1}
\begin{split}
\frac{\left\la\sigma_{J_{1}}(\xi)(a),a\right\ra}{\left(\det G_{\nabla}\right)^{1/4}}
=&
- \frac{\mathop{\rm tr}\left(G_{\nabla}^{-1}\circ \FF_{\n}^\sharp 
\circ (\xi \wedge a)^\sharp\right) }{2}
\left\la i(\xi^\sharp) 
\left(G_{\n}^{-1}\circ \FF_{\n}^\sharp \right)^{\flat}, 
a\right\ra \\
=&
\left\la \left(G_{\n}^{-1}\circ \FF_{\n}^\sharp \right)^{\flat}, 
\xi \wedge a \right\ra^2 
= 
\left\{a\left(G_{\nabla}^{-1}\circ \FF_{\n}^\sharp(\xi^\sharp)\right)\right\}^2,
\end{split}
\end{align}
where we use the same argument as in \eqref{202010112359-1} for the second equality. 
\\

Next, we work on $J_2$. We have 
$$
\left\la\sigma_{J_2}(\xi)(a),a\right\ra 
= 
\left(\det G_{\nabla}\right)^{1/4} 
\left \la \left(G_{\nabla}^{-1} \circ (\xi \wedge a)^\sharp \circ \FF_{\n}^\sharp \circ G_{\nabla}^{-1} \circ \FF_{\n}^\sharp \right)^\flat, 
\xi \wedge a \right \ra. 
$$
By \eqref{eq:FG comm} and $G_\n=\id_{TX} - \FF_\n^\sharp\circ \FF_\n^\sharp$, we have 
\begin{align} \label{eq:EGE}
\FF_\n^\sharp \circ G_\n^{-1} \circ \FF_\n^\sharp
= G_\n^{-1} \circ \FF^\sharp \circ \FF_\n^\sharp
= - \id_{TX} + G_\n^{-1}. 
\end{align}
This together with $(\xi \wedge a)^\sharp = \xi \otimes a^\sharp- a \otimes \xi^\sharp$ implies that 
\begin{align} \label{eq:symb2}
\begin{split}
&\left(\det G_{\nabla}\right)^{-1/4} \left\la\sigma_{J_2}(\xi)(a),a\right\ra \\
=& 
\left \la - \left(G_{\nabla}^{-1} \circ (\xi \wedge a)^\sharp \right)^\flat
+ \left(G_{\nabla}^{-1} \circ (\xi \wedge a)^\sharp \circ G_\n^{-1} \right)^\flat, \xi \wedge a \right \ra \\
=&
\left \la - \left(G_{\nabla}^{-1} \circ (\xi \wedge a)^\sharp \right) (\xi^\sharp)
+ \left(G_{\nabla}^{-1} \circ (\xi \wedge a)^\sharp \circ G_\n^{-1} \right) (\xi^\sharp), a^\sharp \right \ra \\
=&
- a(G_\n^{-1}(a^\sharp))|\xi|^2 + a(G_\n^{-1}(\xi^\sharp))a(\xi^{\sharp}) 
+ a(G_\n^{-1}(a^\sharp))\xi(G_\n^{-1}(\xi^\sharp)) -\left\{a(G_\n^{-1}(\xi^\sharp))\right\}^2. 
\end{split}
\end{align}

Next, we work on $J_3$. 
By $(\xi \wedge a)^\sharp = \xi \otimes a^\sharp- a \otimes \xi^\sharp$, we compute 
\begin{align} \label{eq:symb3}
\begin{split}
&\left(\det G_{\nabla}\right)^{-1/4} \left\la\sigma_{J_3}(\xi)(a),a\right\ra \\
=& 
\left \la \left(G_{\nabla}^{-1} \circ \FF_{\n}^\sharp \circ (\xi \wedge a)^\sharp \circ G_{\nabla}^{-1} \circ \FF_{\n}^\sharp \right)^\flat, 
\xi \wedge a \right \ra \\
=&
\left \la \left(G_{\nabla}^{-1} \circ \FF_{\n}^\sharp \circ (\xi \wedge a)^\sharp \circ G_{\nabla}^{-1} \circ \FF_{\n}^\sharp \right) 
(\xi^\sharp), a^\sharp \right \ra \\
=&
\left \la \left(G_{\nabla}^{-1} \circ \FF_{\n}^\sharp \right) (a^\sharp), a^\sharp \right \ra
\xi  \left ( \left(G_{\nabla}^{-1} \circ \FF_{\n}^\sharp \right) (\xi^\sharp) \right) \\
&- 
\left \la \left(G_{\nabla}^{-1} \circ \FF_{\n}^\sharp \right) (\xi^\sharp), a^\sharp \right \ra
a  \left ( \left(G_{\nabla}^{-1} \circ \FF_{\n}^\sharp \right) (\xi^\sharp) \right) 
=
- \left\{a\left(G_{\nabla}^{-1}\circ \FF_{\n}^\sharp(\xi^\sharp)\right)\right\}^2,
\end{split}
\end{align}
where use 
$\left \la \left(G_{\nabla}^{-1} \circ \FF_{\n}^\sharp \right) (a^\sharp), a^\sharp \right \ra =0$ 
since $G_{\nabla}^{-1} \circ \FF_{\n}^\sharp$ is skew-symmetric. \\

Finally, we work on $J_4$. We have 
\begin{align} \label{eq:symb4}
\begin{split}
\left(\det G_{\nabla}\right)^{-1/4} \left\la\sigma_{J_4}(\xi)(a),a\right\ra 
=&
\left \la \left(G_\n^{-1} \circ (\xi \wedge a)^\sharp \right)^\flat, \xi \wedge a \right \ra \\
=&
\left \la \left(G_\n^{-1} \circ (\xi \wedge a)^\sharp \right) (\xi^\sharp), a \right \ra \\
=&
a\left(G_\n^{-1}(a^\sharp)\right) |\xi|^2 - a\left(G_\n^{-1} (\xi^\sharp)\right) a(\xi^\sharp). 
\end{split}
\end{align}
Then, by \eqref{eq:J1234}, \eqref{eq:symb1}, \eqref{eq:symb2}, \eqref{eq:symb3} and \eqref{eq:symb4}, 
we obtain the desired formula. 
\end{proof}

Since $G_\n$ is positive definite, a bilinear form defined by 
\[\langle a,b\rangle_{\nabla}:=a\left(G_{\nabla}^{-1}(b^\sharp)\right)\]
for $a,b\in T_{p}^{*}X$ is an inner product. 
Introducing the notation $|a|^2_{\nabla}:=\langle a,a\rangle_{\nabla}$, 
\eqref{202102031137} is represented as 
\begin{equation}\label{202102031155}
\left\la\sigma_{\Ll}(\xi)(a),a\right\ra
=\left(\det G_{\nabla}\right)^{1/4}\left(
|a|^2_{\nabla}|\xi|^2_{\nabla}
-\langle a, \xi \rangle_{\nabla}^2
\right). 
\end{equation}
Then, by the Cauchy--Schwarz inequality, the following proposition is clear. 

\begin{proposition} \label{prop:deg symbol}
The principal symbol $\sigma_{\Ll}(\xi): T_{p}^{*}X \to T_{p}^{*}X$ of $\Ll$ 
has only one zero eigenvalue with eigenvector $\xi$ and the other eigenvalues are positive. 
\end{proposition}

Since $L$ is a line bundle, the curvature is invariant under the addition of closed 1-forms, 
that is, $F_{\n+\i a} = F_\n$ for any $\n \in \Aa_0$ and $a \in Z^1$, where $Z^1$ is the space of closed 1-forms. 
This implies that $V$ is invariant under the addition of closed 1-forms, and hence, 
$H$ degenerates in the direction of $\i Z^1$. 
Proposition \ref{prop:deg symbol} implies that the principal symbol of $H$ degenerates only in this direction, 
just as in the case of the mean curvature for submanifolds. 
The statement for submanifolds is given, for example, in \cite[Chapter 2]{Zhu}. 
This motivates us to prove the short-time existence and uniqueness of the following flow by DeTurck's trick. \\

Fix $T\in(0,\infty]$. Assume that $\{\, \n_{t}\,\}_{t\in[0,T)}$ is a continuous path in $\Aa_0$ and 
smooth on $(0,T)$.  
For such $\{ \n_t \}$, we may define 
the \emph{line bundle mean curvature flow} if 
it satisfies 
\begin{equation}\label{eq:lbmcf}
	\frac{\p}{\p t}\left(\frac{ \n_t}{\i}\right) = H(\n_t). 
\end{equation}
for all $t\in(0,T)$ and call $\n_{0}$ the initial connection of the flow. 
We remark that the left hand side is a real valued 1-form.

Since the functional $V$ corresponds to the volume functional for submanifolds via the real Fourier--Mukai transform, 
the flow \eqref{eq:lbmcf} can be considered as the ``mirror'' of the mean curvature flow for submanifolds. 
Note that this type of flow is introduced in \cite[Definition 3.1]{JY} 
for Hermitian metrics on a holomorphic line bundle to find dHYM metrics.

We prove the short-time existence and uniqueness of the flow \eqref{eq:lbmcf} using DeTurck's trick.
DeTurck's trick is known as an easy way to prove the short-time existence and uniqueness of the Ricci flow, 
and it also works for the mean curvature flow. 
In general, instead of some original (degenerate) flow, a kind of modified flow is studied first and its result is reduced to the original flow. 
In this case, by the similarity between mean curvature flows and line bundle mean curvature flows, 
we can introduce the following modified flow. \\

Fix $\n_0 \in \Aa_0$. 
Instead of \eqref{eq:lbmcf}, we consider 
\begin{equation}\label{eq:lbmcf2}
	\frac{\p}{\p t}\left(\frac{ \n_t}{\i}\right) = H(\n_t)-d\left(\left(\det G_{\n_{0}}\right)^{1/4}d^{\ast}a_{t}\right). 
\end{equation}
We make some remarks. 
The $d^{*}$ in the second term is defined by the fixed background metric $g$ on $X$. 
The connection to define $G_{\n_0}$ 
in the second term is the fixed $\n_{0}$ for all time. 
The form $a_{t}$ in the second term is defined by $a_t:=(\n_t - \n_0)/\i$, in other words, $\n_t=\n_0+\i a_{t}$. 
Thus, $a_t$ is a real valued 1-form on $X$. 

Denote by $\tilde{H}(\n_t)$ the right hand side of \eqref{eq:lbmcf2}. 
Define $\tilde{\Ll}: \Om^1 \to \Om^1$ by $\tilde{\Ll} (a) = \delta_{\n_0} \tilde{H} (\i a)$, where 
$\delta_{\n_0} \tilde{H}$ is the linearization of $\tilde{H}:\Aa_0 \to \Om^1$ at $\n_0 \in \Aa_0$. 
Let $\sigma_{\tilde{\Ll}}(\xi): T_{p}^{*}X \to T_{p}^{*}X$ be the principal symbol of $\tilde{\Ll}$ for $\xi \neq 0$. 

\begin{lemma}\label{202101282305}
All eigenvalues of $\sigma_{\tilde{\Ll}}(\xi)$ are positive. 
\end{lemma}
\begin{proof}
For $a\in\Omega^{1}$, we have 
\[\tilde{\Ll}(a)=\Ll(a)-d\left(\left(\det G_{\n_{0}}\right)^{1/4}d^{\ast} a\right), \]
where we set $\Ll = \delta_{\n_0} H (\i (\cdot))$. 
Then, by \eqref{202102031155}, we have 
\[
\begin{aligned}
\left\la\sigma_{\tilde \Ll}(\xi)(\i a),a\right\ra
=&
\left\la\sigma_\Ll (\xi)(\i a),a\right\ra 
+ \left(\det G_{\n_0}\right)^{1/4} \la \xi \cdot i(\xi^\sharp)a, a \ra \\
=&
\left(\det G_{\n_0}\right)^{1/4}\left(
|a|^2_{\n_0}|\xi|^2_{\n_0}
-\langle a,\xi\rangle_{\n_0}^2
+\langle a, \xi\rangle^2
\right). 
\end{aligned}
\]
Assume that $a\neq 0$. 
By the Cauchy--Schwarz inequality, $|a|^2_{\n_0}|\xi|^2_{\n_0}-\langle a,\xi\rangle_{\n_0}^2\geq 0$. 
Thus, it is clear that $|a|^2_{\n_0}|\xi|^2_{\n_0}
-\langle a,\xi\rangle_{\n_0}^2
+\langle a, \xi\rangle^2$ is nonnegative. 
Suppose that this is zero for some $a\in T_{p}^{\ast} X$. 
Then, we should have  $|a|^2_{\n_0}|\xi|^2_{\n_0}-\langle a,\xi\rangle_{\n_0}^2=0$ and $\langle a, \xi\rangle=0$. 
The first equality implies that $a=\alpha\xi$ for some $\alpha\neq 0$. 
But, this contradicts to the second one, $\langle a, \xi\rangle=0$. 
Thus, we have proved that $|a|^2_{\n_0}|\xi|^2_{\n_0}
-\langle a,\xi\rangle_{\n_0}^2
+\langle a, \xi\rangle^2$ is strictly positive for $a,\xi\neq 0$. 
Thus, the proof is completed. 
\end{proof}

By Lemma \ref{202101282305}, the linearization of \eqref{eq:lbmcf2} at $\nabla_{0} \in \Aa_0$ is strongly parabolic. 
Thus, by the standard theory of partial differential equations, we obtain the following.

\begin{proposition}\label{202101282317}
\, 
\begin{enumerate}
\item
For any $\n_0 \in \Aa_0$, there exist $\eps >0$ and a smooth family of Hermitian connections 
$\{ \n_t \}_{t \in [0, \eps]}$ satisfying \eqref{eq:lbmcf2} and $\n_t|_{t=0}=\n_0$. 

\item
Suppose that 
$\{ \n^1_t \}_{t \in [0,\eps]}$ and $\{ \n^2_t \}_{t \in [0,\eps]}$ satisfy \eqref{eq:lbmcf2} for $\eps >0$. 
If $\n^1_t|_{t=0}=\n^2_t|_{t=0}$, then $\n^1_t=\n^2_t$ for any $t \in [0,\eps]$. 
\end{enumerate}
\end{proposition}

Using Proposition \ref{202101282317}, we show the following. 

\begin{theorem}\label{short-ex}
\, 
\begin{enumerate}
\item
For any $\n_0 \in \Aa_0$, there exist $\eps >0$ and a smooth family of Hermitian connections 
$\{ \n_t \}_{t \in [0,\eps]}$ satisfying \eqref{eq:lbmcf} and $\n_t|_{t=0}=\n_0$. 

\item
Suppose that 
$\{ \n^1_t \}_{t \in [0,\eps]}$ and $\{ \n^2_t \}_{t \in [0,\eps]}$ satisfy \eqref{eq:lbmcf} for $\eps >0$. 
If $\n^1_t|_{t=0}=\n^2_t|_{t=0}$, then $\n^2_t-\n^1_t$ is 
a (pure imaginary-valued) time-dependent exact 1-form for any $t \in [0,\eps]$. 
\end{enumerate}
\end{theorem}
\begin{proof}
First, we prove the existence. 
Let $\tilde{\nabla}_{t}:=\nabla_{0}+\sqrt{-1}a_{t}$ be the unique short-time solution of \eqref{eq:lbmcf2}. 
Its existence is ensured by Proposition \ref{202101282317}. 
Define a time-dependent 1-form $\eta$ by 
\[
\eta_{t}:=\int_{0}^{t}d\left(\left(\det G_{\n_{0}}\right)^{1/4}d^{\ast} a_{s} \right)ds. 
\]
We remark that $d\eta_{t}=0$ for all $t$ and $\eta_{0}=0$. 
Let $\nabla_{t}:=\tilde{\nabla}_{t}+\sqrt{-1}\eta_{t}$. Then, we have $\nabla_{t}|_{t=0}=\nabla_{0}$ and 
$F_{\nabla_{t}}=F_{\tilde{\nabla}_{t}}$ since $d\eta_{t}=0$. 
Moreover,  we have 
\[
\begin{aligned}
\frac{\partial}{\partial t}\left(\frac{\nabla_{t}}{\i}\right)=& \frac{\partial}{\partial t}\left(\frac{\tilde{\nabla}_{t}}{\sqrt{-1}}\right)+\frac{\partial}{\partial t}\eta_{t}\\
=&H(\tilde{\n}_t)-d\left(\left(\det G_{\n_{0}}\right)^{1/4}d^{\ast} a_{t}\right)
+d\left(\left(\det G_{\n_{0}}\right)^{1/4}d^{\ast} a_{t}\right)\\
=& H(\n_t), 
\end{aligned}
\]
where the last equality follows from the fact that $H(\tilde{\n})= H(\nabla)$ for $\tilde{\nabla}$ and $\nabla$ with $F_{\tilde{\n}}=F_{\nabla}$. 
Thus, $\nabla_{t}$ is a solution of the line bundle mean curvature flow with initial connection $\nabla_{0}$. 

Next, we consider the uniqueness of the solution. 
Let $\nabla^{1}_{t}$ and $\nabla^{2}_{t}$ be solutions of the line bundle mean curvature flow with initial connection $\nabla_{0}$. 
Let $f^{i}_t$ ($i=1,2$) be the unique solution of the following parabolic PDE: 
\[\frac{\partial}{\partial t }f^{i}_t
=\left(\det G_{\nabla_{0}}\right)^{1/4}(-d^{\ast}df^{i}_t)-\left(\det G_{\nabla_{0}}\right)^{1/4}d^{\ast}\left(\frac{\nabla^{i}_{t}-\nabla_{0}}{\sqrt{-1}}\right)\]
with the initial condition $f^{i}_{0}=0$ for $i=1,2$. 
The uniqueness follows from the fact that the PDE is strongly parabolic. 
Put $\tilde{\nabla}_{t}^{i}:=\nabla_{t}^{i}+\sqrt{-1}df^{i}_{t}$ for $i=1,2$. 
Then, we have 
\[
\begin{aligned}
\frac{\partial}{\partial t}\left(\frac{\tilde{\nabla}_{t}^{i}}{\sqrt{-1}}\right)=&H(\nabla^{i}_{t})+d\left(\left(\det G_{\nabla_{0}}\right)^{1/4}(-d^{\ast}df^{i}_t)\right)
-d\left(\left(\det G_{\nabla_{0}}\right)^{1/4}d^{\ast}\left(\frac{\nabla^{i}_{t}-\nabla_{0}}{\sqrt{-1}}\right)\right)\\
=&
H(\nabla^{i}_{t}) 
-d\left(\left(\det G_{\nabla_{0}}\right)^{1/4}d^{\ast}\left(\frac{\nabla^{i}_{t}+\sqrt{-1}df^{i}_t-\nabla_{0}}{\sqrt{-1}}\right)\right)\\
=&H(\tilde{\nabla}^{i}_{t})-d\left(\left(\det G_{\nabla_{0}}\right)^{1/4}d^{\ast}\left(\frac{\tilde{\nabla}^{i}_{t}-\nabla_{0}}{\sqrt{-1}}\right)\right), 
\end{aligned}
\]
where we use $H(\nabla^{i}_{t})=H(\tilde{\nabla}^{i}_{t})$. 
Thus, $\tilde{\nabla}_{t}^{i}$ is the solution of \eqref{eq:lbmcf2} with initial condition $\nabla_{0}$. 
Since the solution is unique, we have $\tilde{\nabla}_{t}^{1}=\tilde{\nabla}_{t}^{2}$. 
Thus, we have $\nabla^{1}_{t}=\nabla^{2}_{t}+\sqrt{-1}d(f^{2}_{t}-f^{1}_{t})$ and 
the proof is completed. 
\end{proof}

\section{The ``mirror" of the Cayley equality}  \label{sec:Spin7dDT}
In this section, 
we first recall the definition of deformed Donaldson--Thomas connections 
for a manifold with a ${\rm Spin}(7)$-structure ($\Sp$-dDT connections) 
and its moduli space from \cite{KYSpin7} together with a detailed description of the orbit of the unitary gauge group. 
Then, we show a ``mirror" of the Cayley equality. 
Using this, if the $\Sp$-structure is torsion-free, 
we show that $\Sp$-dDT connections are global minimizers of the volume functional $V$, 
just as Cayley submanifolds are homologically volume-minimizing. 
As an application of this, 
we show that any $\Sp$-dDT connection of a flat line bundle 
over a compact connected $\Sp$-manifold is a flat connection.

\SkipTocEntry \subsection{Preliminaries for the moduli space}
Use the notation (and identities) of Subsection \ref{sec:Spin7 geometry}. 
Let $X^8$ be an 8-manifold with a ${\rm Spin}(7)$-structure $\Phi$
and $L \to X$ be a smooth complex line bundle with a Hermitian metric $h$.
Let $\mathcal{A}_{0}$ be the space of Hermitian connections on $(L,h)$. 
We regard the curvature 2-form $F_\n$ of $\n$ as a $\i \R$-valued closed 2-form on $X$. 
Define maps 
$\Ff^1_{{\rm Spin}(7)}:\Aa_{0} \rightarrow \i \Om^{2}_7$ and 
$\Ff^2_{{\rm Spin}(7)}:\Aa_{0} \rightarrow \Om^{4}_7$ by 
\begin{equation*}
\begin{aligned}
\Ff^{1}_{{\rm Spin}(7)}(\nabla)
= 
\pi^2_{7} \left( F_\nabla + \frac{1}{6} * F_\nabla^3 \right),  \qquad
\Ff^{2}_{{\rm Spin}(7)}(\nabla)
=
\pi^{4}_{7}(F_{\nabla}^2). 
\end{aligned}
\end{equation*}
Each element of 
\begin{align} \label{eq:premod Spin7}
\widehat \Mm_\Sp := (\Ff^{1}_{{\rm Spin}(7)})^{-1}(0) \cap (\Ff^{2}_{{\rm Spin}(7)})^{-1}(0)
\end{align}
is called a \emph{deformed Donaldson--Thomas connection}. 
We call this a $\Sp$-dDT connection for short. 
It is known by \cite{KYSpin7} that 
for $\n \in \Aa_0$ satisfying $* F_\n^4/24 \neq 1$, 
$\Ff^1_{{\rm Spin}(7)} (\n) = 0$
implies 
$\Ff^2_{{\rm Spin}(7)} (\n) = 0$. 

The moduli space of  $\Sp$-dDT connections is defined as follows. 
Let 
$\Gg_U = \{\, f \cdot \id_L \mid f \in C^\infty (X, S^1) \,\}$ 
be the group of unitary gauge transformations of $(L,h)$,  
where we consider $S^1= \{ z \in \C \mid |z|=1 \}$ 
and $C^\infty (X, S^1)$ is the space of smooth maps from $X$ to $S^1$. 
The group $\Gg_U$ acts canonically on $\Aa_0$ 
by $(\lambda, \nabla) \mapsto \lambda^{-1} \circ \nabla \circ \lambda$ 
for $\lambda \in \Gg_U$ and $\n \in \Aa_0$. 
When $\lambda=f \cdot \id_L$, we have 
$\lambda^{-1} \circ \nabla \circ \lambda = \nabla + f^{-1}df \cdot \id_L.$ 
Thus, the $\Gg_U$-orbit through $\n \in \Aa_0$ is given by 
$\n + \Kk_U \cdot \id_L$, where  
\begin{align} \label{eq:Gu orbit Spin7}
\Kk_U = \left \{ f^{-1} d f \in \i \Om^1 \mid f \in C^\infty (X, S^1) \right \}. 
\end{align}
Since the curvature 2-form $F_\nabla$ is invariant under the action of $\Gg_U$, 
the {\em moduli space} $\Mm_\Sp$ of $\Sp$-dDT connections of $(L,h)$ 
is given by 
\begin{align} \label{eq:mod Spin7}
\Mm_\Sp 
:= \widehat \Mm_\Sp/\Gg_U. 
\end{align}
The space $\Kk_U$ is described more explicitly as follows if $X$ is compact and connected. 
It would be well-known for experts, but we give the proof for completeness. 
We use this in Section \ref{sec:hol red}. 
Note that the following holds for any smooth compact connected oriented Riemannian manifold $X$ of any dimension. 

\begin{lemma} \label{lem:orbit}
Suppose that $X$ is compact and connected. 
Then, we have 
\begin{align} \label{eq:orbit coh}
\left \{ [f^{-1} df] \in \i H^1_{dR} (X) \mid f \in C^\infty (X, S^1) \right \} = 2 \pi \i H^1 (X, \Z),  
\end{align} 
where we identify $H^1 (X, \Z)$ with its image in $H^1_{dR} (X)$, 
that is, 
$$
H^1 (X, \Z) = \left \{ [\alpha] \in H^1_{dR} (X) \mid \int_A \alpha \in \Z \mbox{ for any } A \in H_1 (X, \Z) \right \}. 
$$ 
In particular, we have 
\begin{align} \label{eq:KU}
\begin{split}
\Kk_U =& \i \left( 2\pi \Hh^1_\Z (X) \oplus d \Om^0(X) \right) \\
:=& \left \{ \i( 2\pi \alpha_\Z + d f_0) \mid \alpha_\Z \in \Hh^1_\Z (X), f_0 \in \Om^0(X) \right \}, 
\end{split}
\end{align} 
where $\Hh^1_\Z (X)$ is the space of harmonic forms representing $H^1 (X, \Z)$. 
\end{lemma}

\begin{proof}
We first prove \eqref{eq:orbit coh}. 
Define a 1-form $\alpha_{S^1} \in \Om^1(S^1)$ on $S^1$ by 
$$
(\alpha_{S^1})_z = \frac{1}{2 \pi \i} \frac{d z}{z} = \frac{d \theta}{2 \pi} \qquad \mbox{for $z=e^{\i \theta} \in S^1$}. 
$$
Then, $\alpha_{S^1}$ is closed and $[\alpha_{S^1}] \in H^1(S^1, \Z)$. 
Take any $f \in C^\infty (X, S^1)$. 
Since $f^{-1} df =2 \pi \i f^* \alpha_{S^1}$, we see that $[f^{-1} df] \in 2 \pi \i H^1 (X, \Z)$.

Conversely, take any $[\alpha] \in H^1 (X, \Z)$. 
Define a homomorphism $T_{[\alpha]}: \pi_1(X) \to \Z$ by 
$$
T_{[\alpha]} ([g]) = \int_{g_* [S^1]} \alpha
$$
where $g:S^1 \to X$ is a continuous map representing $[g] \in \pi_1(X)$ and $[S^1]$ is the fundamental class of $S^1$.  
Recall that the isomorphism ${\rm deg}: \pi_1(S^1) \to \Z$ is given by 
$$
{\rm deg} ([h]) = \int_{h_* [S^1]} \alpha_{S^1}, 
$$
where $h: S^1 \to S^1$ is a continuous map representing $[h] \in \pi_1(S^1)$. 
Then, by \cite[Proposition 1B.9]{Hatcher}, there is a continuous map $f:X \to S^1$ 
such that $f_*={\rm deg}^{-1} \circ T_{[\alpha]}: \pi_1 (X) \to \pi_1 (S^1)$. 
By \cite[Proposition 17.8]{BT}, we may assume that $f \in C^\infty (X, S^1)$. 
Then, we see that 
$$
T_{[\alpha]} ([g]) = ({\rm deg} \circ f_*) ([g]) \qquad \Longleftrightarrow \qquad 
\int_{g_* [S^1]} \alpha = \int_{f_* g_*[S^1]} \alpha_{S^1} = \int_{g_*[S^1]} f^* \alpha_{S^1} 
$$
for any continuous map $g:S^1 \to X$. Hence, we obtain 
$$
2 \pi \i [\alpha] = 2 \pi \i [f^* \alpha_{S^1}] = [f^{-1} df], 
$$
which implies \eqref{eq:orbit coh}.

Next, we prove \eqref{eq:KU}. By \eqref{eq:orbit coh}, we see that 
$\Kk_U \subset \i \left( 2\pi \Hh^1_\Z (X) \oplus d \Om^0(X) \right)$. 
Conversely, take any $\alpha_\Z \in \Hh^1_\Z (X)$ and  $f_0 \in \Om^0(X)$. 
By \eqref{eq:orbit coh}, there exist $f_\Z \in C^\infty (X, S^1)$ and $g_\Z \in \Om^0(X)$ 
such that $2 \pi \i \alpha_\Z = f_\Z^{-1} df_\Z + \i d g_\Z$. 
Then, $f = e^{\i (g_\Z +f_0)} f_\Z$ satisfies $f^{-1} df = \i( 2\pi \alpha_\Z + d f_0)$ 
and the proof is completed.  
\end{proof}

\SkipTocEntry \subsection{Statements} 
The following statement is considered to be a ``mirror" of 
the Cayley equality \cite[Chapter I\hspace{-.1em}V,Theorem 1.28]{HL}
via the real Fourier--Mukai transform 
as stated in \cite[Lemma 5.5]{KYFM}. 
Theorem \ref{thm:Cayley eq} is the core of Sections \ref{sec:Spin7dDT}-\ref{sec:hol red}. 
The proof is given in Proposition \ref{prop:Cayley eq pt}. 

\begin{theorem} \label{thm:Cayley eq}
For any $\n \in \Aa_0$, we have 
\begin{align*}
&\left( 1+ \frac{1}{2} \la F_\n^2, \Phi \ra + \frac{* F^4_\n}{24}  \right)^2
+
4 \left| \pi^2_7 \left( F_\n + \frac{1}{6} * F_\n^3\right) \right|^2
+
2 \left| \pi^4_7 \left( F_\n^2 \right) \right|^2 \\
=&
\det (\id_{TX} + (-\i F_\n)^\sharp), 
\end{align*}
where $(-\sqrt{-1}F_\n)^\sharp \in \Gamma (X, \mathop{\mathrm{End}} TX)$ is defined by $u\mapsto \left(-\sqrt{-1}i(u)F_{\nabla}\right)^{\sharp}$. 
In particular, 
\[
\left| 1+ \frac{1}{2} \la F_\n^2, \Phi \ra + \frac{* F^4_\n}{24}  \right|
\leq \sqrt{\det (\id_{TX} + (-\i F_\n)^\sharp)}
\]
for any $\n \in \Aa_0$ and the equality holds 
if and only if $\n$ is a $\Sp$-dDT connection. 
\end{theorem}

\begin{remark} \label{rem:novanish}
By Theorem \ref{thm:Cayley eq}, 
it is immediate to see that 
$1+\la F_\n^2, \Phi \ra/2 + * F^4_\n/24$  is nowhere vanishing 
for any $\Sp$-dDT connection $\n$. 
This is also proved in \cite[Remark A.10]{KYSpin7}. 
\end{remark}

By integrating the ``mirror'' of the Cayley equality, 
we can relate the volume functional $V$ to a $\Sp$-dDT connection and 
we see that the volume of a $\Sp$-dDT connection is topological in the torsion-free case.

\begin{theorem} \label{thm:Cayvol ineq}
Suppose that $X$ is compact and connected.  
For any $\n \in \Aa_0$, we have 
\begin{align} \label{eq:Cayley eq}
\left| \int_X \left( 1 + \frac{1}{2} \la F_\n^2, \Phi \ra + \frac{* F^4_\n}{24}  \right) \vol_g \right|
\leq V(\n)
\end{align}
and the equality holds if and only if $\n$ is a $\Sp$-dDT connection. 

If the $\Sp$-structure $\Phi$ is torsion-free, the left hand side of \eqref{eq:Cayley eq} 
is given by
$$
\left| {\rm Vol}(X) +\left( -2 \pi^2 c_1(L)^2 \cup [\Phi]  + \frac{2}{3} \pi^4 c_1(L)^4 \right) 
\cdot [X] \right|, 
$$
where $c_1(L)$ is the first Chern class of $L$. 
In particular, 
for any $\Sp$-dDT connection $\n$, 
$V(\n)$ is topological. 
\end{theorem}

\begin{proof}
By Theorem \ref{thm:Cayley eq}, we have 
\begin{align*}
\left| \int_X \left( 1+\frac{1}{2} \la F_\n^2, \Phi \ra + \frac{* F^4_\n}{24}  \right) \vol_g \right| 
\leq 
\int_X \left| 1+\frac{1}{2} \la F_\n^2, \Phi \ra + \frac{* F^4_\n}{24}  \right| \vol_g
\leq
V(\n). 
\end{align*}
Then, if the equality holds in \eqref{eq:Cayley eq}, 
Theorem \ref{thm:Cayley eq} implies that $\n$ is a $\Sp$-dDT connection. 

Conversely, if  $\n$ is a $\Sp$-dDT connection, we have 
$$
V(\n) = \int_X \left| 1+\frac{1}{2} \la F_\n^2, \Phi \ra + \frac{* F^4_\n}{24}  \right| \vol_g
$$
by Theorem \ref{thm:Cayley eq}. 
By Remark \ref{rem:novanish} and the assumption that $X$ is connected, 
$1+\la F_\n^2, \Phi \ra/2 + * F^4_\n/24$ has constant sign 
for any $\Sp$-dDT connection $\n$. 
Hence, 
the equality holds in \eqref{eq:Cayley eq}. 

If the $\Sp$-structure $\Phi$ is torsion-free, we have 
$$
\left( -2 \pi^2 c_1(L)^2 \cup [\Phi]  + \frac{2}{3} \pi^4 c_1(L)^4 \right) 
\cdot [X] 
=
\int_X \left(\frac{1}{2} \la F_\n^2, \Phi \ra + \frac{* F^4_\n}{24}  \right) \vol_g
$$
by $c_1(L) = \left[ \i F_\n/2 \pi \right]$, 
and the proof is completed. 
\end{proof}

Then, Theorem \ref{thm:Cayvol ineq} implies the following.

\begin{corollary} \label{cor:volmin}
If $X$ is compact and connected and the $\Sp$-structure $\Phi$ is torsion-free, 
any $\Sp$-dDT connection is a global minimizer of $V$. 
\end{corollary}

\begin{proof}
Let $\n$ be a $\Sp$-dDT connection and $\n' \in \Aa_0$ 
be any Hermitian connection. 
Then, by Theorem \ref{thm:Cayvol ineq}, we have 
\[
V(\n) = 
\left| {\rm Vol}(X) +\left( -2 \pi^2 c_1(L)^2 \cup [\Phi]  + \frac{2}{3} \pi^4 c_1(L)^4 \right) 
\cdot [X] \right|
\leq V(\n'), 
\]
and the proof is completed. 
\end{proof}

By Corollary \ref{cor:volmin}, 
we see that 
$H(\n) = 0$ for any $\Sp$-dDT connection $\n$ 
if $X$ is a compact connected $\Sp$-manifold, 
where $H(\n)$ is the mean curvature at $\n$ defined by \eqref{eq:MC}. 

By Theorems \ref{thm:Cayley eq} and \ref{thm:Cayvol ineq}, 
we see the following. 

\begin{corollary}\label{cor:flat Spin7}
Let $(X, \Phi)$ be a compact connected $\Sp$-manifold 
and $L \to X$ be a smooth complex line bundle with a Hermitian metric $h$. 
Then, we have the following. 
\begin{enumerate}

\item 
Suppose that there is a $\Sp$-dDT connection $\n_0$ such that 
$1+ \la F_{\n_0}^2, \Phi \ra/2 + *F_{\n_0}^4/24 >0$ (resp. $<0$). 
Then, for any $\Sp$-dDT connection $\n$, we have 
$1+ \la F_\n^2, \Phi \ra/2 + *F_\n^4/24 >0$ (resp. $<0$). 

\item
Suppose that $L$ is a flat line bundle. 
Then, any $\Sp$-dDT connection is a flat connection. 
In particular, the moduli space of $\Sp$-dDT connections 
is $H^1(X, \R)/2 \pi H^1(X, \Z)$. 
\end{enumerate}
\end{corollary}

\begin{proof}
By the proof of Theorem \ref{thm:Cayvol ineq}, 
$$
\int_X \left( 1 + \frac{1}{2} \la F_\n^2, \Phi \ra + \frac{* F^4_\n}{24}  \right) \vol_g
$$
is independent of $\n \in \Aa_0$. 
By Remark \ref{rem:novanish} and the assumption that $X$ is connected, 
$1+\la F_\n^2, \Phi \ra/2 + * F^4_\n/24$ has constant sign 
for any $\Sp$-dDT connection $\n$. 
Thus, we obtain (1). 

Next, we prove (2). 
By Theorem \ref{thm:Cayvol ineq}, 
$V(\n_0)=V(\n)$ for $\Sp$-dDT connections $\n_0$ and $\n$. 
Let $\n_0$ be a flat connection. 
Then, by Lemma \ref{lem:det}, we have 
$$
\int_X \vol_g
=
\int_X \sqrt{
1+ |F_\n|^2 + \left| \frac{F_\n^2}{2!} \right|^2
+ \left| \frac{F_\n^3}{3!} \right|^2
+ \left| \frac{F_\n^4}{4!} \right|^2
}
\vol_g, 
$$
which implies that $F_\n=0$. 
Hence, we have $\n=\n_0+\i a$ for $a \in Z^1(X)$, where $Z^1(X)$ is the space of closed 1-forms on $X$. 
Then, by Lemma \ref{lem:orbit}, we see that 
$$
\Mm_\Sp \cong (\i Z^1(X))/ \Kk_U \cong H^1(X, \R)/2 \pi H^1(X, \Z). 
$$ 
\end{proof}


\section{The ``mirror" of the associator equality} 
\label{sec:G2dDT}
In this section, we first recall the definition of deformed Donaldson--Thomas connections 
for a manifold with a $G_2$-structure ($G_2$-dDT connections) and its moduli space from \cite{KY}. 
Then, we show a ``mirror" of the associator equality. 
Using this, if the $G_2$-structure is closed, we show that 
$G_2$-dDT connections are global minimizers of the volume functional $V$, 
just as associative submanifolds are homologically volume-minimizing. 
As an application of this, 
we show that any $G_2$-dDT connection of a flat line bundle 
over a compact connected manifold with a closed $G_2$-structure 
is a flat connection.

\SkipTocEntry \subsection{Preliminaries for the moduli space}
First, recall some definitions. 
Let $X^7$ be a 7-manifold with a $G_2$-structure $\varphi \in \Om^3$ 
and $L \to X$ be a smooth complex line bundle with a Hermitian metric $h$. 
Let $\mathcal{A}_{0}$ be the space of Hermitian connections of $(L,h)$. 
We regard the curvature 2-form $F_\n$ of $\n$ as a $\i \R$-valued closed 2-form on $X$. 
Define a map $\Ff_{G_2}: \Aa_{0} \rightarrow \i \Om^{6}$ by 
\begin{equation} \label{eq:deform map G2}
\Ff_{G_2} (\nabla) = \frac{1}{6} F_\nabla^3 + F_\nabla \wedge * \varphi. 
\end{equation} 
Each element of 
\begin{align} \label{eq:premod G2}
\widehat \Mm_{G_2} = \Ff_{G_2}^{-1}(0)
\end{align} 
is called a \emph{deformed Donaldson--Thomas connection}. 
We call this a \emph{$G_2$-dDT connection} for short.
As in Section \ref{sec:Spin7dDT}, 
the group $\Gg_U$ of unitary gauge transformations of $(L,h)$ acts on $\Aa_0$ 
preserving the curvature 2-forms. 
Note that $\Gg_U$-orbits are described explicitly as in Lemma \ref{lem:orbit} 
when $X$ is compact and connected. 
The {\em moduli space} $\Mm_\Sp$ of $G_2$-dDT connections of $(L,h)$ 
is given by 
\begin{align} \label{eq:mod G2}
\Mm_{G_2} := \widehat \Mm_{G_2}/\Gg_U. 
\end{align}

\SkipTocEntry \subsection{Statements} 
The following statement is considered to be a ``mirror" of 
the associator equality \cite[Chapter I\hspace{-.1em}V,Theorem 1.6]{HL}
via the real Fourier--Mukai transform
as stated in \cite[Lemma 4.3]{KYFM}. 
The proof is given in Corollary \ref{cor:asso eq pt}. 

\begin{theorem} \label{thm:asso eq}
For any $\n \in \Aa_0$, we have 
\begin{align*}
&\left( 1+ \frac{1}{2} \la F_\n^2, * \varphi \ra \right)^2
+
\left| * \varphi \wedge F_\n + \frac{1}{6} F_\n^3 \right|^2
+
\frac{1}{4} \left|\varphi \wedge * F_\n^2 \right|^2 \\
=&
\det \left( \id_{TX} + (-\i F_\n)^\sharp \right), 
\end{align*}
where 
$(-\sqrt{-1}F_\n)^\sharp \in \Gamma (X, \mathop{\mathrm{End}} TX)$ is defined by 
$u\mapsto \left(-\sqrt{-1}i(u)F_{\nabla}\right)^{\sharp}$. 
In particular, 
$$
\left| 1+ \frac{1}{2} \la F_\n^2, * \varphi \ra \right| 
\leq \sqrt{\det (\id_{TX} + (-\i F_\n)^\sharp)}
$$ 
for any $\n \in \Aa_0$ and the equality holds 
if and only if $\n$ is a $G_2$-dDT connection. 
\end{theorem}

\begin{proof}
The first equation is proved in Corollary \ref{cor:asso eq pt}. 
By \cite[Remark 3.3]{KY}, 
$\varphi \wedge * F_\n^2=0$ if $\n$ is a $G_2$-dDT connection. 
Then, the last statement holds. 
\end{proof}

\begin{remark} \label{rem:novanish G2}
By Theorem \ref{thm:asso eq}, 
it is immediate to see that 
$1+ \la F_\n^2, * \varphi \ra/2$ is nowhere vanishing 
for any $G_2$-dDT connection $\n$. 
This is also implied by \cite[Section 3.3]{KY}.

We might similarly consider a ``mirror" of the coassociator equality, 
but the resulting equation would be the same as in Theorem \ref{thm:asso eq}. 
It is because the real Fourier--Mukai transform of a coassociative submanifold 
is again a $G_2$-dDT connection, which is obtained by 
the real Fourier--Mukai transform of an associative submanifold. 
See \cite[Propositions 4.1 and 6.1]{KYFM}. 
Since $G_2$-dDT connections behave like associative submanifolds in deformation theory as shown in \cite{KY}, it would be appropriate to call the identity in Theorem \ref{thm:asso eq} 
 a ``mirror" of the associator equality. 
\end{remark}

By Theorem \ref{thm:asso eq}, 
we can prove the following as in Theorem \ref{thm:Cayvol ineq}, 
Corollaries \ref{cor:volmin} and \ref{cor:flat Spin7}. 

\begin{theorem} \label{thm:assovol ineq}
Suppose that $X$ is compact and connected. 
For any $\n \in \Aa_0$, we have 
\begin{align} \label{eq:asso eq}
\left|\int_X \left( 1+ \frac{1}{2} \la F_\n^2, * \varphi \ra \right) \vol_g \right|
\leq V(\n) 
\end{align}
and the equality holds if and only if $\n$ is a $G_2$-dDT connection. 
If the $G_2$-structure $\varphi$ is closed, 
the left hand side of \eqref{eq:asso eq} is given by 
$$
\left| {\rm Vol}(X) + \left( -2 \pi^2 c_1(L)^2 \cup [\varphi]  \right) \cdot [X] \right|, 
$$
where $c_1(L)$ is the first Chern class of $L$. 
In particular, for any $G_2$-dDT connection $\n$, 
$V(\n)$ is topological. 
\end{theorem}

\begin{corollary} \label{cor:volmin G2}
If $X$ is compact and connected and the $G_2$-structure $\varphi$ is closed, 
any $G_2$-dDT connection is a global minimizer of $V$. 
\end{corollary}

\begin{remark}
Karigiannis and  Leung \cite{KL} constructed a Chern-Simons type functional 
whose critical points are $G_2$-dDT connections. 

By Corollary \ref{cor:volmin G2}, 
we see that 
$H(\n) = 0$ for any $G_2$-dDT connection $\n$ 
if $X$ is compact and connected and the $G_2$-structure $\varphi$ is closed, 
where $H(\n)$ is the mean curvature at $\n$ defined by \eqref{eq:MC}. 
The functional $V$ can have critical points other than $G_2$-dDT connections, 
just as associative submanifolds are not only critical points of the (standard) volume functional. 

We give an example of a connection $\n$ that satisfies $H(\n)=0$ but is not a $G_2$-dDT connection 
on a noncompact manifold $X$. (Note that we can define $H(\n)$ on a noncompact manifold by the formula \eqref{eq:MC}.) 
Set $X=\R^3 \times T^4$. Denote by $(x^1, \cdots, x^7)$ the coordinates on $X=\R^3 \times T^4$. 
Set $\n=d+ \i x^1 dx^4$, which is a Hermitian connection of the trivial line bundle $X \times \C \to X$. 
Note that $\n$ is obtained by the real Fourier--Mukai transform 
of $\{ (x^1, x^2, x^3, x^1, 0, 0, 0) \in \R^3 \times T^4 \mid x^j \in \R \}$, 
which corresponds to a plane that is not associative. 
(For more details about the real Fourier--Mukai transform, see \cite{KYFM}.) 
Set $e^j = dx^j, e_j = \p/\p x^j$ and we may assume that $G_2$-structure $\varphi$ is given by \eqref{varphi}. 
Then, since $F_\n= \i e^{14}$, 
we see that 
$$
F_\n^3=0, \qquad F_\n \wedge * \varphi = \i e^{123467} \neq 0,  
$$
which implies that $\n$ is not a $G_2$-dDT connection. 

Next, we show that $H(\n)=0$. 
Recall the notation in Section \ref{sec:vol}. 
Since 
$\FF_\n^\sharp= e^1 \otimes e_4 - e^4 \otimes e_1$, we see that 
$G_{\nabla}:=\id_{TX} + e^1 \otimes e_1 + e^4 \otimes e_4$. 
These equations imply that 
$\det G_\n$ is constant and 
$\left(G_\n^{-1}\circ \FF_\n^\sharp \right)^{\flat}$ is a parallel 2-form. 
Then, by the definition of $H(\n)$ in \eqref{eq:MC}, we see that $H(\n)=0$. 
\end{remark}

\begin{corollary} \label{cor:flat G2}
Let $X^7$ be a compact connected manifold with a $G_2$-structure $\varphi \in \Om^3$ 
with $d \varphi =0$. 
Let $L \to X$ be a smooth complex line bundle with a Hermitian metric $h$. 
Then, we have the following. 
\begin{enumerate}

\item 
Suppose that there is a $G_2$-dDT connection $\n_0$ such that 
$1+ \la F_{\n_0}^2, * \varphi \ra/2 >0$ (resp. $<0$). 
Then, for any $G_2$-dDT connection $\n$, we have 
$1+ \la F_{\n}^2, * \varphi \ra/2 >0$ (resp. $<0$). 

\item
Suppose that $L$ is a flat line bundle. 
Then, any $G_2$-dDT connection is a flat connection. 
In particular, the moduli space of $G_2$-dDT connections 
is $H^1(X, \R)/2 \pi H^1(X, \Z)$. 
\end{enumerate}
\end{corollary}

\begin{remark}
Corollary \ref{cor:flat G2} (2) implies that 
we need to consider a non-flat line bundle to construct non-trivial examples of $G_2$-dDT connections 
on a compact connected 7-manifold with a closed $G_2$-structure. 

Recently, Lotay and Oliveira \cite{LO} constructed 
nontrivial examples of $G_2$-dDT connections on the trivial complex line bundle over a 3-Sasakian 7-manifold. 
This does not contradict Corollary \ref{cor:flat G2} (2) 
since the $G_2$-structure in \cite{LO} is not closed but coclosed. 
\end{remark}

\section{The ``mirror'' of the special Lagrangian equality} \label{sec:SL34}

In this section, we first recall the definition of deformed Hermitian Yang--Mills (dHYM) connections 
for a K\"ahler manifold and its moduli space from \cite{KY}. 
Then, we show a ``mirror" of the special Lagrangian equality 
on a 3 or 4 dimensional K\"ahler manifold. 
We limit ourselves to these dimensions since 
we use the ``mirror" of the Cayley equality in Theorem \ref{thm:Cayley eq} 
and reduce from there. 
Using this ``mirror" of the special Lagrangian equality, we show that 
dHYM connections with phase $e^{\i \theta}$ are global minimizers of the volume functional $V$, 
just as special Lagrangian submanifolds are homologically volume-minimizing. 
As an application of this, 
we show that any dHYM connection with phase 1 of a flat line bundle 
over a compact connected K\"ahler manifold is a flat connection. 
Assuming that a Hermitian connection $\n$ satisfies $F_\n^{0,2}=0$, 
where $F_\n^{0,2}$ is the $(0,2)$-part of the curvature $F_\n$, 
we show that some of these results also hold in any dimension. 
We also relate $H(\n)$ in \eqref{eq:MC} to the angle function and explain the relation to \cite{JY}.

\subsection{Preliminaries for the moduli space} 
First, recall some definitions. 
Let $X$ be an $n$-dimensional K\"ahler manifold with a K\"ahler form $\om$. 
Set $\Lambda^{p,q} = \Lambda^p (T^{1,0} X)^* \otimes \Lambda^q (T^{0, 1} X)^*$. 
Define real vector bundles 
$\lb \Lambda^{p,q} \rb$ for $p \neq q$ and $[\Lambda^{p, p}]$ by 
\[
\begin{aligned}
\lb \Lambda^{p,q} \rb 
=& (\Lambda^{p, q} \oplus \Lambda^{q, p}) \cap \Lambda^{p + q} T^* X
= \{\,\alpha \in \Lambda^{p, q} \oplus \Lambda^{q, p} \mid \bar \alpha = \alpha \,\}, \\
[\Lambda^{p, p}]
=&
\Lambda^{p, p} \cap \Lambda^{2p} T^* X
= \{\, \alpha \in \Lambda^{p, p} \mid \bar \alpha = \alpha \,\}. 
\end{aligned}
\]
Denote by $\lb \Om^{p,q} \rb$ and $[\Om^{p,p}]$
the space of sections of $\lb \Lambda^{p,q} \rb$ and $[\Lambda^{p, p}]$, respectively. 
Denote by $\pi^{\lb p,q \rb}$ the projection $\Om^{p+q} \rightarrow \lb \Om^{p,q} \rb$.

Let $L \to X$ be a smooth complex line bundle with a Hermitian metric $h$. 
Let $\mathcal{A}_{0}$ be the space of Hermitian connections of $(L,h)$. 
We regard the curvature 2-form $F_\n$ of $\n$ as a $\i \R$-valued closed 2-form on $X$. 
Fixing $\theta \in \R$, define a map 
$\Ff = (\Ff_1, \Ff_2): \Aa_0 \rightarrow \lb \Om^{2, 0} \rb \oplus \Om^{2n}$ by 
\begin{equation} \label{eq:deform map dHYM}
\begin{aligned}
\Ff_{dHYM} (\n) =& (\Ff^1_{dHYM} (\n), \Ff^2_{dHYM} (\n))\\
=& \left( \pi^{\lb 2,0 \rb} (-\sqrt{-1}F_\n), \ \mathop{\mathrm{Im}}\left(e^{-\i \theta} (\omega + F_\nabla)^n\right) \right). 
\end{aligned}
\end{equation} 
Each element of 
\begin{align} \label{eq:premod dHYM}
\widehat \Mm_{dHYM} = \Ff_{dHYM}^{-1}(0)
\end{align} 
is called a \emph{deformed Hermitian Yang--Mills (dHYM) connection} with phase $e^{\sqrt{-1}\theta}$. 
As in Section \ref{sec:Spin7dDT}, 
the group $\Gg_U$ of unitary gauge transformations of $(L,h)$ acts on $\Aa_0$ 
preserving the curvature 2-forms. 
Note that $\Gg_U$-orbits are described explicitly as in Lemma \ref{lem:orbit} 
when $X$ is compact and connected. 
The {\em moduli space} $\Mm_{dHYM}$ of dHYM connections of $(L,h)$ 
is given by 
\begin{align} \label{eq:mod dHYM}
\Mm_{dHYM} := \widehat \Mm_{dHYM}/\Gg_U. 
\end{align}
By \cite[Theorem 1.2 (1)]{KY}, $\Mm_{dHYM}$ is an empty set or a $b^1$-dimensional torus, 
where $b^1$ is the first Betti number of $X$.

\subsection{Statements} 

If $\dim_\C X=3$, we have the following. 
The proof is given in Corollary \ref{cor:SL3 eq pt}. 

\begin{theorem}\label{thm:SL3 eq}
For any $\n \in \Aa_0$, we have 
\begin{align*}
\left| \frac{(\om+ F_\n)^3}{3!} \right|^2
+ 2 \left| \pi^{\lb 3,1 \rb} \left( \frac{(\om+ F_\n)^2}{2!} \right) \right|^2 
=
\det (\id_{TX} + (-\i F_\n)^\sharp), 
\end{align*}
where 
$(-\sqrt{-1}F_\n)^\sharp \in \Gamma (X, \mathop{\mathrm{End}} TX)$ 
is defined by $u\mapsto \left(-\sqrt{-1}i(u)F_{\nabla}\right)^{\sharp}$. 
In particular, 
$$
\left| {\rm Re} \left( e^{-\i \theta} \frac{(\om+ F_\n)^3}{3!} \right) \right| 
\leq \sqrt{\det (\id_{TX} + (-\i F_\n)^\sharp)}
$$ 
for any $\theta \in \R$ and $\n \in \Aa_0$. 
The equality holds 
if and only if $\n$ is a dHYM connection with phase $e^{\i \theta}$. 
\end{theorem}

If $\dim_\C X=4$, we have the following. 
The proof is given in Proposition \ref{prop:SL4 eq pt}. 

\begin{theorem}\label{thm:SL4 eq}
For any $\n \in \Aa_0$, we have 
\begin{align*}
&\left| \frac{(\om+ F_\n)^4}{4!} \right|^2
+ 2 \left| \pi^{\lb 4,2 \rb} \left( \frac{(\om+ F_\n)^3}{3!} \right) \right|^2 
+ 8 \left| \pi^{\lb 4,0 \rb} \left( \frac{(\om+ F_\n)^2}{2!} \right) \right|^2 \\
=&
\det (\id_{TX} + (-\i F_\n)^\sharp), 
\end{align*}
where 
$(-\sqrt{-1}F_\n)^\sharp \in \Gamma (X, \mathop{\mathrm{End}} TX)$ 
is defined by $u\mapsto \left(-\sqrt{-1}i(u)F_{\nabla}\right)^{\sharp}$. 
In particular, 
$$
\left| {\rm Re} \left( e^{-\i \theta} \frac{(\om+ F_\n)^4}{4!} \right) \right| 
\leq \sqrt{\det (\id_{TX} + (-\i F_\n)^\sharp)}
$$ 
for any  $\theta \in \R$ and $\n \in \Aa_0$. 
The equality holds 
if and only if $\n$ is a dHYM connection with phase $e^{\i \theta}$. 
\end{theorem}

\begin{remark} \label{rem:novanish SL34}
By Theorems \ref{thm:SL3 eq} and \ref{thm:SL4 eq}, 
it is immediate to see that for $n=3,4$, 
${\rm Re} \left( e^{-\i \theta} (\om+ F_\n)^n/n! \right)$ is nowhere vanishing 
for any $\theta \in \R$ and 
dHYM connection $\n$ with phase $e^{\i \theta}$. 
\end{remark}

By Theorems \ref{thm:SL3 eq} and \ref{thm:SL4 eq}, 
we can prove the following as in Theorem \ref{thm:Cayvol ineq}, 
Corollaries \ref{cor:volmin} and \ref{cor:flat Spin7}. 

\begin{theorem} \label{thm:SL34vol ineq}
Let $X$ be a compact connected $n$-dimensional K\"ahler manifold for $n=3,4$. 
For any $\theta \in \R$ and $\n \in \Aa_0$, we have 
\begin{align*}
\left| {\rm Re} \left( e^{-\i \theta} \frac{([\om] - 2 \pi \i c_1(L))^n}{n!} \right)  \cdot [X] \right|
=
\left|\int_X {\rm Re} \left( e^{-\i \theta} \frac{(\om+ F_\n)^n}{n!} \right) \right|
\leq V(\n),  
\end{align*}
where $c_1(L)$ is the first Chern class of $L$. 
The equality holds if and only if $\n$ is a dHYM connection with phase $e^{\i \theta}$. 
In particular, for any dHYM connection $\n$ with phase $e^{\i \theta}$, $V(\n)$ is topological. 
\end{theorem}

\begin{corollary} \label{cor:volmin SL34}
If $X$ is a compact connected $n$-dimensional K\"ahler manifold for $n=3,4$, 
any dHYM connection with phase $e^{\i \theta}$ is a global minimizer of $V$. 
\end{corollary}

\begin{corollary} \label{cor:flat SL34}
Let $X$ be a compact connected $n$-dimensional K\"ahler manifold for $n=3,4$. 
Let $L \to X$ be a smooth complex line bundle with a Hermitian metric $h$. 
Then, we have the following. 
\begin{enumerate}
\item 
Suppose that there is a dHYM connection $\n_0$ with phase $e^{\i \theta}$ 
such that 
${\rm Re} \left( e^{-\i \theta} * (\om+ F_{\n_0})^n/n! \right) >0$ (resp. $<0$). 
Then, for any dHYM connection $\n$ with phase $e^{\i \theta}$, we have 
${\rm Re} \left( e^{-\i \theta} * (\om+ F_\n)^n/n! \right) >0$ (resp. $<0$). 

\item
Suppose that $L$ is a flat line bundle. 
Then, any dHYM connection with phase 1 is a flat connection. 
In particular, the moduli space of dHYM connections with phase 1 is $H^1(X, \R)/2 \pi H^1(X, \Z)$. 
\end{enumerate}
\end{corollary}

\begin{remark}
A flat connection $\n$ satisfies 
$$
\mathop{\mathrm{Im}}\left(e^{-\i \theta} (\omega + F_\nabla)^n\right) 
= 
\mathop{\mathrm{Im}}\left(e^{-\i \theta} \omega^n\right)
= 
- \sin \theta \cdot \om^n.  
$$
Thus, a flat connection $\n$ is dHYM if and only if $\sin \theta =0$. 
Then, we see that 
any dHYM connection with phase $\pm 1$ is a flat connection. 
Since the set of dHYM connections with phase 1 is the same as  that with phase $-1$, we only state the phase 1 case in Corollary \ref{cor:flat SL34} (2). 
Though Corollary \ref{cor:flat SL34} (2) is a special case of \cite[Theorem 1.2 (1)]{KY}, 
we put this to emphasize the correspondence with Corollary \ref{cor:flat G2} (2).

Theorems \ref{thm:SL3 eq} and \ref{thm:SL4 eq}
are considered to be a ``mirror" of 
the special Lagrangian equality \cite[p.90, Remark]{HL}. 
We can also deduce these identities on a trivial torus bundle 
via the real Fourier--Mukai transform as in \cite[Lemmas 4.3 and 5.5]{KYFM}. 
For general dimensions, we conjecture the following 
from the special Lagrangian equality \cite[p.90, Remark]{HL} and Theorems \ref{thm:SL3 eq} and \ref{thm:SL4 eq}. 
\begin{enumerate}
\item
When $\dim_\C X=2p+1$, 
$\det (\id_{TX} + (-\i F_\n)^\sharp)$ will be described by a linear combination of 
$\left| \pi^{\lb 2p+1,2(p-k)+1 \rb} \left( (\om+ F_\n)^{p-k}/(p-k)! \right) \right|^2$ 
for $0 \leq k \leq p$. 

\item
When $\dim_\C X=2p$, 
$\det (\id_{TX} + (-\i F_\n)^\sharp)$ will be described by a linear combination of 
$\left| \pi^{\lb 2p,2(p-k) \rb} 
\left( (\om+ F_\n)^{p-k}/(p-k)! \right) \right|^2$
for $0 \leq k \leq p$. 
\end{enumerate}
We did not try to prove them because the proof will require a large amount of computations and 
we could not find an interesting application, 
while there are applications for 3 and 4 dimensional cases in Section \ref{sec:hol red}. 
However, if we assume from the beginning that 
the $(0,2)$-part $F_\n^{0,2}$ of $F_\n$ vanishes, which is equivalent to 
$\pi^{\lb 2,0 \rb}(F_\n)=0$, 
computations become easier considerably and we obtain the similar results as we see in the next subsection. 
\end{remark}

\subsection{The case $F_\n^{0,2}=0$} \label{subsec:F02 0}

In this subsection, we consider Hermitian connections $\n \in \Aa_0$ with $F_\n^{0,2}=0$. 
The condition $F_\n^{0,2}=0$ will correspond to the Lagrangian condition via mirror symmetry. 
We first recall some definitions and identities from \cite[Section 4.1]{KY}. 
Set
\begin{align} \label{eq:A}
\Aa = \{ \n \in \Aa_0 \mid F_\n^{0,2}=0 \}
      = \n_0 + \i (Z^1 \oplus d_c \Om^0) \cdot \id_L, 
\end{align}
where $\n_0 \in \Aa$ is any fixed connection, $Z^1$ is the space of closed 1-forms on $X$ 
and $d_c$ is the complex differential. 
The second equation follows from \cite[Lemma A.2]{KY}. 

For $\n \in \Aa$, define a Hermitian metric $\eta_\n$ and its associated 2-form $\om_\n$ by 
$$
\eta_\n =  (\id_{TX} + (-\i F_\n)^\sharp)^* g, \qquad 
\om_\n =  (\id_{TX} + (-\i F_\n)^\sharp)^* \om
$$
and define $\zeta_\n:X \to \C$ by 
$$
(\om+F_\n)^n=\zeta_\n \om^n. 
$$
By \cite[Lemma 4.4]{KY}, we have $|\zeta_\n| \geq 1$. 
Then, there exist $r_\n:X \to [1, \infty)$ and $\theta_\n:X \to \R/2 \pi \Z$ such that 
\begin{align} \label{eq:r theta}
\zeta_\n =r_\n e^{\i \theta_\n}. 
\end{align}

We first prove the following. 

\begin{proposition} \label{prop:SLn eq}
Let $(X, \om, g, J)$ be a K\"ahler manifold with $\dim_\C X=n$ and 
$L \to X$ be a smooth complex line bundle with a Hermitian metric $h$. 
Then, for any Hermitian connection $\n \in \Aa$, 
we have 
\begin{align*}
\left| \frac{(\om+ F_\n)^n}{n!} \right|^2 = \det (\id_{TX} + (-\i F_\n)^\sharp). 
\end{align*}
\end{proposition}

\begin{proof}
We prove this using computations in \cite[Section 4.1]{KY}. 
Since the statement is pointwise, we fix a point $p\in X$ arbitrarily. 
Then, there exists an orthonormal basis 
$u_{1},\dots,u_{n}, v_{1}\dots,v_{n}$ satisfying $v_{i}=Ju_{i}$ and $\lambda_{1},\dots,\lambda_{n} \in \R$ such that 
$(-\sqrt{-1}F_\n)^\sharp
=\sum_{i=1}^{n} \lambda_{i} (u^i \otimes v_i-v^i \otimes u_i )$ 
at $p$, where $\{\, u^{i},v^{i} \,\}_{i=1}^{n}$ is the dual basis of $\{\, u_{i},v_{i} \,\}_{i=1}^{n}$. 
Then, by \cite[(4.6)]{KY}, we have 
$$
\left| \frac{(\om+ F_\n)^n}{n!} \right|^2
=
r_\n^2
= \prod_{i=1}^{n} (1+\lambda_{i}^2)
= \det  (\id_{TX} + (-\i F_\n)^\sharp). 
$$
\end{proof}

Theorems \ref{thm:SL3 eq} and \ref{thm:SL4 eq} imply Proposition \ref{prop:SLn eq} if $\dim_\C X=3$ or $4$. 
Indeed, if $\n \in \Aa$, we have $F_\n \in \Om^{1,1}$. 
This implies that $\om + F_\n \in \Om^{1,1}$, and hence, 
$\pi^{\lb 3,1 \rb} \left( (\om+ F_\n)^2 \right) =0$ 
and $\pi^{\lb 4,2 \rb} \left( (\om+ F_\n)^3 \right) = 
\pi^{\lb 4,0 \rb} \left( (\om+ F_\n)^2 \right) = 0$.

\begin{remark} \label{rem:diff JY}
Proposition \ref{prop:SLn eq} is essentially proved in \cite[p.875]{JY}. 
We explain the minor difference from \cite{JY}. 
\begin{enumerate}
\item
The operator $-\i K: T^{1,0} X \to T^{1,0} X$ in \cite{JY} is given by 
extending $(-\i F_\n)^\sharp$ complex linearly and restricting it to $T^{1,0} X$. 
Since $F_\n^{0,2}=0$, $(-\i F_\n)^\sharp$ commutes with $J$, and hence, $K (T^{1,0} X) \subset T^{1,0} X$. 
Then, we see that 
\begin{align*}
\mathrm{det}_\R (\id_{TX} + (-\i F_\n)^\sharp)
&= \mathrm{det}_\C \left(\id_{TX} + (-\i F_\n)^\sharp: TX \otimes \C \to TX \otimes \C \right) \\
&= \mathrm{det}_\C (\id_{T^{1,0} X} - \i K) \cdot \overline{\mathrm{det}_\C (\id_{T^{1,0} X} - \i K)} \\ 
&= |\mathrm{det}_\C (\id_{T^{1,0} X} - \i K)|^2. 
\end{align*}

\item
In \cite{JY}, $\n \in \Aa_0$ is called a dHYM connection with phase $e^{\i \theta}$ 
if $F_\n^{0,2}=0$ and $\mathop{\mathrm{Im}}\left(e^{-\i \theta} (\omega - F_\nabla)^n\right) = 0$, 
which is different from \eqref{eq:deform map dHYM} by sign of $F_\n$. 
This causes the difference in the sign of $K$ from \cite{JY}. 
In particular, the angle function $\theta$ in \cite[(2.4)]{JY} corresponds to $- \theta_\n$ in our paper. 
\end{enumerate}

By Proposition \ref{prop:SLn eq}, 
we see that the following statements hold for compact and connected K\"ahler manifolds $X$ of any dimension. 
\begin{enumerate}
\item
For any $\theta \in \R$ and $\n \in \Aa$, we have 
\begin{align*}
\left| {\rm Re} \left( e^{-\i \theta} \frac{([\om] - 2 \pi \i c_1(L))^n}{n!} \right)  \cdot [X] \right|
=
\left|\int_X {\rm Re} \left( e^{-\i \theta} \frac{(\om+ F_\n)^n}{n!} \right) \right|
\leq V(\n).  
\end{align*}
The equality holds if and only if $\n$ is a dHYM connection with phase $e^{\i \theta}$. 

\item
For any dHYM connection $\n$ with phase $e^{\i \theta}$, $V(\n)$ is topological. 

\item
Any dHYM connection with phase $e^{\i \theta}$ is a minimizer of $V|_\Aa :\Aa \to \R$. 

\item
The statement of Corollary \ref{cor:flat SL34} holds for any dimension of $X$. 
\end{enumerate}

We can regard Corollary \ref{cor:volmin SL34} as a generalization of (3) above. 
That is, any dHYM connection with phase $e^{\i \theta}$ minimizes $V$ not only on $\Aa$ but also on $\Aa_0$ 
in the case of dimension 3 or 4. 
\end{remark}

Next, we relate our mean curvature $H_\n$ in \eqref{eq:MC} to $d \theta_\n$. 
They are expected to be related closely as in the case of Lagrangian submanifolds. 
Indeed, Jacob and Yau \cite{JY} considers 
the space of Chern connections of Hermitian metrics on a holomorphic line bundle 
and define the mean curvature by the exterior derivative of the angle function.

\begin{theorem} \label{thm:Dazord} 
Suppose that $X$ is compact and connected K\"ahler manifold. 
For any $\n \in \Aa$, we have 
$$
H(\n)
= - (\det G_\n)^{1/4} (G_\n^{-1})^* (J d \theta_\n),  
$$
where $G_\n$ is defined by \eqref{eq:Gn}, $J$ is the complex structure on $X$ 
and we set $J d \theta_\n = d \theta_\n (J (\cdot))$. 

In particular, 
for $\n \in \Aa$, 
$H(\n)=0$ if and only if $\n$ is a dHYM connection with phase $e^{\i \theta}$ for some $\theta \in \R$. 
\end{theorem}

\begin{proof}
In the proof, we use Greek indices $\mu, \nu, \xi$ to run over $1,\cdots, 2n$ 
and Latin indices $i, j$ range over $1,\cdots, n$. 
Use the notation of Section \ref{sec:vol}. 
We denote by $D$ the Levi-Civita connection with respect to $g$ to distinguish from elements in $\Aa_0$. 
Fix $p \in X$. Take an orthonormal frame $\{ e_\mu \}_{\mu=1}^{2n}$ of $TX$ around $p$ with respect to $g$ such that 
$$
D_{e_\mu} e_\nu =0, \qquad e_{2 i} = J (e_{2i-1}), \qquad 
\FF_\n = \sum_{i=1}^n \lambda_i e^{2i-1} \wedge e^{2i} 
$$
for $\lambda_i \in \R$ at $p \in X$. Denote by $\{ e^\mu \}_{\mu=1}^{2n}$ the dual. 
Then, we see that at $p \in X$ 
\begin{align}\label{eq:Daz 1}
\begin{split}
\FF_\n^\sharp &= \sum_{i=1}^n \lambda_i (e^{2i-1} \otimes e_{2i} - e^{2i} \otimes e_{2i-1}), \\
G_\n &= \sum_{i=1}^n (1+ \lambda_i^2) (e^{2i-1} \otimes e_{2i-1} + e^{2i} \otimes e_{2i}), \\
G_\n^{-1} &= \sum_{i=1}^n \frac{1}{1+ \lambda_i^2} (e^{2i-1} \otimes e_{2i-1} + e^{2i} \otimes e_{2i}). 
\end{split}
\end{align}
Set 
$$
D \FF_\n = \frac{1}{2} \sum_{\mu, \nu, \xi=1}^{2n} F_{\mu \nu; \xi} e^\xi \otimes e^{\mu} \wedge e^{\nu}
$$
with $F_{\mu \nu; \xi} = - F_{\nu \mu; \xi}$. 
Note that $F_{\mu \nu; \xi} \in \R$. 
For simplicity, set $D_\xi \FF_\n = D_{e_\xi} \FF_\n$ and $D_\xi \FF_\n^\sharp = D_{e_\xi} \FF_\n^\sharp$. 
Then, 
\begin{align}\label{eq:Daz 2}
D_\xi \FF_\n^\sharp = \sum_{\mu, \nu=1}^{2n} F_{\mu \nu; \xi} e^{\mu} \otimes e_{\nu}. 
\end{align}
Since $D_\xi \FF_\n^\sharp \in \Om^{1,1}$, we have at $p \in X$ 
\begin{align}\label{eq:Daz 3}
F_{2i-1,2j-1;\xi} = F_{2i,2j;\xi}, \qquad F_{2i-1,2j;\xi} = - F_{2i,2j-1;\xi}. 
\end{align}
By the Bianchi identity, we have 
$0=\sum_{\xi=1}^{2n} e^\xi \wedge D_\xi \FF_\n 
= \sum_{\mu, \nu, \xi=1}^{2n} (F_{\mu \nu; \xi}/2) e^\mu \wedge e^\nu \wedge e^\xi$, and hence,  
\begin{align}\label{eq:Daz 4}
F_{\mu \nu; \xi} + F_{\nu \xi; \mu} + F_{\xi \mu; \nu} = 0. 
\end{align}
Using these equations, we compute $H(\n)$ at $p \in X$. 
By \eqref{eq:MC}, we have 
\begin{align} \label{eq:MC02 1}
\begin{split}
H(\n) 
&= - d^{\ast}\left(\left(\det G_\n \right)^{1/4}\left(G_\n^{-1}\circ \FF_\n^\sharp \right)^{\flat}\right) \\
&= \sum_{\xi=1}^{2n} e_\xi \left(\det G_\n \right)^{1/4} \cdot i(e_\xi) \left(G_\n^{-1}\circ \FF_\n^\sharp \right)^{\flat} 
+ \left(\det G_\n \right)^{1/4} i(e_\xi) D_\xi \left(G_\n^{-1}\circ \FF_\n^\sharp \right)^{\flat}. 
\end{split}
\end{align}
At $p \in X$, we have 
\begin{align*}
e_\xi \left(\det G_\n \right)^{1/4}
=& 
\frac{\left(\det G_\n \right)^{1/4}}{4} {\rm tr} \left(G_\n^{-1}\circ D_\xi G_\n \right) \\
=& 
\frac{\left(\det G_\n \right)^{1/4}}{4} {\rm tr} \left(G_\n^{-1}\circ 
\left( - (D_\xi \FF_\n^\sharp) \circ \FF_\n^\sharp - \FF_\n^\sharp \circ (D_\xi \FF_\n^\sharp) \right) \right) \\
=& 
- \frac{\left(\det G_\n \right)^{1/4}}{2} {\rm tr} \left( (D_\xi \FF_\n^\sharp) \circ G_\n^{-1} \circ E_\n^\sharp \right), 
\end{align*}
where we use \eqref{eq:FG comm}. 
Substituting \eqref{eq:Daz 1} and \eqref{eq:Daz 2} into this, it follows that 
\begin{align*}
e_\xi \left(\det G_\n \right)^{1/4} 
=&
- \frac{\left(\det G_\n \right)^{1/4}}{2} 
\sum_{i=1}^{n}
\left \{ g \left( \left( (D_\xi \FF_\n^\sharp) \circ G_\n^{-1} \circ E_\n^\sharp \right)(e_{2i-1}), e_{2i-1} \right) \right. \\
&\left. + g \left( \left( (D_\xi \FF_\n^\sharp) \circ G_\n^{-1} \circ E_\n^\sharp \right)(e_{2i}), e_{2i} \right)
\right \} \\
=&
\left(\det G_\n \right)^{1/4} \sum_{i=1}^{n} \frac{\lambda_i}{1 + \lambda_i^2} F_{2i-1,2i;\xi}. 
\end{align*}
Hence, we obtain at $p \in X$
\begin{align} \label{eq:MC02 11}
\begin{split}
&\sum_{\xi=1}^{2n} 
e_\xi \left(\det G_\n \right)^{1/4} \cdot i(e_\xi) \left(G_\n^{-1}\circ \FF_\n^\sharp \right)^{\flat} \\
=&
\sum_{j=1}^{n} e_{2j-1} \left(\det G_\n \right)^{1/4} \cdot \left( \left(G_\n^{-1}\circ \FF_\n^\sharp \right) (e_{2j-1}) \right)^{\flat}
+
e_{2j} \left(\det G_\n \right)^{1/4} \cdot \left( \left(G_\n^{-1}\circ \FF_\n^\sharp \right) (e_{2j}) \right)^{\flat} \\
=&
\left(\det G_\n \right)^{1/4} 
\sum_{i,j=1}^{n} 
\frac{\lambda_i \lambda_j}{(1+\lambda_i^2) (1+\lambda_j^2)} 
\left( F_{2i-1, 2i; 2j-1} e^{2j} - F_{2i-1, 2i; 2j} e^{2j-1} \right). \\
\end{split}
\end{align}

Next, we compute 
$$
\sum_{\xi=1}^{2n} i(e_\xi) D_\xi \left(G_\n^{-1}\circ \FF_\n^\sharp \right)^{\flat} = 
\sum_{\xi=1}^{2n} i(e_\xi) \left( (D_\xi G_\n^{-1}) \circ \FF_\n^\sharp + G_\n^{-1}\circ (D_\xi \FF_\n^\sharp) \right)^{\flat}. 
$$
By the same computation as in \eqref{eq:diff Ginv}, we have 
$$
D_\xi G_\n^{-1} = 
G_\n^{-1} \circ 
\left( (D_\xi \FF_\n^\sharp) \circ \FF_{\n}^\sharp + \FF_{\n}^\sharp \circ (D_\xi \FF_\n^\sharp) \right) \circ G_{\nabla}^{-1}. 
$$
Then, by \eqref{eq:EGE} and \eqref{eq:FG comm}, it follows that 
\begin{align*}
&i(e_\xi) D_\xi \left(G_\n^{-1}\circ \FF_\n^\sharp \right)^{\flat} \\
=& 
\left( \left( G_\n^{-1} \circ (D_\xi \FF_\n^\sharp) \circ G_\n^{-1} 
+ \FF_\n^\sharp \circ G_\n^{-1} \circ (D_\xi \FF_\n^\sharp) \circ G_\n^{-1} \circ \FF_\n^\sharp 
\right) (e_\xi) \right)^\flat. 
\end{align*}
By \eqref{eq:Daz 1} and \eqref{eq:Daz 2}, we have at $p \in X$ 
\begin{align*}
\left( G_\n^{-1} \circ (D_\xi \FF_\n^\sharp) \circ G_\n^{-1} \right) (e_{2i-1})
=&
\sum_{j=1}^n 
\frac{F_{2i-1,2j-1;\xi} e_{2j-1} + F_{2i-1,2j;\xi} e_{2j}}{(1+ \lambda_i^2) (1+ \lambda_j^2)}, \\
\left( G_\n^{-1} \circ (D_\xi \FF_\n^\sharp) \circ G_\n^{-1} \right) (e_{2i})
=&
\sum_{j=1}^n \frac{F_{2i,2j-1;\xi} e_{2j-1} + F_{2i,2j;\xi} e_{2j}}{(1+ \lambda_i^2) (1+ \lambda_j^2)}, 
\end{align*}
which implies that 
\begin{align*}
\left ( \FF_\n^\sharp \circ G_\n^{-1} \circ (D_\xi \FF_\n^\sharp) \circ G_\n^{-1} \circ \FF_\n^\sharp 
\right) (e_{2i-1})
&= \sum_{j=1}^n \frac{\lambda_i \lambda_j \left( F_{2i,2j-1;\xi} e_{2j} - F_{2i,2j;\xi} e_{2j-1} \right)}{(1+ \lambda_i^2) (1+ \lambda_j^2)}, \\
\left ( \FF_\n^\sharp \circ G_\n^{-1} \circ (D_\xi \FF_\n^\sharp) \circ G_\n^{-1} \circ \FF_\n^\sharp 
\right) (e_{2i})
&= \sum_{j=1}^n \frac{- \lambda_i \lambda_j \left( F_{2i-1,2j-1;\xi} e_{2j} - F_{2i-1,2j;\xi} e_{2j-1} \right)}{(1+ \lambda_i^2) (1+ \lambda_j^2)}.  
\end{align*}
Hence, it follows that at $p \in X$ 
\begin{align*} 
&\sum_{\xi=1}^{2n} i(e_\xi) D_\xi \left(G_\n^{-1}\circ \FF_\n^\sharp \right)^{\flat} \\
=&
\sum_{i, j=1}^n \frac{ F_{2i-1, 2j-1 ;2i-1} e^{2j-1} + F_{2i-1, 2j; 2i-1} e^{2j} + F_{2i, 2j-1 ;2i} e^{2j-1} + F_{2i, 2j; 2i} e^{2j}}
{(1+ \lambda_i^2) (1+ \lambda_j^2)} \\
&+
\sum_{i, j=1}^n \frac{ \lambda_i \lambda_j 
\left( F_{2i, 2j-1 ;2i-1} e^{2j} - F_{2i, 2j; 2i-1} e^{2j-1} -(F_{2i-1, 2j-1 ;2i} e^{2j} - F_{2i-1, 2j; 2i} e^{2j-1}) \right)}
{(1+ \lambda_i^2) (1+ \lambda_j^2)}. 
\end{align*}
By \eqref{eq:Daz 3} and \eqref{eq:Daz 4}, we have 
\begin{align*}
F_{2i-1, 2j-1 ;2i-1} + F_{2i, 2j-1 ;2i} 
=& 
F_{2i, 2j ;2i-1} - F_{2i-1, 2j ;2i} = -F_{2i-1, 2i; 2j}, \\
F_{2i-1, 2j; 2i-1} + F_{2i, 2j; 2i}
=&
-F_{2i, 2j-1;2i-1} + F_{2i-1, 2j-1; 2i} \\
=&
F_{2j-1, 2i ;2i-1} + F_{2i-1, 2j-1; 2i}
= F_{2i-1, 2i; 2j-1}, \\
F_{2i, 2j-1 ;2i-1} -F_{2i-1, 2j-1 ;2i} 
=& 
F_{2i, 2j-1 ;2i-1} + F_{2j-1, 2i-1 ;2i} 
=
-F_{2i-1,2i;2j-1}, \\
- F_{2i, 2j; 2i-1} + F_{2i-1, 2j; 2i} 
=&
F_{2j, 2i; 2i-1} + F_{2i-1, 2j; 2i}
= 
F_{2i-1,2i;2j}. 
\end{align*}
Hence, we obtain at $p \in X$ 
\begin{align} \label{eq:MC02 12}
\begin{split}
&\sum_{\xi=1}^{2n} i(e_\xi) D_\xi \left(G_\n^{-1}\circ \FF_\n^\sharp \right)^{\flat} \\
=&
\sum_{i, j=1}^n 
\left( 
\frac{ - F_{2i-1, 2i ;2j} e^{2j-1} + F_{2i-1, 2i; 2j-1} e^{2j}} {(1+ \lambda_i^2) (1+ \lambda_j^2)} 
+
\frac{ \lambda_i \lambda_j 
\left( - F_{2i-1, 2i ;2j-1} e^{2j} + F_{2i-1, 2i ;2j} e^{2j-1} \right)}
{(1+ \lambda_i^2) (1+ \lambda_j^2)} 
\right). 
\end{split}
\end{align}
Then, by \eqref{eq:MC02 1}, \eqref{eq:MC02 11} and \eqref{eq:MC02 12}, 
we obtain at $p \in X$ 
\begin{align} \label{eq:MC02 2}
H (\n) = 
\sum_{i, j=1}^n 
\frac{ \left(\det G_\n \right)^{1/4}  
\left( - F_{2i-1, 2i ;2j} e^{2j-1} + F_{2i-1, 2i; 2j-1} e^{2j} \right) }{(1+ \lambda_i^2) (1+ \lambda_j^2)}. 
\end{align}

Now, we relate \eqref{eq:MC02 2} to $d \theta_\n$. We first prove the following. 
\begin{lemma} \label{lem:diff theta}
At $p \in X$, we have 
$$
e_\xi (\theta_\n) = \sum_{i=1}^n \frac{F_{2i-1,2i;\xi}}{1+\lambda_i^2}. 
$$
\end{lemma}

\begin{proof}
Since $\zeta_\n=r_\n e^{\i \theta_\n}$, we have 
$$
e_\xi (\theta_\n) = e_\xi (\Im (\log \zeta_\n)) = \Im \left( \frac{e_\xi (\zeta_\n)}{\zeta_\n} \right). 
$$
Differentiating $(\om + \i \FF_\n)^n = \zeta_\n \om^n$, we have 
$$
n \i D_\xi \FF_\n \wedge (\om + \i \FF_\n)^{n-1} = e_\xi (\zeta_\n) \om^n. 
$$
Dividing both hand sides by $\zeta_\n$ and take the imaginary part, we obtain 
$$
n D_\xi \FF_\n \wedge {\rm Re} \left( \frac{(\om + \i \FF_\n)^{n-1}}{\zeta_\n} \right) = e_\xi (\theta_\n) \om^n. 
$$
By \cite[(4.3)]{KY}, we have 
${\rm Re} \left( (\om + \i \FF_\n)^{n-1}/\zeta_\n \right) = \om_\n^{n-1}/r_\n^2 = (n-1)! *_\n \om_\n/r_\n^2$, 
where $*_\n$ is the Hodge star defined by $\eta_\n$.  
Using $\om_\n^n=r_\n^2 \om^n$ by \cite[(4.2)]{KY}, we obtain 
\begin{align} \label{eq:diff theta 1}
\left \la D_\xi \FF_\n, \om_\n \right \ra_\n = e_\xi (\theta_\n),  
\end{align}
where $\la \cdot, \cdot \ra_{\n}$ is the metric on $T^* X$ induced from $\eta_\n$. 

We write the left hand side of \eqref{eq:diff theta 1} at $p \in X$. 
Denote by $\la \cdot, \cdot \ra$ is the metric on $T^* X$ induced from $g$. 
We compute at $p \in X$ 
\begin{align*}
\left \la D_\xi \FF_\n, \om_\n \right \ra_\n
=&
\left \la ((\id_{TX} + \FF_\n^\sharp)^{-1})^* D_\xi \FF_\n, \om \right \ra \\
=&
\sum_{i=1}^n
(D_\xi \FF_\n) \left( (\id_{TX} + \FF_\n^\sharp)^{-1}(e_{2i-1}), (\id_{TX} + \FF_\n^\sharp)^{-1}(e_{2i}) \right). 
\end{align*}
By \eqref{eq:Daz 1}, we have 
$$
(\id_{TX} + \FF_\n^\sharp)^{-1}(e_{2i-1}) = \frac{e_{2i-1} - \lambda_i e_{2i}}{1+ \lambda_i^2}, \qquad
(\id_{TX} + \FF_\n^\sharp)^{-1}(e_{2i}) = \frac{\lambda_i e_{2i-1} + e_{2i}}{1+ \lambda_i^2}. 
$$
Hence, we obtain 
$$
\left \la D_\xi \FF_\n, \om_\n \right \ra_\n
=
\sum_{i=1}^n \frac{(D_\xi \FF_\n) (e_{2i-1}, e_{2i})}{1+ \lambda_i^2}
=
\sum_{i=1}^n \frac{F_{2i-1, 2i; \xi}}{1+ \lambda_i^2}. 
$$
By \eqref{eq:diff theta 1}, the proof of Lemma \ref{lem:diff theta} is completed. 
\end{proof}

By Lemma \ref{lem:diff theta}, we compute at $p \in X$ 
$$
J d \theta_\n 
= \sum_{j=1}^n e_{2j-1} (\theta_\n) e^{2j-1}(J (\cdot)) + e_{2j} (\theta_\n) e^{2j}(J (\cdot))
=
\sum_{i, j=1}^n \frac{-F_{2i-1, 2i; 2j-1} e^{2 j} + F_{2i-1, 2i; 2j} e^{2j-1}}{1+\lambda_i^2}. 
$$
Then, by \eqref{eq:Daz 1}, we obtain 
\begin{align} \label{eq:Daz 5}
(G_\n^{-1})^* (J d \theta_\n)
=
\sum_{i, j=1}^n \frac{-F_{2i-1, 2i; 2j-1} e^{2 j} + F_{2i-1, 2i; 2j} e^{2j-1}}{(1+\lambda_i^2)(1+\lambda_j^2)}. 
\end{align}
Then, by \eqref{eq:MC02 2} and \eqref{eq:Daz 5}, the proof is completed. 
\end{proof}

\begin{remark} \label{rem:JY MC}
We can recover \cite[Proposition 3.4]{JY} from \eqref{eq:MC} and Theorem \ref{thm:Dazord}. 

Suppose that $L \to X$ is a holomorphic line bundle and the holomorphic structure is fixed. 
Set 
$$
\Aa_{met}= \{ h \mid \mbox{ a Hermitian metric of } L \}. 
$$
Fixing $h \in \Aa_{met}$, we have $\Aa_{met}= \{ e^{-f} h \mid f \in \Om^0 \}$. 
Denote by $\Aa_0$ the space of Hermitian connections of $(L,h)$. 
Define 
$\mathcal{I}: \Aa_{met} \to \Aa_0$ by 
$$
\mathcal{I} (e^{-f} h)= \n - \frac{\i}{2} d_c f, 
$$
where $\n$ is the Chern connection of $h$. 
Since $2 \p \bp=-\i d d_c$ in our notation given in \cite[Appendix A.2]{KY},  
we see that 
$F_{h'} = F_{\mathcal{I} (h')}$ for any $h' \in \Aa_{met}$,  
where 
$F_{h'}$ is the curvature 2-form of the Chern connection of $h' \in \Aa_{met}$. 
Then, the volume functional $V_{JY}: \Aa_{met} \to \R$ in \cite[Definition 3.1]{JY} is given by 
$$
V_{JY} = V \circ \mathcal{I}. 
$$
Thus, 
$$
\left. \frac{d}{dt} V_{JY} (e^{-t f} h) \right|_{t=0} 
= 
\delta_\n V \left(-\frac{\i}{2} d_c f \right) 
$$
is computed in \cite[Proposition 3.4]{JY}. 
By Theorem \ref{thm:Dazord}, we have 
$$
\delta_\n V \left(-\frac{\i}{2} d_c f \right) = 
\frac{1}{2} \la d_c f, H(\n) \ra_{L^2} 
= - \frac{1}{2} \int_X \left \la(G_\n^{-1})^* (J d \theta_\n), d_c f \right \ra (\det G_\n)^{1/4} \vol_g. 
$$
Denote by ${}^t T$ the transpose of a linear map $T:T^* X \to T^* X$ with respect to $g$. 
By \eqref{eq:Gn}, we have 
$$
(G_\n^{-1})^* 
= {}^t\! \left( \left(\id_{TX} + \FF_\n^\sharp \right)^{-1} \right)^* \left( \left(\id_{TX} + \FF_\n^\sharp \right)^{-1} \right)^*. 
$$
Then, 
\begin{align*}
\left \la (G_\n^{-1})^* (J d \theta_\n), d_c f \right \ra  
=&
\left \la \left( \left(\id_{TX} + \FF_\n^\sharp \right)^{-1} \right)^* (J d \theta_\n), 
\left( \left(\id_{TX} + \FF_\n^\sharp \right)^{-1} \right)^* d_c f  \right \ra \\
=&
\la J d \theta_\n, d_c f \ra_\n,  
\end{align*}
where $\la \cdot, \cdot \ra_{\n}$ is the metric on $T^* X$ induced from $\eta_\n$. 
Since $\eta_\n$ is Hermitian and $d_c f = - J df$, we have 
$$
\la J d \theta_\n, d_c f \ra_\n = - \la d \theta_\n, d f \ra_\n. 
$$
Hence, by \eqref{eq:Gn} again, we obtain 
\begin{align*}
\left. \frac{d}{dt} V_{JY} (e^{-t f} h) \right|_{t=0} 
= 
\delta_\n V \left(-\frac{\i}{2} d_c f \right) 
= \frac{1}{2} \int_X \la d \theta_\n, d f \ra_\n \sqrt{\det \left(\id_{TX} + \FF_\n^\sharp \right)} \vol_g. 
\end{align*}
Note that the formula above differs from \cite[Proposition 3.4]{JY} by a constant 
because we use the convention that $|a+ \bar a|^2= 2 |a|^2$ for $a \in (T^{1,0} X)^*$ (See also \cite[(3.7)]{JY}.)  
and adopt the different definition for dHYM connections as explained in Remark \ref{rem:diff JY} (2). 
\end{remark}

\section{Holonomy reductions} \label{sec:hol red} 

If a $\Sp$- or $G_2$-manifold admits a further reduction 
of the holonomy group, 
we have $\Sp$- or $G_2$-dDT connections 
from $G_2$-dDT or dHYM connections, respectively. 
In this section, 
as applications of results proved in previous sections, 
we show that 
all of $\Sp$- or $G_2$-dDT connections 
essentially arise in this way 
if the holonomy group reduces and 
the manifold is compact and connected.

\subsection{Reduction from $\Sp$ to $G_2$}

Let $(Y^7, \varphi)$ be a $G_2$-manifold and 
$L \rightarrow Y^7$ be a smooth complex line bundle with a Hermitian metric $h$. 
Set $X^8=S^1 \times Y^7$ and denote by 
$$
\pi_{S^1}: X^8=S^1 \times Y^7 \to S^1, \qquad 
\pi_Y: X^8=S^1 \times Y^7 \to Y^7
$$ 
the projections. 
The 8-manifold $X^8=S^1 \times Y^7$ is a $\Sp$-manifold 
with the $\Sp$-structure 
$$
\Phi = \pi_{S^1}^* dx \wedge \pi_Y^* \varphi + \pi_Y^* *_7 \varphi, 
$$ 
where $x$ is a coordinate of $S^1$
and $*_7$ is the Hodge star on $Y^7$. 
In \cite[Lemma 7.1]{KYFM}, 
we prove that 
a Hermitian connection $\nabla$ of $(L,h)$ 
is a $G_2$-dDT connection  
if and only if the pullback $\pi_Y^* \nabla$ is a $\Sp$-dDT connection 
of $\pi_Y^*L$. 
We can improve this if $Y^7$ is compact and connected as follows.

\begin{theorem} \label{thm:Spin7G2dDT}
Suppose that $Y^7$ is a compact and connected $G_2$-manifold. 

\begin{enumerate}
\item
The pullback $\pi_Y^*L \to X^8$ admits a $\Sp$-dDT connection if and only if 
$L \to Y^7$ admits a $G_2$-dDT connection. 

\item
For any $\Sp$-dDT connection $\widetilde \n$ of $\pi_Y^* L$, 
there exist a $G_2$-dDT connection $\n$ of $L$ and 
a closed 1-form $\xi \in \i \Om^1 (X^8)$ 
such that 
$\widetilde \n = \pi_Y^* \n + \xi$. 

\item 
Denote by $\Mm_\Sp$ the moduli space of $\Sp$-dDT connections of $\pi_Y^* L$ 
and denote by $\Mm_{G_2}$ the moduli space of $G_2$-dDT connections of $L$ 
as defined by \eqref{eq:mod Spin7} and \eqref{eq:mod G2}. 
Then, $\Mm_\Sp$ is homeomorphic to $S^1 \times \Mm_{G_2}$. 
\end{enumerate}
\end{theorem}

\begin{remark}
An analogous result for $\Sp$-instantons is given in \cite[Theorem 1.17]{Wang}, 
where the moduli space of irreducible $\Sp$-instantons of $\pi_Y^* L$ is homeomorphic to 
the product of $S^1$ and the moduli space of irreducible $G_2$-instantons of $L$. 
An analogous result for compact Cayley submanifolds is given in \cite[Proposition 5.20]{Ohst}, 
where submanifolds are allowed to have boundaries.
This implies that the moduli space of all local Cayley deformations of $S^1 \times A^3$, 
where $A^3$ is a given compact associative submanifold in $Y^7$, 
is identified with the moduli space of all local associative deformations of $A^3$. 
\\

\end{remark}

\begin{proof}
We first prove (1). 
By \cite[Lemma 7.1]{KYFM}, for any $G_2$-dDT connection $\nabla$ of $(L,h)$, 
$\pi_Y^* \nabla$ is a $\Sp$-dDT connection of $\pi_Y^*L$. 
Thus, we only have to prove the converse. 

Suppose that a $\Sp$-dDT connection $\widetilde \n$  of $\pi_Y^*L$ is given. 
By Remark \ref{rem:novanish} and the connectedness of $X^8$, 
we may assume that 
$1+\la F_{\widetilde \n}^2, \Phi \ra/2 + * F^4_{\widetilde \n}/24 >0$ 
everywhere without loss of generality. 
Then, 
Theorem \ref{thm:Cayvol ineq} implies that 
\begin{align*}
V(\widetilde \n)
&= \int_X \left( 1 + \frac{1}{2} \la F_{\widetilde \n}^2, \Phi \ra + \frac{* F^4_{\widetilde \n}}{24}  \right) \vol_g \\
&= \int_X 
\left( \left( 1 + \frac{1}{2} \la F_{\widetilde \n}^2, \pi_{S^1}^* dx \wedge \pi_Y^* \varphi + \pi_Y^* *_7 \varphi \ra \right) \vol_g
+ \frac{F^4_{\widetilde \n}}{24} \right). 
\end{align*}
We compute 
$$
\int_X F^4_{\widetilde \n} 
= \left( (-2\pi \i c_1(\pi^*_Y L))^4 \right) \cdot [X]
= 16 \pi^4 \left( \pi_Y^* c_1(L)^4 \right) \cdot [X]
= 0 
$$
since $c_1(L)^4 \in H^8_{dR} (Y^7) = \{ 0 \}$. 
Similarly, since 
$\la F_{\widetilde \n}^2, \pi_{S^1}^* dx \wedge \pi_Y^* \varphi \ra \vol_g
= F_{\widetilde \n}^2 \wedge \pi_Y^* *_7 \varphi$, 
we have 
\begin{align*}
\int_X \la F_{\widetilde \n}^2, \pi_{S^1}^* dx \wedge \pi_Y^* \varphi \ra \vol_g
=&
\left( (-2\pi \i c_1(\pi^*_Y L))^2 \cup [\pi_Y^* *_7 \varphi] \right) \cdot [X] \\
=& 
-4 \pi^2 \pi^*_Y \left( c_1(L)^2 \cup [*_7 \varphi] \right) \cdot [X]
= 0
\end{align*}
since $c_1(L)^2 \cup [*_7 \varphi] \in H^8_{dR} (Y^7) = \{ 0 \}$. 
Hence, we obtain 
\begin{align} \label{eq:Spin7G2 1}
\int_X \left ( 1 + \frac{1}{2} \la F_{\widetilde \n}^2, \pi_Y^* *_7 \varphi \ra \right) \vol_g 
= V(\widetilde \n). 
\end{align} 
By Lemma \ref{lem:Cay asso pt},  
\eqref{eq:Spin7G2 1} implies that 
$1 + \la F_{\widetilde \n}^2, \pi_Y^* *_7 \varphi \ra/2 = \sqrt{\det (\id_{TX} + (-\i F_{\widetilde \n})^\sharp)}$ 
and 
\begin{align} \label{eq:Spin7G2 2}
i \left( \frac{\p}{\p x} \right)F_{\widetilde \n}=0 \quad \mbox{and} \quad 
F_{\widetilde \n} \wedge \pi_Y^* *_7 \varphi + \frac{1}{6} F_{\widetilde \n}^3 =0. 
\end{align}
Now, fix a Hermitian connection $\n_0$ of $L$ and set $\widetilde \n = \pi_Y^* \n_0 + \widetilde \gamma$ 
for $\widetilde \gamma  \in \i \Om^1(X^8)$. 
Then, $F_{\widetilde \n} = \pi_Y^* F_{\n_0} + d \widetilde \gamma$. 
Since 
$$
0=i \left( \frac{\p}{\p x} \right)F_{\widetilde \n} 
= i \left( \frac{\p}{\p x} \right) d \widetilde \gamma 
\quad \mbox{and} \quad
L_{\frac{\p}{\p x}} d \widetilde \gamma = 0,  
$$
where $L_{\p/\p x}$ is the Lie derivative in the direction of $\p/\p x$, 
it follows that 
$d \widetilde \gamma \in \i \pi_Y^* \Om^2(Y^7)$. 
Since $\pi_Y^*:H^2_{dR} (Y^7) \to H^2_{dR} (X^8)$ is injective, 
there exists a 1-form $\gamma \in \i \Om^1(Y^7)$ such that 
$d \widetilde \gamma = \pi_Y^* d \gamma$. 
Thus, there exists a closed 1-form $\xi \in  \i \Om^1(X^8)$ such that 
$\widetilde \gamma = \pi_Y^* \gamma + \xi$. 
Hence, we obtain  
$$
\widetilde \n = \pi_Y^* \n_0 + \pi_Y^* \gamma + \xi. 
$$
By \eqref{eq:Spin7G2 2}, the fact that 
$F_{\widetilde \n} = \pi_Y^* F_{\n_0 + \gamma}$ 
and $\pi_Y^*$ is injective, 
we see that $\n_0 + \gamma$ is a $G_2$-dDT connection. 
Then, the proof of (1) is completed. 
This argument also implies (2). \\

Next, we prove (3). 
Let $\widehat \Mm_\Sp$ and $\widehat \Mm_{G_2}$ be the set of 
$\Sp$-dDT connections of $\pi_Y^* L$ and 
$G_2$-dDT connections of $L$ as defined by \eqref{eq:premod Spin7} and \eqref{eq:premod G2}, respectively. 
Define a continuous map 
$\widehat \Theta: \i Z^1(S^1) \times \widehat \Mm_{G_2} \to \widehat \Mm_\Sp$ by 
$$
\widehat \Theta (\alpha, \n) = \pi_{S^1}^* \alpha + \pi_Y^* \n,  
$$
where $Z^1(S^1)$ is the space of closed 1-forms on $S^1$. 
By Lemma \ref{lem:orbit}, it induces a continuous map 
$$
\Theta: \frac{\i Z^1(S^1)}{\i \left( 2\pi \Hh^1_\Z (S^1) \oplus d \Om^0(S^1) \right)} 
\times \Mm_{G_2} \to \Mm_\Sp. 
$$
We show that $\Theta$ is a homeomorphism by constructing the inverse map. 
Fixing $pt_{S^1} \in S^1$ and $pt_Y \in Y^7$, 
define $\iota_{S^1}: S^1 \hookrightarrow S^1 \times Y^7$ and $\iota_Y: Y^7 \hookrightarrow S^1 \times Y^7$
by 
$\iota_{S^1}(z)=(z, pt_Y)$ and  $\iota_Y (y) = (pt_{S^1}, y)$. 
Define a continuous map 
$
\widehat \Xi: \widehat \Mm_\Sp \to \i Z^1(S^1) \times \widehat \Mm_{G_2} 
$
by 
$$
\widehat \Xi (\widetilde \n) = \left( \iota_{S^1}^* \left( \widetilde \n - \pi_Y^* \iota_Y^* \widetilde \n \right), 
\iota_Y^* \widetilde \n \right),  
$$
where we regard $\widetilde \n - \pi_Y^* \iota_Y^* \widetilde \n$ as a $\i \R$-valued 1-form on $X^8$. 
First, we show that $\widehat \Xi$ is well-defined. 
For any $\widetilde \n \in \widehat \Mm_\Sp$, there exist a $G_2$-dDT connection $\n \in \widehat \Mm_{G_2}$ 
and a closed 1-form $\xi \in \i \Om^1 (X^8)$ such that 
\begin{align} \label{eq:Spin7G2 3}
\widetilde \n = \pi_Y^* \n + \xi
\end{align} 
by (2). Then, 
$$
\iota_{S^1}^* \left( \widetilde \n - \pi_Y^* \iota_Y^* \widetilde \n \right)
=
\iota_{S^1}^* \left( \pi_Y^* (\n- \iota_Y^* \widetilde \n) + \xi \right)
= 
\iota_{S^1}^* \xi, 
$$
and hence,  $\iota_{S^1}^* \left( \widetilde \n - \pi_Y^* \iota_Y^* \widetilde \n \right) \in \i Z^1(S^1). $
We also have 
$$
\iota_Y^* \widetilde \n = \iota_Y^* \pi_Y^* \n + \iota_Y^* \xi = \n + \iota_Y^* \xi.
$$ 
Since $F_{\iota_Y^* \widetilde \n} = F_{\n + \iota_Y^* \xi} = F_\n$, 
we see that $\iota_Y^* \widetilde \n \in \widehat \Mm_{G_2}$.

The map $\widehat \Xi$ induces the continuous map 
$$
\Xi: \Mm_\Sp \to 
\frac{\i Z^1(S^1)}{\i \left( 2\pi \Hh^1_\Z (S^1) \oplus d \Om^0(S^1) \right)} \times \Mm_{G_2}.   
$$
Indeed, if $\Sp$-dDT connections $\widetilde \n, \widetilde \n'$ are in the same $\Gg_U$-orbit, 
Lemma \ref{lem:orbit} implies that 
$\widetilde \n' - \widetilde \n = \eta$ for $\eta \in \i \left( 2\pi \Hh^1_\Z (X^8) \oplus d \Om^0(X^8) \right)$. 
Then, since 
\begin{align*}
\iota_{S^1}^* (\eta - \pi_Y^* \iota_Y^* \eta) = \iota_{S^1}^* \eta &\in \i \left( 2\pi \Hh^1_\Z (S^1) \oplus d \Om^0(S^1) \right), \\
\iota_Y^* \eta &\in \i \left( 2\pi \Hh^1_\Z (Y^7) \oplus d \Om^0(Y^7) \right), 
\end{align*}
we see that $\widehat \Xi$ induces $\Xi$. 

Now, we show that $\Theta$ and $\Xi$ are mutually inverse. 
For any $(\alpha, \n) \in \i Z^1(S^1) \times \widehat \Mm_{G_2}$, we compute 
\begin{align*}
(\widehat \Xi \circ \widehat \Theta) (\alpha, \n)
=&
\left( \iota_{S^1}^* 
\left( \pi_{S^1}^* \alpha + \pi_Y^* \n - \pi_Y^* \iota_Y^* (\pi_{S^1}^* \alpha + \pi_Y^* \n) \right),  
\iota_Y^* (\pi_{S^1}^* \alpha + \pi_Y^* \n) \right) \\
=&
(\alpha, \n), 
\end{align*}
which implies that $\Xi \circ \Theta$ is the identity map.  
For any $\widetilde \n \in \widehat \Mm_\Sp$, we compute 
\begin{align*}
(\widehat \Theta \circ \widehat \Xi) (\widetilde \n)
=
\pi_{S^1}^* \iota_{S^1}^* \left( \widetilde \n - \pi_Y^* \iota_Y^* \widetilde \n \right) 
+ \pi_Y^* \iota_Y^* \widetilde \n. 
\end{align*}
Substituting \eqref{eq:Spin7G2 3} into this, we further compute 
\begin{align*}
(\widehat \Theta \circ \widehat \Xi) (\widetilde \n) = 
\pi_{S^1}^* \iota_{S^1}^* \xi + \pi_Y^* \n + \pi_Y^* \iota_Y^* \xi
=
\widetilde \n + \left( \pi_{S^1}^* \iota_{S^1}^* \xi + \pi_Y^* \iota_Y^* \xi - \xi \right). 
\end{align*}
Since $H^1_{dR} (S^1) \oplus H^1_{dR} (Y^7) \cong H^1_{dR} (X^8)$ via 
$([\eta_{S^1}], [\eta_Y]) \mapsto [\pi_{S^1}^* \eta_{S^1} + \pi_Y^* \eta_Y]$ 
by the K\"unneth formula, we see that 
there exist 
$\eta_{S^1} \in \i Z^1 (S^1), \eta_Y \in \i Z^1 (Y^7)$ and $f_0 \in \i \Om^0(X^8)$ such that 
$\xi = \pi_{S^1}^* \eta_{S^1} + \pi_Y^* \eta_Y + d f$. 
Then, 
$$
\pi_{S^1}^* \iota_{S^1}^* \xi + \pi_Y^* \iota_Y^* \xi - \xi
= 
\pi_{S^1}^* \iota_{S^1}^* d f + \pi_Y^* \iota_Y^* d f - d f \in \i d \Om^0(X^8). 
$$
Thus, by Lemma \ref{lem:orbit}, 
$(\widehat \Theta \circ \widehat \Xi) (\widetilde \n)$ and $\widetilde \n$
are in the same $\Gg_U$-orbit, 
which implies that $\Theta \circ \Xi$ is the identity map. 
Hence, we obtain the homeomorphism 
$$
\Mm_\Sp \cong 
\frac{\i Z^1(S^1)}{\i \left( 2\pi \Hh^1_\Z (S^1) \oplus d \Om^0(S^1) \right)} \times \Mm_{G_2}
\cong 
\frac{H^1(S^1, \R)}{2 \pi H^1(S^1, \Z)} \times \Mm_{G_2},  
$$
and the proof is completed. 
\end{proof}

\subsection{Reduction from $G_2$ to ${\rm SU}(3)$}
Let $(Y^6,\om,g,J,\Om)$ be a real 6 dimensional Calabi--Yau manifold and 
$L \rightarrow Y$ be a smooth complex line bundle with a Hermitian metric $h$. 
Set $X^7=S^1 \times Y^6$ and denote by 
$$
\pi_{S^1}: X^7=S^1 \times Y^6 \to S^1, \qquad 
\pi_Y: X^7=S^1 \times Y^6 \to Y^6
$$ 
the projections. 
The 7-manifold $X^7=S^1 \times Y^6$ is a $G_2$-manifold 
and the $G_2$-structure $\varphi$ and its Hodge dual $* \varphi$ are given by 
\[
\varphi=\pi_{S^1}^* dx \wedge \pi_Y^* \om + \pi_Y^* \mathop{\mathrm{Re}} \Om, \qquad 
* \varphi = \frac{1}{2} \pi_Y^* \om^2 - \pi_{S^1}^* dx \wedge \pi_Y^* \Im \Om, 
\]
where $x$ is a coordinate of $S^1$. 
In \cite[Lemma 3.5]{KY}, we prove that 
a Hermitian connection $\nabla$ of $(L,h)$ is a dHYM connection with phase $1$ 
if and only if the pullback $\pi_Y^* \nabla$ 
is a $G_2$-dDT connection of 
$\pi_Y^* L$. 
We can improve this if $Y^6$ is compact and connected as follows.

\begin{theorem} \label{thm:G2dDTdHYM}
Suppose that $Y^6$ is a compact and connected Calabi--Yau manifold. 
\begin{enumerate}
\item
The pullback $\pi_Y^*L \to X^7$ admits a $G_2$-dDT connection if and only if 
$L \to Y^6$ admits a dHYM connection with phase 1. 

\item
For any $G_2$-dDT connection $\widetilde \n$ of $\pi_Y^* L$, 
there exist a dHYM connection $\n$ with phase 1 of $L$ 
and a closed 1-form $\xi \in \i \Om^1(X^7)$ such that 
$\widetilde \n = \pi_Y^* \n + \xi$. 

\item 
Denote by $\Mm_{G_2}$ the moduli space of $G_2$-dDT connections of $\pi_Y^* L$ 
and denote by $\Mm_{dHYM}$ the moduli space of dHYM connections with phase 1 of $L$ 
as defined by \eqref{eq:mod G2}. 
Then, $\Mm_{G_2}$ is homeomorphic to $S^1 \times \Mm_{dHYM}$. 

In particular, $\Mm_{G_2}$ is a torus of dimension $1+b^1(Y^6)=b^1(X^7)$ by \cite[Theorem 1.2 (1)]{KY}. 
\end{enumerate}
\end{theorem}

\begin{remark}
An analogous result for $G_2$-instantons is given in \cite[Theorem 1.17]{Wang}, 
where the moduli space of irreducible $G_2$-instantons of $\pi^* L$ is homeomorphic to 
the product of $S^1$ and the moduli space of irreducible Hermitian Yang--Mills connections with 0-slope of $L$. 
An analogous result for compact associative submanifolds is given in \cite[Proposition 4.6]{Gayet}, 
where 
the moduli space of all local associative deformations of $\{ * \} \times A^3$, 
where $A^3$ is a given compact special Lagrangian submanifold in $Y^6$, 
is identified with the product of $S^1$ and 
the moduli space of all local special Lagrangian deformations of $A^3$. 

Corollary \ref{cor:flat G2} (2) and Theorem \ref{thm:G2dDTdHYM} 
imply that we will need to consider a non-flat line bundle over a manifold with full holonomy $G_2$ 
to construct nontrivial examples on a compact and connected $G_2$-manifold.

It is conjectured in \cite[Conjecture 1.5]{CJY} that the existence of a dHYM connection is 
equivalent to a certain stability condition and the conjecture is partially proved in \cite{Chen}. 
Theorem \ref{thm:G2dDTdHYM} (1) implies that 
the existence of a $G_2$-dDT connection of $\pi_Y^*L \to X^7$ would be equivalent to this stability condition.  
More generally, 
the existence of a $G_2$-dDT connection of a general line bundle over a general $G_2$-manifold 
might be related to a certain stability condition. 
\end{remark}

\begin{proof}
We first prove (1). 
The proof is also almost the same as for Theorem \ref{thm:Spin7G2dDT}. 
By \cite[Lemma 3.5]{KY}, for any dHYM connection $\n$ with phase $1$ of $(L,h)$, 
$\pi_Y^* \nabla$ is a $G_2$-dDT connection of $\pi_Y^*L$. 
Thus, we only have to prove the converse. 

Suppose that a $G_2$-dDT connection $\widetilde \n$  of $\pi_Y^*L$ is given. 
By Remark \ref{rem:novanish G2} and the connectedness of $X^7$, 
we may assume that 
$1+\la F_{\widetilde \n}^2, * \varphi \ra/2 >0$ 
everywhere without loss of generality. 
Denote by $*=*_7$ and $*_6$ the Hodge stars on $X^7$ and $Y^6$, respectively. 
Then, 
Theorem \ref{thm:assovol ineq} implies that 
\begin{align*}
V(\widetilde \n)
&= \int_X \left( 1 + \frac{1}{2} \la F_{\widetilde \n}^2, * \varphi \ra \right) \vol_g \\
&= \int_X \left( 1 + \frac{1}{2} \left \la F_{\widetilde \n}^2,  \frac{1}{2} \pi_Y^* \om^2 
- \pi_{S^1}^* dx \wedge \pi_Y^* \Im \Om \right \ra \right) \vol_g.  
\end{align*}
Since 
$
- \la F_{\widetilde \n}^2, \pi_{S^1}^* dx \wedge \pi_Y^* \Im \Om \ra \vol_g 
=
- F_{\widetilde \n}^2 \wedge \pi_Y^* *_6  \Im \Om
=
F_{\widetilde \n}^2 \wedge \pi_Y^* {\rm Re} \Om, 
$
where we use \eqref{eq:SU3 wk}, it follows that 
\begin{align*}
- \int_X \la F_{\widetilde \n}^2, \pi_{S^1}^* dx \wedge \pi_Y^* \Im \Om \ra \vol_g
=&
\left( (-2\pi \i c_1(\pi^*_Y L))^2 \cup [\pi_Y^* {\rm Re} \Om] \right) \cdot [X] \\
=& 
-4 \pi^2 \pi^*_Y \left( c_1(L)^2 \cup [{\rm Re} \Om] \right) \cdot [X]
= 0
\end{align*}
since $c_1(L)^2 \cup [{\rm Re} \Om] \in H^7_{dR} (Y^6) = \{ 0 \}$. 
Hence, we obtain 
\begin{align} \label{eq:G2dDTdHYM 1}
\int_X \left ( 1 +\frac{1}{4} \la F_{\widetilde \n}^2, \pi_Y^* \om^2 \ra \right) \vol_X 
= V(\widetilde \n). 
\end{align} 
By Lemma \ref{lem:asso SL3 pt}, 
\eqref{eq:G2dDTdHYM 1} implies that 
$1 + \la F_{\widetilde \n}^2, \pi_Y^* \om^2 \ra/4 = \sqrt{\det (\id_{TX} + (-\i F_{\widetilde \n})^\sharp)}$ 
and 
\begin{align*}
i \left( \frac{\p}{\p x} \right)F_{\widetilde \n}=0, \quad 
{\rm Im} \left( \frac{(\pi_Y^* \om+ F_{\widetilde \n})^3}{3!} \right) =0 
\quad \mbox{and} \quad 
\pi^{\lb 2,0 \rb} (F_{\widetilde \n})=0. 
\end{align*}
Then, (1) follows by the same argument as in Theorem \ref{thm:Spin7G2dDT} (1). 
This argument also implies (2). 
Since (3) follows from (2) by the same argument as in Theorem \ref{thm:Spin7G2dDT}, 
the proof is completed. 
\end{proof}

\subsection{Reduction from $\Sp$ to ${\rm SU} (4)$}
Let $(X^8,\om,g,J,\Om)$ be a real 8 dimensional Calabi--Yau manifold 
and $L \rightarrow Y$ be a smooth complex line bundle with a Hermitian metric $h$. 
Then, $X^8$ has an induced $\Sp$-structure $\Phi$ given by 
\[\Phi=\frac{1}{2} \om^2 + \mathop{\mathrm{Re}} \Om. \]
In \cite[Lemma 7.2]{KYFM}, 
if 
the $(0,2)$-part $F_\n^{0,2}$ of the curvature of $F_\n$ of 
a Hermitian connection $\nabla$ of $(L,h)$ vanishes, 
we show that $\n$ is a dHYM connection with phase $1$ 
if and only if it is a $\Sp$-dDT connection. 
We can improve this if $X^8$ is compact and connected as follows.

\begin{theorem} \label{thm:Spin7dDTdHYM}
Suppose that 
$X^8$ is a compact and connected Calabi--Yau manifold. 
\begin{enumerate}
\item
The complex line bundle $L \to X^8$ admits a holomorphic structure and a $\Sp$-dDT connection 
if and only if $L \to X^8$ admits a dHYM connection with phase 1. 

\item
Suppose that there exists a dHYM connection $\n_0$ with phase $1$. 
Then, any $\Sp$-dDT connection is a dHYM connection with phase $1$. 
Hence, the moduli space of $\Sp$-dDT connections 
agrees with that of dHYM connections with phase $1$, 
which is a $b^1$-dimensional torus by \cite[Theorem 1.2 (1)]{KY}. 
\end{enumerate}
\end{theorem}

\begin{remark}
An analogous result for $\Sp$-instantons is given in \cite[Theorem 3.1]{Lewis}, 
where any $\Sp$-instanton is a Hermitian Yang--Mills connection 
if a vector bundle admits a Hermitian Yang--Mills connection. 
An analogous result for compact Cayley submanifolds is given in \cite[Proposition 5.12]{Ohst}, 
where submanifolds are allowed to have boundaries.
This implies that 
the moduli space of all local Cayley deformations of $C^4$, where $C^4$ is a given compact special Lagrangian submanifold in $X^8$, 
is identified with the moduli space of all local special Lagrangian deformations of $C^4$.

Corollary \ref{cor:flat Spin7} (2), Theorems \ref{thm:Spin7G2dDT} and \ref{thm:Spin7dDTdHYM} 
imply that we will need to consider a non-flat line bundle over a manifold with full holonomy $\Sp$ 
to construct nontrivial examples on a compact and connected $\Sp$-manifold. 
\end{remark}

\begin{proof}
We first prove (1). The proof is almost the same as for 
Theorems \ref{thm:Spin7G2dDT} and \ref{thm:G2dDTdHYM}. 
By \cite[Lemma 7.2]{KYFM}, 
any dHYM connection $\n$ with phase $1$ is a $\Sp$-dDT connection. 
The vanishing of the $(0,2)$-part $F_\n^{0,2}$ implies that 
there exists a holomorphic structure on $L$ by \cite[Proposition 1.3.7]{Kob}. 
Thus, we only have to prove the converse. 

Suppose that a $\Sp$-dDT connection $\widetilde \n$  of $L$ is given. 
By Remark \ref{rem:novanish} and the connectedness of $X^8$, 
we may assume that 
$1+\la F_{\widetilde \n}^2, \Phi \ra/2 + * F^4_{\widetilde \n}/24 >0$ 
everywhere without loss of generality. 
Then, 
Theorem \ref{thm:Cayvol ineq} implies that 
\begin{align*}
V(\widetilde \n)
=& \int_X \left( 1 + \frac{1}{2} \la F_{\widetilde \n}^2, \Phi \ra + \frac{* F^4_{\widetilde \n}}{24}  \right) \vol_g \\
=& \int_X \left( \vol_g + \frac{1}{2} F_{\widetilde \n}^2 \wedge \left(\frac{1}{2} \om^2 + {\rm Re} \Om \right) 
+ \frac{F^4_{\widetilde \n}}{24} \right).  
\end{align*}
Since $L$ admits a holomorphic structure, we have $c_1(L) \in H^{1,1}(X)$, which implies that 
$$
\int_X F_{\widetilde \n}^2 \wedge {\rm Re} \Om 
= \left( (-2\pi \i c_1(L))^2 \cup [{\rm Re} \Om] \right) \cdot [X] 
= 0. 
$$
Hence, we obtain 
\begin{align} \label{eq:Spin7dDTdHYM 1}
V(\widetilde \n)
=
\int_X \left( \vol_g + \frac{1}{4} F_{\widetilde \n}^2 \wedge \om^2 + \frac{F^4_{\widetilde \n}}{24} \right)
= 
\int_X {\rm Re} \left( \frac{(\om+ F_{\widetilde \n})^4}{4!} \right). 
\end{align} 
Then, by Theorem \ref{thm:SL4 eq}, 
$\widetilde \n$ is a dHYM connection with phase 1. 
This argument also implies (2) and the proof is completed. 
\end{proof}

\appendix

\section{The proof of Theorems \ref{thm:Cayley eq}, \ref{thm:asso eq}, 
\ref{thm:SL3 eq} and \ref{thm:SL4 eq}}
\label{app:eq pt}

In this appendix, we prove Theorems \ref{thm:Cayley eq}, \ref{thm:asso eq}, 
\ref{thm:SL3 eq} and \ref{thm:SL4 eq}, 
which follow from 
Proposition \ref{prop:Cayley eq pt}, Corollaries \ref{cor:asso eq pt}, \ref{cor:SL3 eq pt} and Proposition \ref{prop:SL4 eq pt}.

\subsection{The Cayley case}

Use the notation of Section \ref{sec:Spin7 geometry}. 
Set $W =\R^8$ with the standard basis $\{\, e_{i} \,\}_{i=0}^{7}$ 
and its dual $\{\, e^{i} \,\}_{i=0}^{7}$. 
Let $g$ be the standard inner product on $W$. 
For a 2-form $F \in \Lambda^2 W^*$, define $F^\sharp \in {\rm End} (W)$ by 
\[
g(F^\sharp (u), v) = F(u, v)
\]
for $u,v \in W$.

\begin{proposition} \label{prop:Cayley eq pt}
For any $F \in \Lambda^2 W^*$, we have 
\begin{align*}
\left( 1- \frac{1}{2} \la F^2, \Phi \ra + \frac{* F^4}{24}  \right)^2
+
4 \left| \pi^2_7 \left( F - \frac{1}{6} * F^3 \right) \right|^2
+
2 \left| \pi^4_7 \left( F^2 \right) \right|^2 
=
\det (I_8 + F^\sharp), 
\end{align*}
where $I_8$ is the identity matrix of dimension $8$. 
\end{proposition}

This statement is essential. Using Proposition \ref{prop:Cayley eq pt}, 
we can show all the main statements in Appendix \ref{app:eq pt} 
(Corollaries \ref{cor:asso eq pt}, \ref{cor:SL3 eq pt} and Proposition \ref{prop:SL4 eq pt}).

Proposition \ref{prop:Cayley eq pt} follows from 
the following Lemmas \ref{lem:det} and \ref{lem:Caypt deg}. 
Note that a similar identity to Lemma \ref{lem:det} holds for any dimension. 

\begin{lemma} \label{lem:det}
For any $F \in \Lambda^2 W^*$, we have 
\begin{align*}
\det (I_8 + F^\sharp) = 
1+ |F|^2 + \left| \frac{F^2}{2!} \right|^2
+ \left| \frac{F^3}{3!} \right|^2
+ \left| \frac{F^4}{4!} \right|^2. 
\end{align*}
\end{lemma}

\begin{proof}
Since $F^\sharp$ is skew-symmetric, 
there exist $\lambda_1,\lambda_2,\lambda_3,\lambda_4 \in \R$ 
and $h \in {\rm O}(8)$ such that 
$$
h^{-1} F^\sharp h = 
\left(
\begin{array}{cc}
0 & - \lambda_1 \\
\lambda_1& 0
\end{array}
\right) 
\oplus 
\left(
\begin{array}{cc}
0 & - \lambda_2 \\
\lambda_2& 0
\end{array}
\right) 
\oplus 
\left(
\begin{array}{cc}
0 & - \lambda_3 \\
\lambda_3& 0
\end{array}
\right) 
\oplus 
\left(
\begin{array}{cc}
0 & - \lambda_4 \\
\lambda_4& 0
\end{array}
\right).  
$$
In other words, we have 
$$
h^* F =\lambda_1 e^{01} + \lambda_2 e^{23} + \lambda_3 e^{45} + \lambda_4 e^{67}. 
$$
Then, we obtain 
\begin{align*}
&\det (I_8 + F^\sharp) \\
=& 
(1 + \lambda_1^2) (1 + \lambda_2^2) (1 + \lambda_3^2) (1 + \lambda_4^2) \\
=&
1 
+ \sum_i \lambda_i^2
+ \sum_{i < j} \lambda_i^2 \lambda_j^2
+ \sum_{i < j< k} \lambda_i^2 \lambda_j^2 \lambda_k^2
+ \lambda_1^2 \lambda_2^2 \lambda_3^2 \lambda_4^2 \\
=&
1+ |F|^2 + \left| \frac{F^2}{2!} \right|^2
+ \left| \frac{F^3}{3!} \right|^2
+ \left| \frac{F^4}{4!} \right|^2. 
\end{align*}
\end{proof}

\begin{lemma} \label{lem:Caypt deg}
For any $F \in \Lambda^2 W^*$, we have 
\begin{align}
\label{eq:Caypt2}
- \la F^2, \Phi \ra + 4 \left| \pi^2_7 \left( F \right) \right|^2 &= |F|^2,  \\
\label{eq:Caypt4}
\frac{1}{4} \la F^2, \Phi \ra^2 + \frac{* F^4}{12} 
- \frac{4}{3}
\left \la \pi^2_7 (F),  \pi^2_7 \left( * F^3 \right) \right \ra
+
2 \left| \pi^4_7 \left( F^2 \right) \right|^2 
&=
\frac{1}{4} \left| F^2 \right|^2, \\
\label{eq:Caypt6}
- \la F^2, \Phi \ra \frac{* F^4}{24} + \frac{1}{9} \left| \pi^2_7 \left( * F^3 \right) \right|^2
&= \frac{1}{36} \left| F^3 \right|^2. 
\end{align}
\end{lemma}

\begin{proof}
Set 
$$
F = F_7 + F_{21} \in \Lambda^2_7 W^* \oplus \Lambda^2_{21} W^*.  
$$
Then, we have 
\begin{align*}
- \la F^2, \Phi \ra 
=&
-* (F \wedge F \wedge \Phi) \\
=&
-* (F \wedge *(3 F_7 - F_{21}))
=
- 3 |F_7|^2 + |F_{21}|^2, 
\end{align*}
which implies \eqref{eq:Caypt2}. 
\\

Next, we show \eqref{eq:Caypt6}. 
By \eqref{eq:Caypt2}, we have 
$$
|F^3|^2=|*F^3|^2 = 4 \left| \pi^2_7 \left( * F^3 \right) \right|^2 - *\left( (*F^3)^2 \wedge \Phi \right).
$$
By Lemma \ref{lem:24form}, it follows that  
$$
*\left( (*F^3)^2 \wedge \Phi \right) = \frac{3}{2} \la F^2, \Phi \ra* F^4. 
$$
Hence, we obtain \eqref{eq:Caypt6}. 
\\

Next, we prove \eqref{eq:Caypt4}. 
Set 
\begin{align*}
\xi_1 &=F_7^2, & \xi_2 &=F_7 \wedge F_{21}, & \xi_3 &=F_{21}^2, \\ 
\eta &= F^2, & \eta_j &= \pi^4_j (\eta)  &  \mbox{for } j&=1,7,27,35. 
\end{align*}
Recall the decomposition \eqref{eq:F2decomp} and that 
$\Lambda^4_1 W^*  \oplus \Lambda^4_7 W^* \oplus \Lambda^4_{27} W^*$ 
and $\Lambda^4_{35} W^*$ are the spaces of 
self dual 4-forms and anti self dual 4-forms, respectively. 
Then, we have $F^2 = \xi_1 + 2 \xi_2 + \xi_3$ and 
\begin{align} \label{eq:Caypt4 1}
\begin{split}
\xi_1&= * \xi_1, \\
\xi_1 \wedge \xi_2 &=0, \\
\xi_2 \wedge (\xi_3 + *\xi_3) &=0,\\
* \xi_2^2 &=* (\xi_1 \wedge \xi_3) = \la \xi_1, \xi_3 \ra. 
\end{split}
\end{align}
We first show that $|\eta_j|^2$ is described in terms of $\xi_k$'s. 

\begin{lemma} \label{lem:Caypt4 xy}
We have 
\begin{align*}
|\eta_1|^2
&=
\frac{18 |\xi_1|^2-36|\xi_2|^2 + 3 |\xi_3|^2 -* \xi_3^2}{42}, \\
|\eta_1|^2 + |\eta_{27}|^2
&=
|\xi_1|^2 + \frac{|\xi_3|^2 + * \xi_3^2}{2} + 2 \la \xi_1, \xi_3 \ra, \\
|\eta_7|^2
&=
2 |\xi_2|^2 + 2 \la \xi_1, \xi_3 \ra, \\
|\eta_{35}|^2
&=
2 |\xi_2|^2 + \frac{1}{2} |\xi_3|^2 -2 \la \xi_1, \xi_3 \ra 
-4 *(\xi_2 \wedge \xi_3) - \frac{1}{2} * \xi_3^2. 
\end{align*}
\end{lemma}
\begin{proof}
Since $|\Phi|^2=14$, we have 
$
\eta_1 = \la \eta, \Phi \ra \Phi/14, 
$
and hence, 
\begin{align} \label{eq:Caypt4 2}
|\eta_1|^2 = \frac{\la \eta, \Phi \ra^2}{14}. 
\end{align}
Since $\eta = F^2$, it follows that 
$$
\la \eta, \Phi \ra = * (F \wedge F \wedge \Phi) = * (F \wedge *(3F_7 - F_{21})) = 3 |F_7|^2 - |F_{21}|^2
$$
and 
$$
|\eta_1|^2 
= \frac{(3 |F_7|^2 - |F_{21}|^2)^2}{14}
= \frac{9 |F_7|^4 -6 |F_7|^2 |F_{21}|^2 + |F_{21}|^4}{14}. 
$$
By Proposition \ref{prop:2form norm}, it follows that 
$$
|F_7|^4= \frac{2}{3} |\xi_1|^2, \qquad 
|F_7|^2 |F_{21}|^2 = 2 |\xi_2 |^2, \qquad 
|F_{21}|^4 = |\xi_3|^2 - \frac{1}{3} * \xi_3^2. 
$$
This gives the desired formula for $|\eta_1|^2$.

Next, we compute $|\eta_1|^2 + |\eta_{27}|^2$. 
By \eqref{eq:F2decomp} and \eqref{eq:Caypt4 1}, we have 
$$
\eta_1 + \eta_{27} = \pi^4_1(\xi_1+\xi_3) + \pi^4_{27}(\xi_1+\xi_3) = \xi_1 + \frac{\xi_3+ * \xi_3}{2}. 
$$
Then, by \eqref{eq:Caypt4 1}, it follows that 
\begin{align*}
|\eta_1|^2 + |\eta_{27}|^2 
&= |\eta_1 + \eta_{27}|^2 \\
&= |\xi_1|^2 + \frac{|\xi_3+ * \xi_3|^2}{4} + \la \xi_1, \xi_3+ * \xi_3 \ra \\
&= |\xi_1|^2 + \frac{|\xi_3|^2 +  * \xi_3^2}{2} + 2 \la \xi_1, \xi_3 \ra. 
\end{align*}

Next, we compute $|\eta_7|^2$. 
By \eqref{eq:F2decomp} and \eqref{eq:Caypt4 1}, we compute 
$$
\eta_7 = 2 \pi^4_7(\xi_2) = \xi_2+ * \xi_2. 
$$
Hence, by \eqref{eq:Caypt4 1}, we obtain 
$$
|\eta_7|^2 = 2|\xi_2|^2 + 2 * \xi_2^2 = 2|\xi_2|^2 + 2 \la \xi_1, \xi_3 \ra. 
$$

Finally, we compute $|\eta_{35}|^2$. 
By \eqref{eq:F2decomp} and \eqref{eq:Caypt4 1}, we have 
$$
\eta_{35} = \pi^4_{35}(2 \xi_2+\xi_3) = \frac{2 \xi_2+\xi_3 - * (2 \xi_2+\xi_3)}{2}. 
$$
Then, by \eqref{eq:Caypt4 1}, we obtain 
\begin{align*}
2 |\eta_{35}|^2 
=&  |2 \xi_2+\xi_3|^2 - * (2 \xi_2+\xi_3)^2 \\
=& 4 |\xi_2|^2 + |\xi_3|^2 + 4 \la \xi_2, \xi_3 \ra - * \left( 4 \xi_2^2 + 4 \xi_2 \wedge \xi_3 + \xi_3^2 \right) \\
=& 
4 |\xi_2|^2 + |\xi_3|^2 -4 \la \xi_1, \xi_3 \ra 
-8 *(\xi_2 \wedge \xi_3) - * \xi_3^2. 
\end{align*}
\end{proof}

Then, 
we compute 
$$
\left \la \pi^2_7 (F),  \pi^2_7 \left( * F^3 \right) \right \ra
=
*(F_7 \wedge F \wedge F^2)
=
* \left((\xi_1 +\xi_2) \wedge (\xi_1 +2 \xi_2 + \xi_3) \right). 
$$
By \eqref{eq:Caypt4 1}, we obtain 
\begin{align} \label{eq:Caypt4 3}
\frac{4}{3} \left \la \pi^2_7 (F),  \pi^2_7 \left( * F^3 \right) \right \ra
=
\frac{4}{3} \left( |\xi_1|^2 + 3 \la \xi_1, \xi_3 \ra + * (\xi_2 \wedge \xi_3) \right). 
\end{align}
By the equation  
$$
*F^4 = * (\eta \wedge \eta)
= |\eta_1|^2 + |\eta_7|^2 + |\eta_{27}|^2 -|\eta_{35}|^2
$$
and \eqref{eq:Caypt4 2}, it follows that 
\begin{align*}
&\frac{1}{4} \la F^2, \Phi \ra^2 + \frac{* F^4}{12} 
+
2 \left| \pi^4_7 \left( F^2 \right) \right|^2 
-
\frac{1}{4} \left| F^2 \right|^2 \\
=&
\frac{7}{2} |\eta_1|^2 + \frac{1}{12} \left(|\eta_1|^2 + |\eta_7|^2 + |\eta_{27}|^2 -|\eta_{35}|^2 \right)
+2 |\eta_7|^2 \\
&-\frac{1}{4} \left(|\eta_1|^2 + |\eta_7|^2 + |\eta_{27}|^2 + |\eta_{35}|^2 \right) \\
=&
\frac{1}{6} 
\left(21 |\eta_1|^2 - \left( |\eta_1|^2 + |\eta_{27}|^2 \right) + 11|\eta_7|^2  - 2 |\eta_{35}|^2 \right).  
\end{align*}
By Lemma \ref{lem:Caypt4 xy}, we further compute 
\begin{align*}
=& 
\frac{1}{6} 
\left(
\left(9|\xi_1|^2-18|\xi_2|^2 + \frac{3 |\xi_3|^2 -* \xi_3^2}{2} \right) \right. \\
&- \left( |\xi_1|^2 + \frac{|\xi_3|^2 + * \xi_3^2}{2} + 2 \la \xi_1, \xi_3 \ra \right) 
+
11 \left( 2 |\xi_2|^2 + 2 \la \xi_1, \xi_3 \ra \right) \\
&-2 
\left. \left( 2 |\xi_2|^2 + \frac{1}{2} |\xi_3|^2 -2 \la \xi_1, \xi_3 \ra 
-4 *(\xi_2 \wedge \xi_3) - \frac{1}{2} * \xi_3^2 \right)
\right) \\
=&
\frac{4}{3} \left(  |\xi_1|^2 + 3 \la \xi_1, \xi_3 \ra + *(\xi_2 \wedge \xi_3) \right). 
\end{align*}
This together with \eqref{eq:Caypt4 3} implies \eqref{eq:Caypt4}. 
\end{proof}

\subsection{The associator case}

We can show the following by Proposition \ref{prop:Cayley eq pt}.  
Set $V =\R^7$ and use the notation of Section \ref{sec:G2 geometry}. 

\begin{corollary} \label{cor:asso eq pt}
For any $F \in \Lambda^2 V^*$, we have 
\begin{align*}
\left( 1 - \frac{1}{2} \la F^2, * \varphi \ra \right)^2
+
\left| * \varphi \wedge F - \frac{1}{6} F^3 \right|^2
+
\frac{1}{4} |\varphi \wedge * F^2|^2 
=
\det \left( I_7 + F^\sharp \right), 
\end{align*}
where $I_7$ is the identity matrix of dimension $7$. 
\end{corollary}

\begin{proof}
By $V^* \ni e^\mu \mapsto e^\mu \in W^*$ for $\mu=1, \cdots, 7$, 
we identify $F \in \Lambda^2 V^*$ with an element of $\Lambda^2 W^*$. 
We rewrite Proposition \ref{prop:Cayley eq pt} to obtain Corollary \ref{cor:asso eq pt}. 
For clarification, denote by $*=*_8$ and $*_7$ 
the Hodge star operators on $W$ and $V$, respectively.

Since $\Phi = e^0 \wedge \varphi + *_7 \varphi$, we have 
$\la F^2, \Phi \ra = \la F^2, *_7 \varphi \ra$. 
Since $F \in \Lambda^2 V^*$, we have 
$F^4=0$ and $I_8 + F^\sharp = 1 \oplus (I_7 + F^\sharp)$. 
Then, 
\begin{align} \label{eq:asso eq pt 1}
\begin{split}
\left( 1- \frac{1}{2} \la F^2, \Phi \ra + \frac{* F^4}{24}  \right)^2
&= 
\left( 1- \frac{1}{2} \la F^2, *_7 \varphi \ra \right)^2, \\
\det (I_8 + F^\sharp) 
&= \det (I_7 + F^\sharp). 
\end{split}
\end{align}
Next, we rewrite 
$4 \left| \pi^2_7 \left( F - *_8 (F^3/6) \right) \right|^2$. 
By Lemma \ref{lem:lambdas}, we have 
\begin{align*}
\pi^2_7 (F) = \sum_{\mu=1}^7 \la F, \lambda^2 (e^\mu) \ra \cdot \lambda^2 (e^\mu). 
\end{align*}
Since 
$$
2 \la F, \lambda^2 (e^\mu) \ra 
= \la F, i(e_\mu) \varphi \ra
= \la e^\mu \wedge F, \varphi \ra
= \la e^\mu, *_7 (F \wedge *_7 \varphi) \ra, 
$$
it follows that 
$$
2 \pi^2_7 (F) = \lambda^2 (*_7 (F \wedge *_7 \varphi)). 
$$
Similarly, since $*_8 F^3 = e^0 \wedge *_7 F^3$, we have 
$$
2 \pi^2_7 (*_8 F^3) 
= 2 \sum_{\mu=1}^7 \la *_8 F^3, \lambda^2 (e^\mu) \ra \cdot \lambda^2 (e^\mu)
= \sum_{\mu=1}^7 \la *_7 F^3, e^\mu \ra \cdot \lambda^2 (e^\mu)
= \lambda^2 (*_7 F^3). 
$$
Hence, we obtain 
\begin{align} \label{eq:asso eq pt 2}
4 \left| \pi^2_7 \left( F - \frac{1}{6} *_8 F^3 \right) \right|^2 
= 
\left| F \wedge *_7 \varphi - \frac{1}{6} F^3 \right|^2. 
\end{align}

Finally, we rewrite $2 \left| \pi^4_7 \left( F^2 \right) \right|^2$. 
By Lemma \ref{lem:lambdas}, we have 
\begin{align*}
\pi^4_7 (F^2) = \sum_{\mu=1}^7 \la F^2, \lambda^4 (e^\mu) \ra \cdot \lambda^4 (e^\mu). 
\end{align*}
Since 
$$
\sqrt{8} \la F^2, \lambda^4 (e^\mu) \ra
=
- \la F^2, e^\mu \wedge \varphi \ra
=
- \la e^\mu, *_7 (\varphi \wedge *_7 F^2) \ra, 
$$
it follows that 
$$
\sqrt{8} \pi^4_7 (F^2) = - \lambda^4 (*_7 (\varphi \wedge *_7 F^2)). 
$$
Hence, we obtain 
\begin{align} \label{eq:asso eq pt 3}
2 \left| \pi^4_7 \left( F^2 \right) \right|^2
= 
\frac{1}{4} \left| \varphi \wedge *_7 F^2 \right|^2. 
\end{align}
By \eqref{eq:asso eq pt 1}, \eqref{eq:asso eq pt 2}, \eqref{eq:asso eq pt 3} and 
Proposition \ref{prop:Cayley eq pt}, we obtain 
Corollary \ref{cor:asso eq pt}. 
\end{proof}

Consider $*_7 \varphi$ as a 4-form of $W=\R^8=\R e_0 \oplus V$ by pullback. 
Then, we see the following.

\begin{lemma} \label{lem:Cay asso pt}
For any $F \in \Lambda^2 W^*$, we have 
\begin{align*}
\left| 1- \frac{1}{2} \la F^2, *_7 \varphi \ra \right|
\leq \sqrt{\det (I_8 + F^\sharp)}. 
\end{align*}
The equality holds if and only if 
$$
i(e_0)F=0, \quad 
F \wedge *_7 \varphi - \frac{1}{6} F^3 =0. 
$$
\end{lemma}

\begin{proof}
Set 
$F=e^0 \wedge F_1 + F_2$, 
where $F_1 \in V^*$ and $F_2 \in \Lambda^2 V^*$.  
Then, by the fact that $F^2= 2 e^0 \wedge F_1 \wedge F_2 + F_2^2$ 
and Corollary \ref{cor:asso eq pt}, 
we have 
\begin{align*}
\left| 1- \frac{1}{2} \la F^2, *_7 \varphi \ra \right|
=
\left| 1- \frac{1}{2} \la F_2^2, *_7 \varphi \ra \right|
\leq 
\sqrt{\det (I_7 + F_2^\sharp)}.  
\end{align*}
The equality holds if and only if 
$F_2 \wedge *_7 \varphi - F_2^3/6 =0$. 
Here we use the fact that  
$F_2 \wedge *_7 \varphi - F_2^3/6 =0$ implies that 
$\varphi \wedge * F_2^2=0$ by \cite[Corollary C.3]{KY}. 

By the proof of Lemma \ref{lem:det}, we have 
\begin{align*}
\det (I_7 + F_2^\sharp) 
&= 
1+ |F_2|^2 + \left| \frac{F_2^2}{2!} \right|^2
+ \left| \frac{F_2^3}{3!} \right|^2, \\
\det (I_8 + F^\sharp) 
&= 
1+ |F|^2 + \left| \frac{F^2}{2!} \right|^2
+ \left| \frac{F^3}{3!} \right|^2
+ \left| \frac{F^4}{4!} \right|^2. 
\end{align*}
Since $F=e^0 \wedge F_1 + F_2, 
F^2 = 2 e^0 \wedge F_1 \wedge F_2 + F_2^2$ and 
$F^3 = 3 e^0 \wedge F_1 \wedge F_2^2 + F_2^3$, 
we have 
$|F|^2=|F_1|^2+|F_2|^2$, 
$|F^2|^2 = 4 |F_1 \wedge F_2|^2 + |F_2^2|^2$ and 
$F^3 = 9 |F_1 \wedge F_2^2|^2 + |F_2^3|^2$. 
Hence, we obtain 
\begin{align*}
\det (I_7 + F_2^\sharp) \leq \det (I_8 + F^\sharp) 
\end{align*}
and the equality holds if and only if $F_1=0$. 
Then, the proof is completed. 
\end{proof}

\subsection{The 3-dimensional special Lagrangian case}
We can show the following by Corollary \ref{cor:asso eq pt}. 
Set $U=\R^6$ and use the notation of Section \ref{sec:SU3 geometry}.

\begin{corollary} \label{cor:SL3 eq pt}
For any $F \in \Lambda^2 U^*$, we have 
\begin{align*}
\left| \frac{(\om+ \i F)^3}{3!} \right|^2
+ 2 \left| \pi^{\lb 3,1 \rb} \left( \frac{(\om+ \i F)^2}{2!} \right) \right|^2 
=
\det (I_6 + F^\sharp), 
\end{align*}
where $I_6$ is the identity matrix of dimension 6. 
In particular, 
$$
\left| {\rm Re} \left( e^{-\i \theta} \frac{(\om+ \i F)^3}{3!} \right) \right| 
\leq \sqrt{\det (I_6 + F^\sharp)}
$$ 
for any $F \in \Lambda^2 U^*$ and $\theta \in \R$. 
The equality holds 
if and only if 
$$
{\rm Im} \left( e^{-\i \theta} \frac{(\om+ \i F)^3}{3!} \right) =0
\quad \mbox{and} \quad \pi^{\lb 2,0 \rb} (F)=0. 
$$
\end{corollary}

Note that the identity depends only on the K\"ahler structure $(\om, g, J)$ and 
is independent of the holomorphic volume form $\Om$ (the Calabi--Yau structure).

\begin{proof}
By $U^* \ni e^i \mapsto e^i \in V^*$ for $i=2, \cdots, 7$, 
we identify $F \in \Lambda^2 U^*$ with an element of $\Lambda^2 V^*$. 
As in the proof of Corollary \ref{cor:asso eq pt}, 
we rewrite Corollary \ref{cor:asso eq pt} to obtain Corollary \ref{cor:SL3 eq pt}. 
For clarification, denote by $*=*_7$ and $*_6$ the Hodge star operators on $V$ and $U$, respectively.

Since $\varphi = e^1 \wedge \om + {\rm Re} \Om$ and $*_7 \varphi = \om^2/2 -e^1 \wedge \Im \Om$, 
we have 
$$
\frac{1}{2} \la F^2, *_7 \varphi \ra
=
\frac{1}{4} \la F^2, \om^2 \ra
= \frac{1}{2} *_6 (F^2 \wedge \om), 
$$
where we use $*_6 \om^2=2 \om$. Thus, 
$$
1- \frac{1}{2} \la F^2, *_7 \varphi \ra = *_6 \left( \frac{\om^3 -3 \om \wedge F^2}{6} \right) 
= *_6 {\rm Re} \left( \frac{(\om+ \i F)^3}{6} \right). 
$$
We also compute 
\begin{align*}
\left| *_7 \varphi \wedge F - \frac{1}{6} F^3 \right|^2
=& \left |\Im \Om \wedge F \right |^2 + \left | \frac{1}{2} \om^2 \wedge F - \frac{1}{6} F^3 \right |^2 \\
=& 
\left |\Im \Om \wedge F \right |^2 + \left | \frac{\Im (\om+ \i F)^3}{6} \right |^2
\end{align*}
and 
$$
\frac{1}{4} \left |\varphi \wedge *_7 F^2 \right |^2 = \frac{1}{4} |{\rm Re} \Om \wedge *_6 F^2|^2. 
$$
Hence, Corollary \ref{cor:asso eq pt} implies that 
\begin{align} \label{eq:SL3 eq pt 1}
\left | \frac{(\om+ \i F)^3}{6} \right |^2 + \left |\Im \Om \wedge F \right |^2 + \frac{1}{4} \left |{\rm Re} \Om \wedge *_6 F^2 \right |^2
= 
\det (I_6 + F^\sharp). 
\end{align}
By Corollary \ref{cor:SU3 2form norm}, Lemma \ref{lem:SU3 2form} and \eqref{eq:Hodge proj SU3}, we have 
\begin{align} \label{eq:SL3 eq pt 2}
\begin{split}
\left|\Im \Om \wedge F \right|^2 
&= 2 \left|\pi^{\lb 2,0 \rb}(F) \right|^2 = 2 \left|\pi^{\lb 2,0 \rb}(*_6 (\om \wedge F)) \right|^2 
= 2 \left|\pi^{\lb 3,1 \rb}(\om \wedge F) \right|^2, \\
\left|{\rm Re} \Om \wedge *_6 F^2 \right|^2 
&= 2 \left|\pi^{\lb 2,0 \rb}(*_6 F^2) \right|^2 = 2 \left|\pi^{\lb 3,1 \rb}(F^2) \right|^2. 
\end{split}
\end{align} 
Thus, it follows that 
\begin{align*}
\left|\Im \Om \wedge F \right|^2 + \frac{1}{4} \left|{\rm Re} \Om \wedge *_6 F^2 \right|^2
=&
2 \left|\pi^{\lb 3,1 \rb}(\om \wedge F) \right|^2 + \frac{1}{2} \left|\pi^{\lb 3,1 \rb}(F^2) \right|^2 \\
=&
\frac{1}{2} \left|\pi^{\lb 3,1 \rb} ((\om + \i F)^2) \right|^2. 
\end{align*}
This together with \eqref{eq:SL3 eq pt 1} implies the identity of Corollary \ref{cor:SL3 eq pt}. 
The last statement follows from \eqref{eq:SL3 eq pt 1} and the first equation of \eqref{eq:SL3 eq pt 2}. 
\end{proof}

Consider $\om$ as a 2-form on $V=\R^7=\R e_1 \oplus U$ by pullback. 
Then, we can prove the following in the same way as Lemma \ref{lem:Cay asso pt} 
using Corollary \ref{cor:SL3 eq pt}. 

\begin{lemma} \label{lem:asso SL3 pt}
For any $F \in \Lambda^2 V^*$, we have 
\begin{align*}
\left| 1- \frac{1}{4} \la F^2, \om^2 \ra \right|
\leq \sqrt{\det (I_7 + F^\sharp)}. 
\end{align*}
The equality holds if and only if 
$$
i(e_1)F=0, \quad 
{\rm Im} \left( \frac{(\om+ \i F)^3}{3!} \right) =0 
\quad \mbox{and} \quad 
\pi^{\lb 2,0 \rb} (F)=0. 
$$
\end{lemma}

\subsection{The 4-dimensional special Lagrangian case}
We can show the following by Proposition \ref{prop:Cayley eq pt}.  
Set $W=\R^8$ and use the notation of Section \ref{sec:SU4 geometry}. 

\begin{proposition}\label{prop:SL4 eq pt}
For any $F \in \Lambda^2 W^*$, we have 
\begin{align*}
&\left| \frac{(\om+ \i F)^4}{4!} \right|^2
+ 2 \left| \pi^{\lb 4,2 \rb} \left( \frac{(\om+ \i F)^3}{3!} \right) \right|^2 
+ 8 \left| \pi^{\lb 4,0 \rb} \left( \frac{(\om+ \i F)^2}{2!} \right) \right|^2 \\
=&
\det (I_8 + F^\sharp). 
\end{align*}
In particular, 
$$
\left| {\rm Re} \left( e^{-\i \theta} \frac{(\om+ \i F)^4}{4!} \right) \right| 
\leq \sqrt{\det (I_8 + F^\sharp)}
$$ 
for any $F \in \Lambda^2 W^*$ and $\theta \in \R$. 
The equality holds 
if and only if 
$$
{\rm Im} \left( e^{-\i \theta} \frac{(\om+ \i F)^4}{4!} \right) =0
\quad \mbox{and} \quad \pi^{\lb 2,0 \rb} (F)=0. 
$$
\end{proposition}

As in Corollary \ref{cor:SL3 eq pt}, 
note that the identity depends only on the K\"ahler structure $(\om, g, J)$ and 
is independent of the holomorphic volume form $\Om$ (the Calabi--Yau structure).

We prove Proposition \ref{prop:SL4 eq pt} by rewriting Proposition \ref{prop:Cayley eq pt} (Lemma \ref{lem:Caypt deg}). 
We first prove the following. 

\begin{lemma} \label{lem:SL4pt deg}
For any $F \in \Lambda^2 W^*$, we have 
\begin{align}
\label{eq:SL4pt2}
- \frac{1}{2} * (\om^2 \wedge F^2) + \frac{1}{36} (* (\om^3 \wedge F))^2 
+ 2 |\pi^{\lb 2,0 \rb} (F)|^2 
&= |F|^2,  \\
\label{eq:SL4pt4}
\begin{split}
\frac{1}{16} (* (\om^2 \wedge F^2))^2 + \frac{* F^4}{12} 
-\frac{1}{18} *(\om^3 \wedge F) * (\om \wedge F^3) & \\
+ 
\frac{1}{2} \left| \pi^{\lb 4,2 \rb} \left( \om \wedge F^2 \right) \right|^2 
- \frac{2}{3}
\left \la \pi^{\lb 2,0 \rb} (F),  \pi^{\lb 2,0 \rb} \left( * F^3 \right) \right \ra
+
2 \left| \pi^{\lb 4,0 \rb} \left( F^2 \right) \right|^2 
&=
\frac{1}{4} \left| F^2 \right|^2, 
\end{split}
\\
\label{eq:SL4pt6}
- \frac{1}{48} *(\om^2 \wedge F^2) * F^4 + \frac{1}{36} (* (\om \wedge F^3))^2
+\frac{1}{18} \left| \pi^{\lb 2,0 \rb} \left( * F^3 \right) \right|^2
&= \frac{1}{36} \left| F^3 \right|^2. 
\end{align}
\end{lemma}

\begin{proof}
We rewrite Lemma \ref{lem:Caypt deg}. 
Since 
$\Phi=\om^2/2+{\rm Re}\Om$ and 
$\pi^2_7(F)= \pi_{\R \om} (F) + \pi_{A_+}(F) = \la F, \om \ra \om/4 +  \pi_{A_+}(F)$ by \eqref{eq:SU4 id7}, 
\eqref{eq:Caypt2} is described as 
$$
- \frac{1}{2} * (\om^2 \wedge F^2) + \la F, \om \ra^2 + \left \{ - \la F^2, {\rm Re} \Om \ra + 4 |\pi_{A_+}(F)|^2 \right \} = |F|^2. 
$$ 
By \eqref{eq:SU4 wk} and Lemma \ref{lem:Apm}, we have 
$$
- \la F^2, {\rm Re} \Om \ra + 4 |\pi_{A_+}(F)|^2
= 2 |\pi_{A_+}(F)|^2 + 2 |\pi_{A_-}(F)|^2
= 2 |\pi^{\lb 2,0 \rb} (F)|^2. 
$$
This together with $* \om = \om^3/6$ implies \eqref{eq:SL4pt2}. 
\\

Next, we show \eqref{eq:SL4pt6}. 
By \eqref{eq:Caypt6}, we have 
$$
- \frac{1}{48} \la F^2, \om^2 \ra * F^4 + \frac{1}{36} \la * F^3, \om \ra^2 + \frac{1}{9} \left| \pi_{A_+} \left( * F^3 \right) \right|^2
- \frac{1}{24} \la F^2, {\rm Re} \Om \ra * F^4
= \frac{1}{36} \left| F^3 \right|^2. 
$$
By Lemmas \ref{lem:24form} and \ref{lem:Apm}, it follows that  
$$
\la F^2, {\rm Re} \Om \ra* F^4
= \frac{2}{3} *\left( (*F)^3 \wedge {\rm Re} \Om \right) 
= \frac{4}{3} \left( \left| \pi_{A_+} \left( * F^3 \right) \right|^2 - \left| \pi_{A_-} \left( * F^3 \right) \right|^2 \right),   
$$
which implies that 
\begin{align*}
\frac{1}{9} \left| \pi_{A_+} \left( * F^3 \right) \right|^2 - \frac{1}{24} \la F^2, {\rm Re} \Om \ra * F^4
=&
\frac{1}{18} \left( \left| \pi_{A_+} \left( * F^3 \right) \right|^2 + \left| \pi_{A_-} \left( * F^3 \right) \right|^2 \right) \\
=&
\frac{1}{18} \left| \pi^{\lb 2,0 \rb} \left( * F^3 \right) \right|^2.  
\end{align*}
Hence, we obtain \eqref{eq:SL4pt6}. 
\\

Next, we prove \eqref{eq:SL4pt4}. 
We rewrite \eqref{eq:Caypt4}. 
Since $\Phi=\om^2/2+{\rm Re}\Om$, we have 
\begin{align}\label{eq:SL4pt4 1} 
\frac{1}{4} \la F^2, \Phi \ra^2 
= \frac{1}{16} \la F^2, \om^2 \ra^2 + \frac{1}{4} \la F^2, \om^2 \ra \la F^2, {\rm Re} \Om \ra + \frac{1}{4} \la F^2, {\rm Re} \Om \ra^2. 
\end{align}
Since $\pi^2_7(F)= \pi_{\R \om} (F) + \pi_{A_+}(F) = \la F, \om \ra \om/4 +  \pi_{A_+}(F)$, we have 
\begin{align}\label{eq:SL4pt4 2} 
\begin{split}
- \frac{4}{3}
\left \la \pi^2_7 (F),  \pi^2_7 \left( * F^3 \right) \right \ra
&= - \frac{4}{3} \left( \frac{1}{4} \la F, \om \ra \la *F^3, \om \ra  + \la \pi_{A_+}(F), \pi_{A_+}(*F^3) \ra \right) \\
&=
- \frac{1}{18} *(\om^3 \wedge F) *(\om \wedge F^3) - \frac{4}{3} \la \pi_{A_+}(F), \pi_{A_+}(*F^3) \ra. 
\end{split}
\end{align}
By \eqref{eq:SU4 id7}, \eqref{eq:SU4 40}, $|{\rm Re} \Om|^2=8$ and Lemma \ref{lem:SU4 4form}, we have 
\begin{align} \label{eq:SL4pt4 3} 
\begin{split}
2 \left| \pi^4_7 \left( F^2 \right) \right|^2 
=&
2 \left| \pi_{\R \Im \Om} \left( F^2 \right) \right|^2 + 2 \left| \pi_{\om \wedge A_-} \left( F^2 \right) \right|^2 \\
=&
2 \left| \pi^{\lb 4,0 \rb} \left( F^2 \right) \right|^2 -2 \left| \pi_{\R {\rm Re} \Om} \left( F^2 \right) \right|^2
+ \left| \pi_{A_-} \left( *(\om \wedge F^2) \right) \right|^2 \\
=& 
2 \left| \pi^{\lb 4,0 \rb} \left( F^2 \right) \right|^2 -\frac{1}{4} \la F^2, {\rm Re} \Om \ra^2 
+ \left| \pi_{A_-} \left( *(\om \wedge F^2) \right) \right|^2. 
\end{split}
\end{align}
Thus, by \eqref{eq:SL4pt4 1}, \eqref{eq:SL4pt4 2} and \eqref{eq:SL4pt4 3}, 
\eqref{eq:Caypt4} is equivalent to 
\begin{align} \label{eq:SL4pt4 4} 
\begin{split}
\frac{1}{16} \la F^2, \om^2 \ra^2 + \frac{* F^4}{12} 
- \frac{1}{18} *(\om^3 \wedge F) *(\om \wedge F^3) 
+ 2 \left| \pi^{\lb 4,0 \rb} \left( F^2 \right) \right|^2& \\
+ \frac{1}{4} \la F^2, \om^2 \ra \la F^2, {\rm Re} \Om \ra 
- \frac{4}{3} \la \pi_{A_+}(F), \pi_{A_+}(*F^3) \ra
+ \left| \pi_{A_-} \left( *(\om \wedge F^2) \right) \right|^2
&=
\frac{1}{4} \left| F^2 \right|^2. 
\end{split}
\end{align}
Now, we prove the following lemma. 
\begin{lemma} \label{lem:SL4pt4 rw}
We have 
\begin{align*}
 \frac{1}{4} \la F^2, \om^2 \ra \la F^2, {\rm Re} \Om \ra 
=&
\frac{2}{3} \la \pi_{A_+}(F), \pi_{A_+}(*F^3) \ra - \frac{2}{3} \la \pi_{A_-}(F), \pi_{A_-}(*F^3) \ra \\
&+ \frac{1}{2} \left| \pi_{A_+} \left( *(\om \wedge F^2) \right) \right|^2
- \frac{1}{2} \left| \pi_{A_-} \left( *(\om \wedge F^2) \right) \right|^2. 
\end{align*}
\end{lemma}
\begin{proof}
Set $F(t)=F+t \om$ for $t \in \R$. 
By Lemma \ref{lem:24form}, we have 
\begin{align} \label{eq:SL4pt4 rw 1} 
*\left( {\rm Re} \Om \wedge (* F(t)^3)^2 \right)
= \frac{3}{2} \la F(t)^2, {\rm Re} \Om \ra* F(t)^4. 
\end{align}
We compute 
\begin{align} \label{eq:SL4pt4 rw 2}
\begin{split}
&\left. \frac{d^2}{dt^2} *\left( {\rm Re} \Om \wedge (* F(t)^3)^2 \right) \right |_{t=0} \\
=&
6 \left. \frac{d}{dt} *\left( {\rm Re} \Om \wedge (* F(t)^3) \wedge * \left( \om \wedge F(t)^2 \right) \right) \right |_{t=0} \\
=&
6 *\left( {\rm Re} \Om \wedge 
\left( 3 *(\om \wedge F^2) \wedge *(\om \wedge F^2) + 2 (*F^3) \wedge *(\om^2 \wedge F) \right) \right)
\end{split}
\end{align}
and 
\begin{align} \label{eq:SL4pt4 rw 3}
\begin{split}
&\frac{3}{2} \left. \frac{d^2}{dt^2} \la F(t)^2, {\rm Re} \Om \ra* F(t)^4 \right |_{t=0} \\
=&
\frac{3}{2} \left. \frac{d}{dt} \left( \la 2 \om \wedge F(t), {\rm Re} \Om \ra* F(t)^4
+ 4 \la F(t)^2, {\rm Re} \Om \ra* (\om \wedge F(t)^3)
 \right) \right |_{t=0} \\
=&
6 \left. \frac{d}{dt}  \la F(t)^2, {\rm Re} \Om \ra* (\om \wedge F(t)^3) \right |_{t=0} \\
=&
18 \la F^2, {\rm Re} \Om \ra \la F^2, \om^2 \ra, 
\end{split}
\end{align}
where we use $* {\rm Re} \Om = {\rm Re} \Om, *\om^2=\om^2$ and $\om \wedge {\rm Re} \Om =0$. 
Hence, \eqref{eq:SL4pt4 rw 1}, \eqref{eq:SL4pt4 rw 2} and \eqref{eq:SL4pt4 rw 3} imply that 
\begin{align} \label{eq:SL4pt4 rw 4}
3 *\left( {\rm Re} \Om \wedge (*(\om \wedge F^2))^2 \right) 
+ 2 *\left( {\rm Re} \Om \wedge *(\om^2 \wedge F) \wedge (*F^3) \right)
=
3 \la F^2, {\rm Re} \Om \ra \la F^2, \om^2 \ra. 
\end{align}
By Lemmas \ref{lem:SU4 2form} and \ref{lem:Apm}, we have 
\begin{align} \label{eq:SL4pt4 rw 5}
\begin{split}
3 *\left( {\rm Re} \Om \wedge (*(\om \wedge F^2))^2 \right) 
&= 6 \left( \left| \pi_{A_+} \left( *(\om \wedge F^2) \right) \right|^2 - \left| \pi_{A_-} \left( *(\om \wedge F^2) \right) \right|^2 \right), \\
2 *\left( {\rm Re} \Om \wedge *(\om^2 \wedge F) \wedge (*F^3) \right)
&= 
4 *\left( {\rm Re} \Om \wedge F \wedge (*F^3) \right) \\
&= 8 \left( \la \pi_{A_+}(F), \pi_{A_+}(*F^3) \ra - \la \pi_{A_-}(F), \pi_{A_-}(*F^3) \ra \right). 
\end{split}
\end{align}
Then, by \eqref{eq:SL4pt4 rw 4} and \eqref{eq:SL4pt4 rw 5}, the proof is completed. 
\end{proof}
Then, by Lemma \ref{lem:SL4pt4 rw}, we obtain 
\begin{align*}
&\frac{1}{4} \la F^2, \om^2 \ra \la F^2, {\rm Re} \Om \ra 
- \frac{4}{3} \la \pi_{A_+}(F), \pi_{A_+}(*F^3) \ra
+ \left| \pi_{A_-} \left( *(\om \wedge F^2) \right) \right|^2 \\
=&
- \frac{2}{3} \left \{ \la \pi_{A_+}(F), \pi_{A_+}(*F^3) \ra + \la \pi_{A_-}(F), \pi_{A_-}(*F^3) \ra \right \} \\
&+ \frac{1}{2} \left \{ \left| \pi_{A_+} \left( *(\om \wedge F^2) \right) \right|^2 + \left| \pi_{A_-} \left( *(\om \wedge F^2) \right) \right|^2 \right \} \\
=&
- \frac{2}{3} \la \pi^{\lb 2,0 \rb} (F), \pi^{\lb 2,0 \rb} (*F^3) \ra 
+ \frac{1}{2} \left| \pi^{\lb 2,0 \rb} \left( *(\om \wedge F^2) \right) \right|^2 \\
=&
- \frac{2}{3} \la \pi^{\lb 2,0 \rb} (F), \pi^{\lb 2,0 \rb} (*F^3) \ra 
+ \frac{1}{2} \left| \pi^{\lb 4,2 \rb} \left( \om \wedge F^2 \right) \right|^2.  
\end{align*}
Then, this together with \eqref{eq:SL4pt4 4} implies \eqref{eq:SL4pt4}. 
\end{proof}

\begin{proof}[Proof of Proposition \ref{prop:SL4 eq pt}] 
Since 
\begin{align*}
* \frac{(\om+ \i F)^4}{4!} 
=& * \left( \frac{\om^4 +4 \i \om^3 \wedge F -6 \om^2 \wedge F^2 -4 \i \om \wedge F^3 + F^4}{4!} \right )\\
=& \left( 1-\frac{1}{4} * (\om^2 \wedge F^2) + \frac{* F^4}{24} \right)
+ \i * \left( \frac{\om^3 \wedge F - \om \wedge F^3}{6} \right ), 
\end{align*}
we have 
\begin{align} \label{eq:SL4 eq pt 1}
\begin{split} 
\left| \frac{(\om+ \i F)^4}{4!} \right|^2
=&
\left( 1-\frac{1}{4} * (\om^2 \wedge F^2) + \frac{* F^4}{24} \right)^2 
+ \left( * \left ( \frac{\om^3 \wedge F - \om \wedge F^3}{6} \right ) \right)^2 \\
=&
1+ \left( \frac{* F^4}{24} \right)^2 
+ \left \{ -\frac{1}{2} * (\om^2 \wedge F^2) + \frac{1}{36} (*(\om^3 \wedge F))^2 \right \} \\
&+ \left \{ \frac{1}{16} (* (\om^2 \wedge F^2))^2 + \frac{* F^4}{12} -\frac{1}{18} *(\om^3 \wedge F) *(\om \wedge F^3) \right \} \\
&+ \left \{ -\frac{1}{48} *(\om^2 \wedge F^2) *F^4 + \frac{1}{36} (*(\om \wedge F^3))^2 \right \}. 
\end{split}
\end{align}
Since 
\begin{align*}
\pi^{\lb 4,2 \rb} \left( \frac{(\om+ \i F)^3}{3!} \right)
=& 
\pi^{\lb 4,2 \rb} \left( \frac{ 3\i \om^2 \wedge F -3 \om \wedge F^2 -\i F^3}{6} \right) \\
=&
- \frac{1}{2} \pi^{\lb 4,2 \rb} (\om \wedge F^2) + \i \pi^{\lb 4,2 \rb} \left( \frac{\om^2 \wedge F}{2} - \frac{F^3}{6} \right)
\end{align*}
and 
$$
\pi^{\lb 4,2 \rb} \left( \frac{\om^2 \wedge F}{2} \right) = \frac{\om^2 \wedge \pi^{\lb 2,0 \rb}(F)}{2} = * \pi^{\lb 2,0 \rb}(F)
$$
by Lemma \ref{lem:SU4 2form}, we have
\begin{align} \label{eq:SL4 eq pt 2}
\begin{split} 
&2 \left| \pi^{\lb 4,2 \rb} \left( \frac{(\om+ \i F)^3}{3!} \right) \right|^2 \\
=&
\frac{1}{2} \left| \pi^{\lb 4,2 \rb} (\om \wedge F^2) \right|^2 
+ 2 \left| \pi^{\lb 2,0 \rb} \left( F - \frac{*F^3}{6} \right) \right|^2 \\
=&
2 \left| \pi^{\lb 2,0 \rb} \left( F \right) \right|^2
+ 
\left \{ \frac{1}{2} \left| \pi^{\lb 4,2 \rb} (\om \wedge F^2) \right|^2 - \frac{2}{3} 
\la \pi^{\lb 2,0 \rb} (F), \pi^{\lb 2,0 \rb} (* F^3) \ra
\right \}
+ \frac{1}{18} |\pi^{\lb 2,0 \rb} (* F^3)|^2. 
\end{split}
\end{align}
Since 
$\pi^{\lb 4,0 \rb} \left( (\om+ \i F)^2 \right) = - \pi^{\lb 4,0 \rb} \left( F^2 \right)$, 
we have 
\begin{align}
\begin{split} \label{eq:SL4 eq pt 3}
8 \left| \pi^{\lb 4,0 \rb} \left( \frac{(\om+ \i F)^2}{2!} \right) \right|^2 
=2 \left| \pi^{\lb 4,0 \rb} \left( F^2 \right) \right|^2. 
\end{split}
\end{align}
Then, by \eqref{eq:SL4 eq pt 1}, \eqref{eq:SL4 eq pt 2}, \eqref{eq:SL4 eq pt 3}, 
Lemmas \ref{lem:SL4pt deg} and \ref{lem:det}, 
we obtain the first identity of Proposition \ref{prop:SL4 eq pt}. \\

By the first identity of Proposition \ref{prop:SL4 eq pt}, 
${\rm Im} \left( e^{-\i \theta} (\om+ \i F)^4/4! \right) =0$ for $\theta \in \R$ 
and $\pi^{\lb 2,0 \rb} (F)=0$ imply that 
$
\left| {\rm Re} \left( e^{-\i \theta} (\om+ \i F)^4/4! \right) \right| = \sqrt{\det (I_8 + F^\sharp)}. 
$
We prove the converse.
Suppose that 
$
\left| {\rm Re} \left( e^{-\i \theta} (\om+ \i F)^4/4! \right) \right| = \sqrt{\det (I_8 + F^\sharp)}
$ 
for $\theta \in \R$. 
Then, by the first identity of Proposition \ref{prop:SL4 eq pt}, \eqref{eq:SL4 eq pt 2} and \eqref{eq:SL4 eq pt 3}, we have 
${\rm Im} \left( e^{-\i \theta} (\om+ \i F)^4/4! \right ) =0$ and 
\begin{align} \label{eq:SL4 eq pt 4}
\pi^{\lb 2,0 \rb} \left( F - \frac{*F^3}{6} \right) =0, \qquad 
\pi^{\lb 4,2 \rb} (\om \wedge F^2)=0, \qquad 
\pi^{\lb 4,0 \rb} \left( F^2 \right)=0. 
\end{align}
We show that \eqref{eq:SL4 eq pt 4} implies that $\pi^{\lb 2,0 \rb} (F)=0$. 
Set 
$$
F = F'+F'' \in [\Lambda^{1,1}] \oplus \lb \Lambda^{2,0} \rb \quad \mbox{and} \quad
F''= \beta + \bar{\beta} \in \Lambda^{2,0} \oplus \Lambda^{0,2}. 
$$
Since 
\begin{align*}
F^2 =& (F')^2 + 2F' \wedge F'' + (F'')^2  \\
&\in [\Lambda^{2,2}] \oplus \lb \Lambda^{3,1} \rb 
\oplus \left( \lb \Lambda^{4,0} \rb \oplus [\Lambda^{2,2}] \right), \\
F^3 =& \left( (F')^3 + 3 F' \wedge (F'')^2 \right) + \left( 3 (F')^2 \wedge F'' + (F'')^3 \right) 
\in [\Lambda^{3,3}] \oplus \lb \Lambda^{4,2} \rb, 
\end{align*}
\eqref{eq:SL4 eq pt 4} is equivalent to 
$$
F''-\frac{1}{6} * \left( 3 (F')^2 \wedge F'' + (F'')^3 \right) =0, \qquad
\om \wedge F' \wedge F''=0, 
\qquad 
\pi^{\lb 4,0 \rb} \left( (F'')^2 \right)=0.
$$
Since 
$(F'')^3 = 3 \beta^2 \wedge \bar \beta + 3 \beta \wedge \bar \beta^2 \in \Lambda^{4,2} \oplus \Lambda^{2,4}$, 
\eqref{eq:SL4 eq pt 4} is equivalent to 
$$
\beta - \frac{1}{2} * \left( (F')^2 \wedge \beta + \beta^2 \wedge \bar \beta \right) =0, \qquad
\om \wedge F' \wedge \beta=0, 
\qquad 
\beta^2 =0, 
$$
and hence, 
\begin{align} \label{eq:SL4 eq pt 5}
\beta - \frac{1}{2} * \left( (F')^2 \wedge \beta \right) =0, \qquad
\om \wedge F' \wedge \beta=0, 
\qquad 
\beta^2 =0.
\end{align}
Since all equations in \eqref{eq:SL4 eq pt 5} are ${\rm U}(4)$ equivariant and $\mathfrak{u} (4) \cong [\Lambda^{1,1}]$, 
we may assume that 
\begin{align*}
\beta =& \beta_1 f^{12} + \beta_2 f^{13} + \beta_3 f^{14} + \beta_4 f^{23} + \beta_5 f^{24} + \beta_6 f^{34}, \\
F' 
=& \mu_1 e^{01} + \mu_2 e^{23} + \mu_3 e^{45} + \mu_4 e^{67} \\
=& \frac{\i}{2} \left( \mu_1 f^1 \wedge \overline{f^1} + \mu_2 f^2 \wedge \overline{f^2} 
+ \mu_3 f^3 \wedge \overline{f^3} + \mu_4 f^4 \wedge \overline{f^4} \right)
\end{align*}
for $\beta_j \in \C$ and $\mu_k \in \R$, 
where $e^i$ and $f^j$ are defined at the beginning of Section \ref{sec:SU4 geometry}. 
Since $\om = (\i/2) \sum f^j \wedge \overline {f^j}$, we see that 
\begin{align*}
4 \om \wedge F' 
=& (\mu_1 + \mu_2) f^{12} \wedge \overline{f^{12}}
+  (\mu_1 + \mu_3) f^{13} \wedge \overline{f^{13}}
+  (\mu_1 + \mu_4) f^{14} \wedge \overline{f^{14}} \\
&+  (\mu_2 + \mu_3) f^{23} \wedge \overline{f^{23}}
+  (\mu_2 + \mu_4) f^{24} \wedge \overline{f^{24}}
+  (\mu_3 + \mu_4) f^{34} \wedge \overline{f^{34}}. 
\end{align*}
Hence, $\om \wedge F' \wedge \beta=0$ is equivalent to 
\begin{align} \label{eq:SL4 eq pt 6}
\begin{split}
&\beta_1 (\mu_3 + \mu_4) = \beta_2 (\mu_2 + \mu_4) = \beta_3 (\mu_2 + \mu_3) \\
=&\beta_4 (\mu_1 + \mu_4) = \beta_5 (\mu_1 + \mu_3) = \beta_6 (\mu_1 + \mu_2) =0. 
\end{split}
\end{align}
Since 
\begin{align*}
2 (F')^2
=& \mu_1 \mu_2 f^{12} \wedge \overline{f^{12}} 
+ \mu_1 \mu_3 f^{13} \wedge \overline{f^{13}} + \mu_1 \mu_4 f^{14} \wedge \overline{f^{14}} \\
&+ \mu_2 \mu_3 f^{23} \wedge \overline{f^{23}} 
+ \mu_2 \mu_4 f^{24} \wedge \overline{f^{24}} +\mu_3 \mu_4 f^{34} \wedge \overline{f^{34}}, \\
2 (F')^2 \wedge \beta
=&
\beta_1 \mu_3 \mu_4 f^{1234} \wedge \overline{f^{34}} - \beta_2 \mu_2 \mu_4 f^{1234} \wedge \overline{f^{24}}
+ \beta_3 \mu_2 \mu_3 f^{1234} \wedge \overline{f^{23}} \\
&+ \beta_4 \mu_1 \mu_4 f^{1234} \wedge \overline{f^{14}}
-\beta_5 \mu_1 \mu_3 f^{1234} \wedge \overline{f^{13}} + \beta_6 \mu_1 \mu_2 f^{1234} \wedge \overline{f^{12}}
\end{align*}
and $* (f^{1234} \wedge \overline{f^{34}}) = 4 f^{12}$, etc., 
$\beta - (1/2) * \left( (F')^2 \wedge \beta \right) =0$ is equivalent to 
\begin{align} \label{eq:SL4 eq pt 7}
\begin{split}
&\beta_1 (1- \mu_3 \mu_4) = \beta_2 (1- \mu_2 \mu_4) = \beta_3 (1- \mu_2 \mu_3) \\
=&\beta_4 (1- \mu_1 \mu_4) = \beta_5 (1- \mu_1 \mu_3) = \beta_6 (1- \mu_1 \mu_2) =0. 
\end{split}
\end{align}
Then, \eqref{eq:SL4 eq pt 6} and \eqref{eq:SL4 eq pt 7} imply that $\beta=0$. Indeed, 
Since $\beta_1 (\mu_3 + \mu_4)=0$, we have 
$$
0=\beta_1 (1- \mu_3 \mu_4) = \beta_1 (1+\mu_3^2). 
$$
Since $\mu_3 \in \R$, we obtain $\beta_1=0$. Similarly, we obtain 
$\beta_2= \cdots = \beta_6=0$, and hence, $\beta =0$. 
\end{proof}

\section{Notation} \label{app:notation} 
We summarize the notation used in this paper. 
We use the following for a smooth manifold $X$. 

\vspace{0.5cm}
\begin{center}
\begin{tabular}{|lc|l|}
\hline
Notation                        & & Meaning \\ \hline \hline
 $i(\,\cdot\,)$                &  & The interior product \\
$\Gamma(X, E)$            & &  The space of all smooth sections of  a vector bundle $E \rightarrow X$\\
$\Om^k$                       & & $\Om^k = \Om^k (X) = \Gamma (X, \Lambda^k T^*X)$ \\
$b^k=b^k(X)$                  & & the $k$-th Betti number of $X$ \\
$H^k_{dR}=H^k_{dR}(X)$     & & the $k$-th de Rham cohomology of $X$ \\
$Z^1$                            & &  the space of closed 1-forms \\
\hline
\end{tabular}
\end{center}
\vspace{0.5cm}

When $(X, g)$ is an oriented Riemannian manifold, we use the following. 

\vspace{0.5cm}
\begin{center}
\begin{tabular}{|lc|l|}
\hline
Notation                        & & Meaning \\ \hline \hline
$v^{\flat} \in T^* X$       & & $v^{\flat} = g(v, \,\cdot\,)$ for $v \in TX$ \\ 
$\alpha^{\sharp} \in TX$ & & $\alpha = g(\alpha^{\sharp}, \,\cdot\,)$ for $\alpha \in T^* X$ \\
$\vol=\vol_g$                 & & The volume form induced from $g$ \\
\hline
\end{tabular}
\end{center}
\vspace{0.5cm}

When $(X, J)$ is a complex manifold, we use the following. 

\vspace{0.5cm}
\begin{center}
\begin{tabular}{|lc|l|}
\hline
Notation                       &  & Meaning \\ \hline \hline
$\Lambda^{p,q}$            & & $\Lambda^{p,q} = \Lambda^p (T^{1,0} X)^* \otimes \Lambda^q (T^{0, 1} X)^*$\\
$\Om^{p,q}$                  & & $\Om^{p,q} = \Gamma (X, \Lambda^{p,q})$  \\
$\lb \Lambda^{p,q} \rb$ (for $p \neq q$)  
& & 
$\lb \Lambda^{p,q} \rb 
= \{\,\alpha \in \Lambda^{p, q} \oplus \Lambda^{q, p} \mid \bar \alpha = \alpha \,\} 
\subset \Lambda^{p + q} T^* X$
\\
$[\Lambda^{p, p}]$          
& & $[\Lambda^{p, p}] 
= \{\, \alpha \in \Lambda^{p, p} \mid \bar \alpha = \alpha \,\} \subset \Lambda^{2p} T^* X$\\
$\lb \Om^{p,q} \rb$ (for $p \neq q$)         
                                   &  & $\lb \Om^{p,q} \rb = \Gamma(X, \lb \Lambda^{p,q} \rb)$ \\
$[\Om^{p,p}]$                 & & $[\Om^{p,p}] = \Gamma(X, [\Lambda^{p, p}])$ \\  
$\pi_S$                         & & \hspace{-1ex}\begin{tabular}{l} The orthogonal projection $\pi_S: \Lambda^k T^*X \to S$ \\
                                       or $\Om^k \to \Gamma(X, S)$ for a subbundle $S \subset T^* X$ \end{tabular} \\
$\pi^{\lb p,q \rb}$           &   & \hspace{-1ex}\begin{tabular}{l} The projection 
                                       $\pi^{\lb p,q \rb} = \pi_{\lb \Lambda^{p,q} \rb}: \Lambda^{p+q} T^* X \to \lb \Lambda^{p,q} \rb$ \\ 
                                       or $\Om^{p+q} \rightarrow \lb \Om^{p,q} \rb$ \end{tabular} \\
\hline
\end{tabular}
\end{center}
\vspace{0.5cm}

When $X$ is a manifold with a $G_2$- or $\Sp$-structure, we use the following.

\vspace{0.5cm}
\begin{center}
\begin{tabular}{|lc|l|}
\hline
Notation                      &  & Meaning \\ \hline \hline
$\Lambda^k_\l T^*X$    &  & \hspace{-1ex}\begin{tabular}{l} The subspace of $\Lambda^k T^*X$ corresponding to  \\
                                          an $\l$-dimensional irreducible subrepresentation \end{tabular} \\
$\Om^k_\l$                  &  &  $\Om^k_\l = \Gamma (X, \Lambda^k_\l T^*X)$ \\
$\pi^k_\l$                    &  & The projection $\Lambda^k T^*X \rightarrow \Lambda^k_\l T^*X$ or 
                                           $\Om^k \rightarrow \Om^k_\l$ \\
\hline
\end{tabular}
\end{center}
\vspace{0.5cm}


\end{document}